\documentclass[12pt]{amsart}
\usepackage{amsfonts}
\usepackage{amsmath}
\usepackage{amsthm}
\usepackage{url}
\usepackage[all]{xy}
\usepackage{latexsym}
\usepackage{amssymb}
\usepackage[cp850]{inputenc}
\usepackage{labelfig}
\usepackage{amsfonts}
\usepackage{graphicx}

\newtheorem{sat}{Theorem}[section]		\newtheorem{lem}[sat]{Lemma}
\newtheorem{kor}[sat]{Corollary}			\newtheorem{prop}[sat]{Proposition}
\newtheorem{defi}{Definition}
\newtheorem*{defi*}{Definition}			\newtheorem*{bei*}{Example}
\newtheorem*{sat*}{Theorem}				\newtheorem*{kor*}{Corollary}
\newtheorem*{rmk*}{Remark}				\newtheorem{quest}{Question}

\let\ssection=\section
\renewcommand{\section}{\setcounter{equation}{0}\ssection}

\newtheorem*{namedtheorem}{\theoremname}
\newcommand{\theoremname}{testing}
\newenvironment{named}[1]{\renewcommand{\theoremname}{#1}\begin{namedtheorem}}{\end{namedtheorem}}

\theoremstyle{remark}
\newtheorem*{bem}{Remark}
\newtheorem{bei}{Example}

\newtheorem*{namedtheoremr}{\theoremnamer}
\newcommand{\theoremnamer}{testing}

\newcommand{\BC}{\mathbb C}			
\newcommand{\BR}{\mathbb R}			\newcommand{\BD}{\mathbb D}
\newcommand{\BN}{\mathbb N}			
\newcommand{\BS}{\mathbb S}			\newcommand{\BZ}{\mathbb Z}
\newcommand{\BF}{\mathbb F}			\newcommand{\BT}{\mathbb T}

\newcommand{\CM}{\mathcal M}		\newcommand{\CN}{\mathcal N}
\newcommand{\CO}{\mathcal O}		
		\newcommand{\CR}{\mathcal R}
\newcommand{\CS}{\mathcal S}		\newcommand{\CT}{\mathcal T}

		\newcommand{\CZ}{\mathcal Z}

\newcommand{\actson}{\curvearrowright}
\newcommand{\D}{\partial}

\DeclareMathOperator{\Out}{Out}		
\DeclareMathOperator{\SL}{SL}		
\DeclareMathOperator{\PSL}{PSL}		
\DeclareMathOperator{\Id}{Id}		
\DeclareMathOperator{\vol}{vol}		

\DeclareMathOperator{\Map}{Map}

\DeclareMathOperator{\Ker}{Ker}

\newcommand{\comment}[1]{}

\DeclareMathOperator{\SO}{SO}

\begin{document}

\title[]{Homomorphisms between mapping class groups}
\author{Javier Aramayona \& Juan Souto}
\thanks{The first author has been partially supported by M.E.C. grant MTM2006/14688. 
The second author has been partially supported by NSF grant DMS-0706878, NSF Career award 0952106, and the 
Alfred P. Sloan Foundation}
\begin{abstract}
Suppose that $X$ and $Y$ are surfaces of finite topological type, where $X$ has genus $g\geq 6$ and $Y$ has genus at most $2g-1$; 
in addition, suppose that $Y$ is not closed if it has genus $2g-1$. 

Our main result asserts that every non-trivial homomorphism $\Map(X) \to \Map(Y)$ 
is induced by an {\em embedding}, i.e.\ a combination of forgetting punctures, deleting 
boundary components and subsurface embeddings. In particular, if $X$ has no
 boundary then every non-trivial endomorphism $\Map(X)\to\Map(X)$ is in fact
 an isomorphism.

As an application of our main theorem we obtain that, under the same hypotheses on genus, if $X$ and $Y$ have finite 
analytic type then every non-constant holomorphic map $\CM(X)\to\CM(Y)$ between the corresponding moduli spaces is a 
forgetful map. In particular, there are no such holomorphic maps unless $X$ and $Y$ have the same genus and $Y$ has at 
most as many marked points as $X$.
\end{abstract}
\maketitle

\centerline{\em A nuestras madres, cada uno a la suya.}
\medskip

\section{Introduction}
Throughout this article we will restrict our attention to connected orientable surfaces of finite topological type, 
meaning of finite genus and with finitely many boundary components and/or cusps. We will feel free to think about 
cusps as marked points, punctures or topological ends. If a surface has empty boundary and no cusps, it is said to 
be closed. The mapping class group $\Map(X)$ of a surface $X$ of finite topological type is the group of isotopy classes of 
orientation preserving homeomorphisms fixing pointwise the union of the boundary and the set of 
punctures. We denote 
by $\CT(X)$ and by $\CM(X)=\CT(X)/\Map(X)$ the Teichm\"uller space and moduli space of $X$, respectively.

\subsection{The conjecture}
The triad formed by the mapping class group $\Map(X)$, Teichm\"uller space $\CT(X)$ and 
moduli space $\CM(X)$ is
often compared with the one formed by $\SL_n\BZ$, the symmetric space 
$\SO_n\backslash\SL_n\BR$, and the locally symmetric 
space $\SO_n\backslash\SL_n\BR/\SL_n\BZ$; here $\SL_n\BZ$ stands as the paradigm of an arithmetic lattice in a higher
 rank semisimple algebraic group. This analogy has motivated many, possibly most, advances in the understanding of the 
mapping class group $\Map(X)$. For example, Grossman \cite{Grossman} proved that $\Map(X)$ is residually finite; Birman, 
Lubotzky and McCarthy \cite{BLM} proved that the Tits alternative holds for subgroups of $\Map(X)$; the Thurston 
classification of elements in $\Map(X)$ mimics the classification of elements in an algebraic group \cite{Thurston}; 
Harvey \cite{Harvey} introduced the curve complex in analogy with the rational Tits' building; Harer's \cite{Harer} 
computation of the virtual cohomological dimension of $\Map(X)$ follows the outline of Borel and Serre's argument for 
arithmetic groups \cite{Borel-Serre}, etc... 

On the other hand, the comparison between $\Map(X)$ and $\SL_n\BZ$ has strong limitations; for instance 
the mapping class group has finite index in its abstract commensurator \cite{Ivanov-comm}, does not have property (T) 
\cite{Andersen} and has infinite dimensional second bounded cohomology \cite{Bestvina-Fujiwara}. In addition, it is not
 known if the mapping class group contains finite index subgroups $\Gamma$ with $H^1(\Gamma;\BR)\neq 0$.

With the dictionary between $\Map(X)$ and $\SL_n\BZ$ in mind, it is natural to ask to what extent there is an analog
of {\em Margulis' superrigidity} in the context of mapping class groups. This question, in various guises, has been
addressed by a number of authors in recent times. For instance, Farb-Masur \cite{Farb-Masur} proved that every homomorphism 
from an irreducible lattice in a higher-rank Lie group to a mapping class group has finite image. Notice that, 
on the other hand, mapping class groups admit non-trivial homomorphisms into higher-rank lattices \cite{Looijenga}. With the same
motivation, one may try to understand homomorphisms between mapping class groups; steps in this direction
include the results of \cite{ALS, Bell-Margalit, Harvey-Korkmaz, Ivanov-2, Ivanov-McCarthy, McCarthy}. 
In the light of this discussion we propose the following general conjecture, which states that, except in some low-genus
cases (discussed in Example \ref{example10}),  some version of Margulis' 
superrigidity holds for homomorphisms between mapping class groups:
\begin{named}{\sc Superrigidity conjecture}
Margulis' superrigidity holds for homomorphisms 
$$\phi:\Map(X)\to\Map(Y)$$ 
between mapping class groups as long as $X$ has at least genus three.
\end{named}

The statement of the superrigidity conjecture is kept intentionally vague for a good reason: different formulations 
of Margulis' superrigidity theorem suggest different forms of the conjecture.

For instance, recall that the geometric version of superrigidity asserts that any homomorphism
 $\Gamma\to\Gamma'$ between two lattices in simple algebraic groups of higher rank is induced by a totally geodesic 
immersion $M_\Gamma\to M_{\Gamma'}$ between the locally symmetric spaces associated 
to $\Gamma$ and $\Gamma'$. One possible way of interpreting
the superrigidity conjecture for mapping class groups is to ask whether every homomorphism between mapping class groups of, say, surfaces of finite analytic type, induces a
holomorphic map between the corresponding moduli spaces. 

\begin{bem}
There are examples \cite{ALS} of injective homomorphisms $\Map(X)\to\Map(Y)$ which map some pseudo-Anosovs to multi-twists and hence are not induced 
by any isometric embedding $\CM(X)\to\CM(Y)$ for any reasonable choice of metric on $\CM(X)$ and $\CM(Y)$. This is 
the reason why we prefer not to ask, as done by Farb and Margalit (see Question 2 of \cite{Bell-Margalit}), whether 
homomorphisms between mapping class groups are {\em geometric}. 
\end{bem}

The Lie theoretic version of superrigidity essentially asserts that every homomorphism $\Gamma\to\Gamma'$ between two irreducible lattices in higher 
rank Lie groups either has finite image or extends to a homomorphism between the ambient groups. The mapping class group $\Map(X)$ is a quotient of the group of homeomorphisms (resp. diffeomorphisms) of $X$ but not a subgroup \cite{Morita,Markovic}, and thus there is no
ambient group as such. A natural interpretation
of this flavor of superrigidity would be to ask whether every homomorphism $\Map(X) \to \Map(Y)$ is induced by a homomorphism between the corresponding groups of homeomorphisms (resp. diffeomorphisms). 

Finally, one has the folkloric version of superrigidity: {\em every homomorphism between two irreducible higher rank 
lattices is one of the ``obvious'' ones.} The word ``obvious'' is rather vacuous; to give it a little bit of content we 
adopt Maryam Mirzakhani's version of the conjecture above: {\em every homomorphism between mapping class groups has either finite image or is 
induced by some manipulation of surfaces.} The statement {\em manipulation of surfaces} is again vague, but it 
conveys the desired meaning. 

\subsection{The theorem}
Besides the lack of counterexamples, the evidence supporting the 
superrigidity conjecture is limited to the results in \cite{Bell-Margalit, Harvey-Korkmaz,  Ivanov-2,Ivanov-McCarthy,
 McCarthy}. The goal of this paper is to prove the conjecture, with respect to any of its possible interpretations,
under suitable genus bounds. Before stating our main result we need a definition:

\begin{defi} 
Let $X$ and $Y$ be surfaces of finite topological type, and consider their cusps to be marked points. Denote
by $\vert X\vert $ and $\vert Y\vert$ the compact surfaces obtained from $X$ and $Y$ by forgetting all their
marked points. By an {\em embedding} $$\iota:X\to Y$$
we will understand a continuous injective map $\iota:\vert X\vert\to\vert Y\vert$ with the property 
that whenever $y\in\iota(\vert X\vert)\subset\vert Y\vert$ is a marked point of $Y$ in the image of $\iota$, 
then $\iota^{-1}(y)$ is also a marked point of $X$. 
\end{defi}

Note that forgetting a puncture, deleting a boundary component, and
embedding $X$ as a subsurface of $Y$ are examples of embeddings. Conversely, 
every embedding is a combination of these three building blocks; compare with 
Proposition \ref{prop-embedding} below.

It is easy to see that every embedding $\iota:X\to Y$ induces a 
homomorphism $\Map(X)\to\Map(Y)$. Our main result is that, as long 
as the genus of $Y$ is less than twice that of $X$, the converse is also true:

\begin{sat}\label{dogs-bollocks}
Suppose that $X$ and $Y$ are surfaces of finite topological type, of genus $g\ge 6$ and $g'\le 2g-1$ respectively; 
if $Y$ has genus $2g-1$, suppose also that it is not closed. Then every nontrivial homomorphism 
$$\phi:\Map(X)\to\Map(Y)$$
is induced by an embedding $X\to Y$.
\end{sat}

\begin{bem}
As we will prove below, the conclusion of Theorem \ref{dogs-bollocks} also applies to homomorphisms 
$\phi:\Map(X)\to\Map(Y)$ when both $X$ and $Y$ have the same genus $g\in\{4,5\}$.
\end{bem}

We now give some examples that highlight the necessity for the genus bounds in Theorem \ref{dogs-bollocks}. 

\begin{bei}\label{example10}
Let $X$ be a surface of genus $g \leq 1$; if $g=0$ then assume that $X$ has at least four marked
points or boundary components. The mapping class group $\Map(X)$ surjects onto 
$\PSL_2\BZ\simeq(\BZ/2\BZ)*(\BZ/3\BZ)$. In particular, any two elements $\alpha,\beta\in\Map(Y)$ 
with orders two and three, respectively, determine a homomorphism $\Map(X)\to\Map(Y)$; notice that 
such elements exist if $Y$ is closed, for example. Choosing $\alpha$ and $\beta$ appropriately, 
one can in fact obtain infinitely many conjugacy classes of homomorphisms $\Map(X)\to\Map(Y)$ 
with infinite image and with the property that every element in the image is either pseudo-Anosov or has finite order.
\end{bei}

Example \ref{example10} shows that some lower bound on the genus of $X$ is necessary in the statement of Theorem 
\ref{dogs-bollocks}. Furthermore, since $\Map(X)$ has non-trivial abelianization if $X$ has genus $2$, there exist
homomorphisms from $\Map(X)$ into mapping class groups of arbitrary closed surfaces $Y$ that are not 
induced by embeddings. On the other hand, we expect Theorem \ref{dogs-bollocks} to be true for surfaces of genus $g\in 
\{3,4,5\}$.

\begin{bem}
Recall that the mapping class group of a punctured disk is a finite index subgroup of the appropriate braid group. 
In particular, Example \ref{example10} should be compared with the rigidity results for homomorphisms between 
braid groups, and from braid groups into mapping class groups, due to 
Bell-Margalit \cite{Bell-Margalit} and Castel \cite{Castel}.
\end{bem}

Next, observe that an upper bound on the genus of the target surface is also necessary in the statement
of Theorem \ref{dogs-bollocks} since, for instance, the mapping class group of every closed surface injects into the 
mapping class group of some non-trivial connected cover \cite{ALS}. Moreover, 
the following example shows that the bound in Theorem \ref{dogs-bollocks} is in fact optimal:

\begin{bei}\label{example1}
Suppose that $X$ has non-empty connected boundary and let $Y$ be the double of $X$. Let $X_1,X_2$ be the two copies of 
$X$ inside $Y$, and for $x\in X$ denote by $x_i$ the corresponding point in $X_i$. Given a homeomorphism $f:X\to X$ 
fixing pointwise the boundary and the cusps define
$$\hat f:Y\to Y,\ \ \hat f(x_i)=(f(x))_i\ \forall x_i\in X_i$$
The map $f\to\hat f$ induces a homomorphism
$$\Map(X)\to\Map(Y)$$
which is not induced by any embedding. 
\end{bei}

%

\subsection{Applications}
After having established that in Theorem \ref{dogs-bollocks} a lower bound for the genus of $X$ is necessary and that 
the upper bound for the genus of $Y$ is optimal, we discuss some consequences of our main result. First, we will observe
that, in the absence of boundary, every embedding of a surface into itself is in fact a homeomorphism; in light of
this, Theorem \ref{dogs-bollocks} implies the following:


\begin{sat}\label{no-boundary2}
Let $X$ be a surface of finite topological type, of genus $g \ge 4$ and with empty boundary.
Then any non-trivial endomorphism $\phi:\Map(X)\to\Map(X)$ is induced by a homeomorphism 
$X\to X$; in particular $\phi$ is an isomorphism.
\end{sat}

\begin{bem}
The analogous statement of Theorem \ref{no-boundary2} for injective endomorphisms was known to be 
true by the work of Ivanov and McCarthy \cite{Ivanov-McCarthy,Ivanov,McCarthy}.
\end{bem}

Theorem \ref{no-boundary2}, as well as other related results discussed in Section
\ref{sec:applications}, are essentially specializations of Theorem \ref{dogs-bollocks} 
to particular situations. 

Returning to the superrigidity conjecture, recall that in order to prove superrigidity for (cocompact) lattices one may associate, to every homomorphism between two lattices, a harmonic
map between the associated symmetric spaces, and then use differential geometric arguments to show that this map is a totally geodesic immersion. This is not the approach we follow in this paper, and neither Teichm\"uller 
space nor moduli space will play any role in the proof of Theorem \ref{dogs-bollocks}.
As a matter of fact, reversing the 
logic behind the proof of superrigidity, Theorem \ref{dogs-bollocks} will actually provide information  about maps 
between moduli spaces, as we describe next. 

Suppose that $X$ and $Y$ are Riemann surfaces of finite analytical type. Endow  the associated Teichm\"uller spaces 
$\CT(X)$ and $\CT(Y)$ with the standard complex structure. The latter is invariant under the action of the corresponding mapping class group and hence we can consider the moduli spaces
$$\CM(X)=\CT(X)/\Map(X),\ \ \CM(Y)=\CT(Y)/\Map(Y)$$ 
as complex orbifolds.

Suppose now that $X$ and $Y$ have the same genus and that $Y$ has at most as many marked points as $X$. Choosing an 
identification between the set of marked points of $Y$ and a subset of the set of marked points of $X$, we obtain a 
holomorphic map $$\CM(X)\to\CM(Y)$$ obtained by forgetting all marked points of $X$ which do not correspond to a 
marked point of $Y$. Different identifications give rise to different maps; we will refer to these maps as 
{\em forgetful maps}. In Section \ref{sec:forgetful-proof} we will prove the following result:

\begin{sat}\label{forgetful}
Suppose that $X$ and $Y$ are Riemann surfaces of finite analytic type and assume that $X$ has genus 
$g\ge 6$ and $Y$ genus $g'\le 2g-1$; in the equality case $g'=2g-1$ assume that $Y$ is not closed.
 Then, every non-constant holomorphic map $$f:\CM(X)\to\CM(Y)$$ is a forgetful map.
\end{sat}

As a direct consequence of Theorem \ref{forgetful} we obtain:

\begin{kor}
Suppose that $X$ and $Y$ are Riemann surfaces of finite analytic type and assume that $X$ has genus $g\ge 6$ and $Y$ genus $g'\le 2g-1$; in the equality case $g'=2g-1$ assume that $Y$ is not closed. If there is a non-constant holomorphic map $f:\CM(X)\to\CM(Y)$, then $X$ and $Y$ have the same genus and $X$ has at least as many marked points as $Y$.\qed
\end{kor}

In order to prove Theorem \ref{forgetful} we will deduce from Theorem
 \ref{dogs-bollocks} that the map $f$ is homotopic to a forgetful map $F$. 
The following result, proved in Section \ref{sec:forgetful-proof}, will immediately yield
the equality between $f$ and $F$:

\begin{prop}\label{Eells-Sampson}
Let $X$ and $Y$ be Riemann surfaces of finite analytical type and let $f_1,f_2:\CM(X)\to\CM(Y)$ be homotopic 
holomorphic maps. If $f_1$ is not constant, then $f_1=f_2$.
\end{prop}

Recall that the Weil-Peterson metric on moduli space is K\"ahler and has negative curvature. In particular, if the
 moduli spaces $\CM(X)$ and $\CM(Y)$ were closed, then Proposition \ref{Eells-Sampson} would follow directly from the 
work of Eells-Sampson  \cite{Eells-Sampson}. In order to prove Proposition \ref{Eells-Sampson} we simply ensure that 
their arguments go through in our context. 
\medskip

\subsection{Strategy of the proof of Theorem \ref{dogs-bollocks}} We now give a brief idea
of the proof of Theorem \ref{dogs-bollocks}. The main technical result of this paper is the following theorem:

\begin{prop}\label{main}
Suppose that $X$ and $Y$ are surfaces of finite topological type of genera $g\ge 6$ and 
$g'\le 2g-1$ respectively; if $Y$ has genus $2g-1$, suppose also that it is not closed. Every nontrivial homomorphism 
$$\phi:\Map(X)\to\Map(Y)$$
maps (right) Dehn twists along non-separating curves to (possibly left) Dehn twist along non-separating curves. 
\end{prop}

Given a non-separating curve $\gamma\subset X$ denote by $\delta_\gamma$ the Dehn twist associated to $\gamma$. 
By Proposition \ref{main}, $\phi(\delta_\gamma)$ is a Dehn twist along some non-separating curve $\phi_*(\gamma)\subset Y$.
 We will observe that the map $\phi_*$ preserves disjointness and intersection number $1$. In particular, $\phi_*$ maps 
chains in $X$ to chains in $Y$. In the closed case, it follows easily that there is a unique embedding $X\to Y$ which
 induces the same map on curves as $\phi_*$; this is the embedding provided by Theorem \ref{dogs-bollocks}. In the 
presence of boundary and/or cusps the argument is rather involved, essentially because one needs to determine which 
cusps and boundary components are to be filled in.

Hoping that the reader is now convinced that Theorem \ref{dogs-bollocks} follows after a moderate amount of work
 from Proposition \ref{main}, we sketch the proof of the latter. The starting point is a result of Bridson 
\cite{Bridson} which asserts that, as long as $X$ has genus at least $3$, any homomorphism
$$\phi:\Map(X)\to\Map(Y)$$ 
maps Dehn twists to roots of multitwists. The first problem that we face when proving Proposition \ref{main} is that
Bridson's result does not rule out that $\phi$ maps Dehn twists to finite order elements. In this direction, one may ask 
the following question:

\begin{quest}\label{question}
Suppose that $\phi:\Map(X)\to\Map(Y)$ is a homomorphism between mapping class groups 
of surfaces of genus at least $3$, with the property that the image of every Dehn twist along a non-separating 
curve has finite order. Is the image of $\phi$ finite?
\end{quest}

The answer to this question is trivially positive if $\D Y\neq\emptyset$,
 for in this case $\Map(Y)$ is torsion-free. We will also give an affirmative answer if $Y$ has punctures:

\begin{sat}\label{blablabla}
Suppose that $X$ and $Y$ are surfaces of finite topological type, that $X$ has genus at least $3$, and that $Y$ is not closed. 
Then any homomorphism $\phi:\Map(X)\to\Map(Y)$ which maps a Dehn twist along a non-separating curve to a finite order element is trivial.
\end{sat}

For closed surfaces $Y$ we only give a partial answer to the above question; more concretely, in 
Proposition \ref{no-torsion} we will prove that, as long as the genus of $Y$ is in a suitable range determined by the
 genus of $X$, then the statement of Theorem \ref{blablabla} remains true.

\begin{bem}
At the end of section \ref{sec:rid-torsion} we will observe that a positive answer to question \ref{question} would imply that the abelianization of finite index subgroups in $\Map(X)$ is finite. This is conjectured to be the case.
\end{bem}

We continue with the sketch of the proof of Proposition \ref{main}. At this point we know that for every non-separating curve
$\gamma\subset X$, the element $\phi(\delta_\gamma)$ is a root of a multitwist and has infinite order. We may thus 
associate to $\gamma$ the multicurve $\phi_*(\gamma)$ supporting the multitwist powers of $\phi(\delta_\gamma)$. 
In principle, and also in practice if $Y$ has sufficiently large genus, $\phi(\delta_\gamma)$ could permute the 
components of $\phi_*(\gamma)$. However, under the genus bounds in Theorem \ref{dogs-bollocks}, we deduce 
from a result of Paris \cite{Paris} that this is not the case. Once we know that $\phi(\delta_\gamma)$ fixes each 
component of $\phi_*(\gamma)$, a simple counting argument yields that $\phi_*(\gamma)$ is actually
a single curve. This implies that $\phi(\delta_\gamma)$ is a root of some power of the Dehn twist along 
$\phi_*(\gamma)$. In the last step, which we now describe, we will obtain that $\phi(\delta_\gamma)$ is in fact a Dehn twist.

Denote by $X_\gamma$ the surface obtained by deleting from $X$ an open regular neighborhood of $\gamma$, and let 
$Y_{\phi_*(\gamma)}'$ be the surface obtained by deleting $\phi_*(\gamma)$ from $Y$. The centralizer of the Dehn 
twist $\delta_\gamma$ in $\Map(X)$ is closely related to $\Map(X_\gamma)$ and the same is true for the centralizer 
of $\phi(\delta_\gamma)$ and $\Map(Y_{\phi_*(\gamma)}')$. More concretely, the homomorphism $\phi$ induces a homomorphism 
$$\Map(X_\gamma)\to\Map(Y_{\phi_*(\gamma)}')$$
Moreover, it follows from the construction that $\phi(\delta_\gamma)$ is in fact a power of the Dehn twist along 
$\phi_*(\gamma)$ if the image of this homomorphism is not centralized by any finite order element in 
$\Map(Y_{\phi_*(\gamma)}')$. This will follow from an easy computation using the 
Riemann-Hurwitz formula together with the next result:

\begin{prop}\label{blabla}
Suppose that $X$ and $Y$ are surfaces of finite topological type. If the genus of $X$ is at least 3 and larger 
than that of $Y$, then there is no nontrivial homomorphism $\phi:\Map(X)\to\Map(Y)$.
\end{prop}

\begin{bem} 
If $X$ is closed, Proposition \ref{blabla} is due to Harvey and Korkmaz \cite{Harvey-Korkmaz}. 
Their argument makes heavy use of torsion in $\Map(X)$; therefore, it cannot be used for
general surfaces.
\end{bem}

Once we know that 
$\phi(\delta_\gamma)$ is a power of a Dehn twist, it follows from the braid relation 
that this power has to be $\pm 1$, as we needed to prove. This finishes the sketch of the proof 
of Proposition \ref{main} and hence of Theorem \ref{dogs-bollocks}.
\medskip

Before concluding the introduction, we would like to mention a related result due to Bridson and Vogtmann 
\cite{Bridson-Vogtmann} on homomorphisms between outer automorphism groups of free groups, namely: 

\begin{sat*}[Bridson-Vogtmann]
Suppose $n>8$. If $n$ is even and $n<m\le 2n$, or if $n$ is odd and $n<m\le 2n-2$, then  every homomorphism 
$\Out(\BF_n)\to\Out(\BF_m)$ factors through a homomorphism $\Out(\BF_n)\to\BZ/2\BZ$.
\end{sat*}

We remark that in their proof, Bridson and Vogtmann make very heavy use of the presence of rather large torsion 
subgroups in $\Out(\BF_n)$. Rather on the contrary, in the present paper torsion is an annoyance. In particular, 
the proof of the Bridson-Vogtmann theorem and that of Theorem \ref{dogs-bollocks} are completely different. In spite of that, 
we would like to mention that the Bridson-Vogtmann theorem played a huge role in this paper: it suggested the 
possibility of understanding homomorphisms between mapping class groups $\Map(X)\to\Map(Y)$ under the assumption that
the genus of $Y$ is not much larger than that of $X$. 
\medskip

\noindent{\bf Acknowledgements.} The authors wish to thank Martin Bridson, Benson Farb, Chris Leininger and especially 
Johanna Mangahas for many very interesting conversations on the topic of this paper.  The authors are also grateful 
to Michel Boileau and Luis Paris for letting them know about the work of Fabrice Castel. The first author
wishes to express his gratitude to Ser Peow Tan and the Institute for Mathematical Sciences of Singapore, where
parts of this work were completed.

\section{Generalities}
\label{sec:general}
In this section we discuss a few well-known facts on mapping class groups. See \cite{Farb-Margalit,Ivanov} for details.

\medskip

Throughout this article, all surfaces under consideration are orientable and have finite topological type, 
meaning that they have finite genus, finitely many boundary components and finitely many punctures. We will feel 
free to consider cusps as marked points, punctures, or ends homeomorphic to $\BS^1\times\BR$. For instance, if $X$ 
is a surface with, say, 10 boundary components and no cusps, by deleting every boundary component we obtain a surface 
$X'$ with 10 cusps and no boundary components.

A simple closed curve on a surface is said to be {\em essential} if it does not bound a disk 
containing at most one puncture; we stress that we consider boundary-parallel curves to be essential. 
From now on, by a {\em curve} we will mean an essential simple closed curve. Also, we will often 
abuse terminology and not distinguish between curves and their isotopy classes. 

We now introduce some notation 
that will be used throughout the paper. Let $X$ be a surface and let $\gamma$
be an essential curve not parallel to the boundary of $X$. We will denote by 
$X_\gamma$ the complement in $X$ of the interior of a closed regular neighborhood of $\gamma$; 
we will refer to the two boundary components of $X_\gamma$ which appear in the boundary of the 
regular neighborhood of $\gamma$ as the {\em new boundary components} of $X_\gamma$. We will 
denote by $X_\gamma'$ the surface obtained from $X_\gamma$ by deleting the new boundary components 
of $X_\gamma$; equivalently, $X_\gamma'=X\setminus\gamma$.


A {\em multicurve} is the union of a, necessarily finite, collection of pairwise disjoint,
 non-parallel curves. 
Given two multicurves $\gamma,\gamma'$ we denote their geometric intersection number by $i(\gamma,\gamma')$.

A {\em cut system} is a multicurve whose complement is a connected surface of genus $0$. Two cut systems
are said to be related by an {\em elementary move} if they share all curves but one, and the remaining
two curves intersect exactly once. The {\em cut system complex} of a surface $X$ is the simplicial graph whose vertices
are cut systems on $X$ and where two cut systems are adjacent if the corresponding cut systems are related by an
 elementary move.

\subsection{Mapping class group}
The mapping class group $\Map(X)$ of a surface $X$ is the group of isotopy classes 
of orientation preserving homeomorphisms $X\to X$ which fix the boundary pointwise and map every cusp to 
itself; here, we also require that the isotopies fix the boundary pointwise. We will also denote by $\Map^*(X)$ 
the group of isotopy classes of all orientation preserving homeomorphisms of $X$. Observe that $\Map(X)$ is 
a subgroup of $\Map^*(X)$ only in the absence of boundary; in this case $\Map(X)$ has finite index in $\Map^*(X)$.

While every element of the mapping class group is an isotopy class of homeomorphisms, it is well-known that the 
mapping class group cannot be realized by a group of diffeomorphisms \cite{Morita}, or even homeomorphisms
 \cite{Markovic}. However, for our purposes the difference between actual homeomorphisms and their isotopy 
classes is of no importance. In this direction, and in order to keep notation under control, we will usually 
make no distinction  between mapping classes and their representatives.

\subsection{Dehn twists}
Given a curve $\gamma$ on $X$, we denote by $\delta_\gamma$ the (right) Dehn twist along $\gamma$. 
It is important to remember that $\delta_\gamma$ is solely determined by the curve $\gamma$ and 
the orientation of $X$. In other words, it is independent of any chosen orientation of $\gamma$. 

Perhaps the main reason why Dehn twists appear so prominently in this paper is because they generate 
the mapping class group:

\begin{sat}[Dehn-Lickorish]\label{lickorish}
If $X$ has genus at least $2$, then $\Map(X)$ is generated by Dehn twists along non-separating curves.
\end{sat}

There are quite a few known concrete sets of Dehn twists which generate the mapping class group. We will 
consider the so-called {\em Humphries generators} \cite{Humphries}; see Figure \ref{fig1} for a 
picture of the involved curves.

\begin{figure}[tbh] \unitlength=1in
\begin{center} 
\includegraphics[width=3.5in]{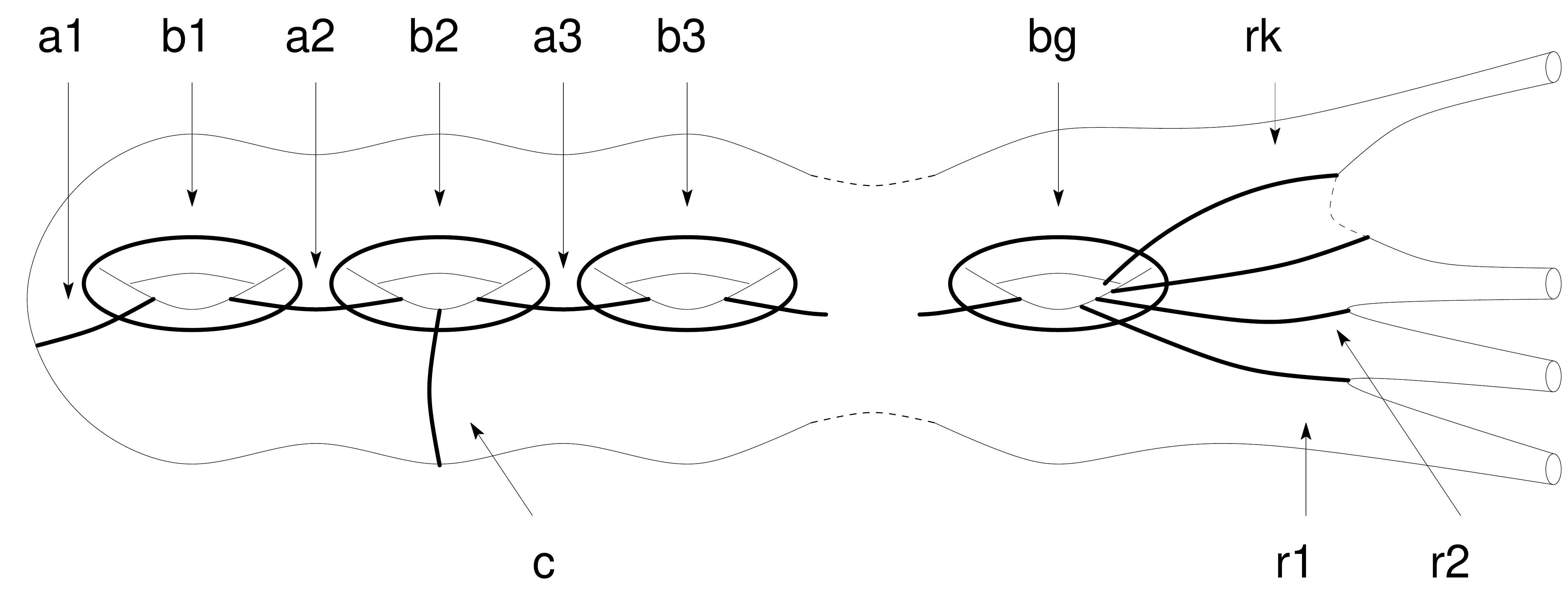}
\end{center}
\caption{The Humphries generators: Dehn twists along the curves $a_i,b_i,c$ and $r_i$ generate $\Map(X)$.}\label{fig1}
\end{figure}

Algebraic relations among Dehn twists are often given by particular configurations of curves.  
We now discuss several of these relations; see \cite{Farb-Margalit,Hamidi-Tehrani,Margalit} and the references therein 
for proofs and details:
\medskip

\noindent{\em Conjugate Dehn twists.} For any curve $\gamma\subset X$ and any $f\in\Map(X)$ we have
$$\delta_{f(\gamma)}=f\delta_\gamma f^{-1}$$
Hence, Dehn twists along any two non-separating curves are conjugate in $\Map(X)$. Conversely, if the Dehn twist along $\gamma$ is conjugate in $\Map(X)$ to a Dehn twist along a non-separating curve, then $\gamma$ is non-separating. 
Observe that Theorem \ref{lickorish} and the fact that Dehn twists along any two non-separating curves are 
conjugate immediately imply the following very useful fact:

\begin{lem}\label{dehn-trivial-trivial}
Let $X$ be a surface of genus at least 3 and let $\phi:\Map(X)\to G$ be a homomorphism.
 If $\delta_\gamma\in\Ker(\phi)$ for some $\gamma\subset X$ non-separating, then $\phi$ is trivial.
\end{lem}

\noindent{\em Disjoint curves.} Suppose $\gamma, \gamma'$ are disjoint curves, meaning $i(\gamma, \gamma')=0$. 
Then $\delta_\gamma$ and $\delta_{\gamma'}$ commute.
\medskip

\noindent{\em Curves intersecting once.} Suppose that two curves $\gamma$ and $\gamma'$ intersect once, 
meaning $i(\gamma,\gamma')=1$. Then
$$\delta_\gamma\delta_{\gamma'}\delta_\gamma=\delta_{\gamma'}\delta_\gamma\delta_{\gamma'}$$
This is the so-called {\em braid relation}; we say that $\delta_\gamma$ and $\delta_{\gamma'}$ {\em braid}.
\medskip

It is known \cite{Hamidi-Tehrani} that if $\gamma$ and $\gamma'$ are two curves in $X$ and $k\in\BZ$ is 
such that $\vert k\cdot i(\gamma,\gamma')\vert\ge 2$, 
then $\delta^k_\gamma$ and $\delta^k_{\gamma'}$ generate a free group $\BF_2$ of rank 2. In particular we have:
 
\begin{lem}\label{inter-1}
Suppose that $k\in\BZ\setminus\{0\}$ and that $\gamma$ and $\gamma'$ are curves such that
 $\delta_\gamma^k$ and $\delta_{\gamma'}^k$ satisfy the braid relation, then either $\gamma=\gamma'$ 
or $k=\pm 1$ and $i(\gamma,\gamma')=1$.
\end{lem}

\noindent{\em Chains.} Recall that a {\em chain} in $X$ is a finite sequence of curves $\gamma_1,\dots,\gamma_k$ such that $i(\gamma_i,\gamma_j)=1$ if $\vert i-j\vert=1$ and $i(\gamma_i,\gamma_j)=0$ otherwise. Let $\gamma_1,\dots,\gamma_k$ be a chain in $X$ and suppose first that $k$ is even. Then the boundary $\D Z$ of a regular neighborhood of $\cup\gamma_i$ is connected and we have
$$(\delta_{\gamma_1}\delta_{\gamma_2}\dots\delta_{\gamma_k})^{2k+2}=\delta_{\D Z}$$
If $k$ is odd then $\D Z$ consists of two components $\D_1Z$ and $\D_2Z$ and the appropriate relation is
$$(\delta_{\gamma_1}\delta_{\gamma_2}\dots\delta_{\gamma_k})^{k+1}=\delta_{\D Z_1}\delta_{\D Z_2}=\delta_{\D Z_2}\delta_{\D Z_1}$$
These two relations are said to be {\em the chain relations}.
\medskip

\noindent{\em Lanterns.} A {\em lantern} is a configuration in of seven curves $a,b,c,d,x,y$ and $z$ in $X$ as represented in figure \ref{fig-lantern}. 

\begin{figure}[tbh] \unitlength=1in
\begin{center} 
\includegraphics[width=2in]{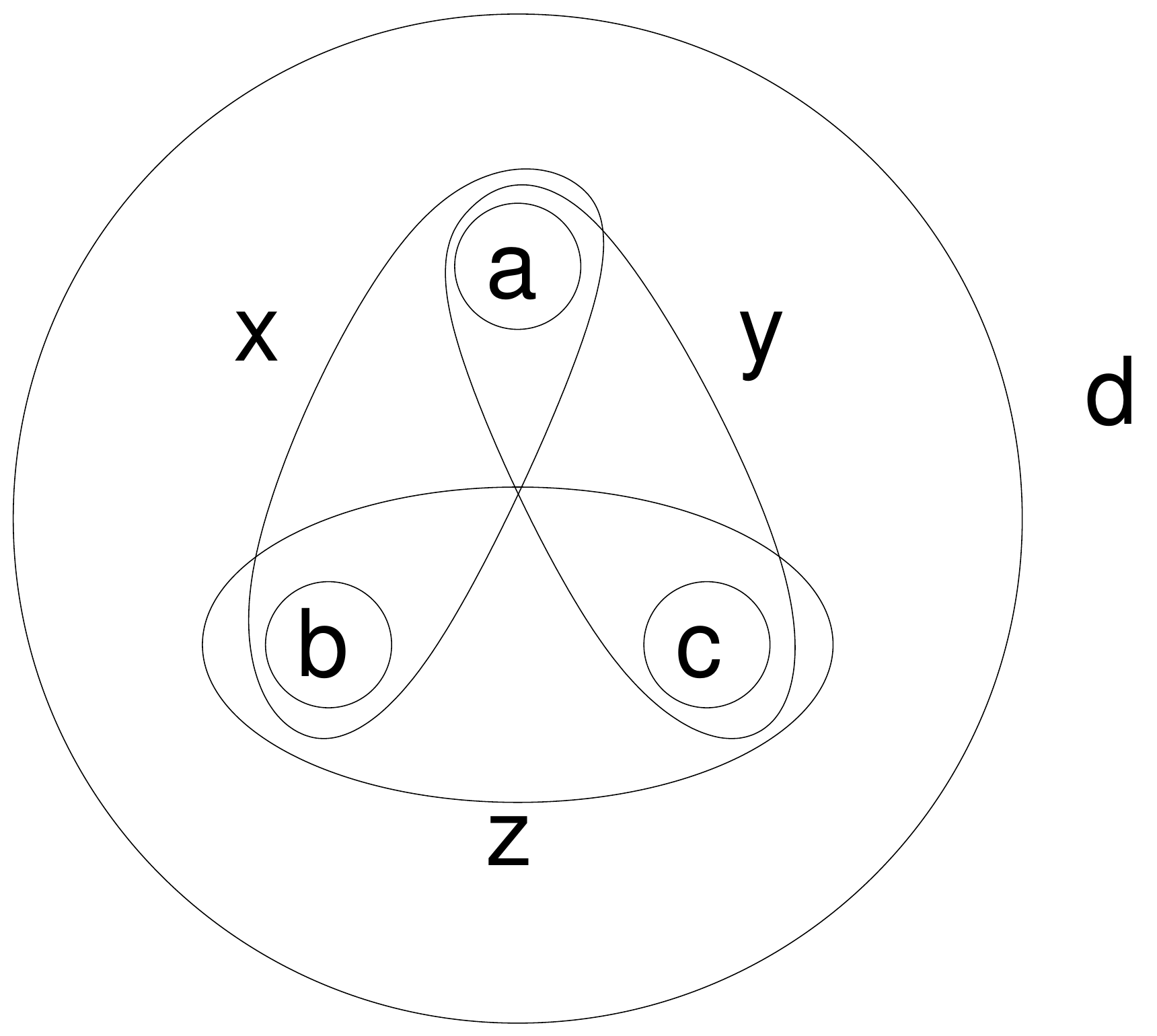}
\end{center}
\caption{A lantern}\label{fig-lantern}
\end{figure}

If seven curves $a,b,c,d,x,y$ and $z$ in $X$ form a lantern then the corresponding Dehn twists satisfy the so-called {\em lantern relation}:
$$\delta_a\delta_b\delta_c\delta_d=\delta_x\delta_y\delta_z$$
Conversely, it is due to Hamidi-Tehrani \cite{Hamidi-Tehrani} and Margalit \cite{Margalit} that, under mild hypotheses,
any seven curves whose associated Dehn twists satisfy the lantern relation form a lantern. More concretely:

\begin{prop}[Hamidi-Tehrani, Margalit]\label{fact-lantern}
Let $a,b,c,d,x,y,z$ be essential curves whose associated Dehn twists satisfy the lantern relation
$$\delta_a\delta_b\delta_c\delta_d=\delta_x\delta_y\delta_z$$
If the curves $a,b,c,d,x$ are paiwise distinct and pairwise disjoint, then $a,b,c,d,x,y,z$ is a lantern. 
\end{prop}

In the course of this paper we will continuously discriminate against separating curves. By a {\em non-separating} 
lantern we understand a lantern with the property that all the involved curves are non-separating. It is well-known, 
and otherwise easy to see, that $X$ contains a non-separating lantern if $X$ has genus at least $3$. In particular
we deduce that, as long as $X$ has genus $g\ge 3$, every non-separating curve belongs to a non-separating lantern.

\subsection{Centralizers of Dehn twists}
Observe that the relation $f\delta_\gamma f^{-1}=\delta_{f(\gamma)}$, for $f\in\Map(X)$ and $\gamma\subset X$ a
curve, implies that $$\CZ(\delta_\gamma) = \{f \in \Map(X) \mid f(\gamma) = \gamma\},$$ where $\CZ(\delta_\gamma)$ denotes the centralizer of $\delta_{\gamma}$ in $\Map(X)$. Notice that $\CZ(\delta_\gamma)$ 
is also equal to the normalizer $\CN(\langle\delta_\gamma\rangle)$ of the subgroup of $\Map(X)$ generated by 
$\delta_\gamma$.

An element in $\Map(X)$ which fixes $\gamma$ may either switch the sides of $\gamma$ or may preserve them. 
We denote by $\CZ_0(\delta_\gamma)$ the group of those elements which preserve sides; observe that 
$\CZ_0(\delta_\gamma)$ has index at most 2 in $\CZ(\delta_\gamma)$.

The group $\CZ_0(\delta_\gamma)$ is closely related to two different mapping class groups. First, 
let $X_\gamma$ be the surface obtained by removing the interior of a closed regular neighborhood $\gamma\times[0,1]$ 
of $\gamma$ from $X$. Every homeomorphism of $X_\gamma$ fixing pointwise the boundary and the punctures extends to a
 homeomorphism $X\to X$ which is the identity on $X\setminus X_\gamma$. This induces a homomorphism 
$\Map(X_\gamma)\to\Map(X)$; more concretely we have the following exact sequence:
\begin{equation}\label{eq:central1}
0\to\BZ\to\Map(X_\gamma)\to \CZ_0(\delta_\gamma)\to 1
\end{equation}
Here, the group $\BZ$ is generated by the difference $\delta_{\eta_1}\delta^{-1}_{\eta_2}$ of the Dehn twists along 
$\eta_1$ and $\eta_2$, the new boundary curves of $X_\gamma$.

Instead of deleting a regular neighborhood of $\gamma$ we could also delete $\gamma$ from $X$. Equivalently, 
let $X_\gamma'$ be the surface obtained from $X_\gamma$ by deleting the new boundary curves of $X_{\gamma}$. Every 
homeomorphism of  $X$ fixing $\gamma$ induces a homeomorphism of $X_\gamma'$. This yields a second exact sequence
\begin{equation}\label{eq:central2}
0\to\langle\delta_\gamma\rangle\to \CZ_0(\delta_\gamma)\to\Map(X_\gamma')\to 1
\end{equation}

\subsection{Multitwists}
\label{subsec:multitwists}
To a multicurve $\eta\subset X$ we associate the group 
$$\BT_\eta=\langle\{\delta_\gamma,\ \gamma\subset\eta\}\rangle\subset\Map(X)$$
generated by the Dehn twists along the components of $\eta$. We refer to the elements in $\BT_\eta$ as {\em multitwists} 
along $\eta$. Observe that $\BT_\eta$ is abelian; more concretely, $\BT_\eta$ is isomorphic to the free abelian group with 
rank equal to the number of components of $\eta$. 

Let $\eta\subset X$ be a multicurve. An element $f\in\BT_\eta$ which does not belong to any $\BT_{\eta'}$, 
for some $\eta'$ properly contained in $\eta$, is said to be a {\em generic multitwist along $\eta$}. Conversely, if 
$f\in\Map(X)$ is a multitwist, then the {\em support} of $f$ is the smallest multicurve $\eta$ such that $f$ 
is a generic multitwist along $\eta$.

Much of what we just said about Dehn twists extends easily to multitwists. For instance, if $\eta\subset X$ is a 
multicurve, then we have 
$$\BT_{f(\eta)}=f\BT_\eta f^{-1}$$
for all $f\in\Map(X)$. In particular, the normalizer $\CN(\BT_\eta)$ of $\BT_\eta$ in $\Map(X)$ is equal to
$$\CN(\BT_\eta)=\{f\in\Map(X)\vert f(\eta)=\eta\}$$
On the other hand, the centralizer $\CZ(\BT_\eta)$ of $\BT_\eta$ is the intersection of the centralizers of its generators; 
hence
$$\CZ(\BT_\eta)=\{f\in\Map(X)\vert f(\gamma)=\gamma\ \hbox{for every component}\ \gamma\subset\eta\}$$
Notice that $\CN(\BT_\eta)/\CZ(\BT_\eta)$ acts by permutations on the set of components of $\eta$.  For further use we 
observe that if the multicurve $\eta$ happens to be a cut system, then $\CN(\BT_\eta)/\CZ(\BT_\eta)$ is in fact isomorphic 
to the group of permutations of the components of $\eta$.

Denote by $\CZ_0(\BT_\eta)$ the subgroup of $\CZ(\BT_\eta)$ fixing not only the components but also the sides of 
each component. Notice that $\CZ(\BT_\eta)/\CZ_0(\BT_\eta)$ is a subgroup of $(\BZ/2\BZ)^{\vert\eta\vert}$ and hence 
is abelian.

\medskip


Observe that it follows from the definition of the mapping class group and from the relation 
$\delta_{f(\gamma)}=f\delta_\gamma f^{-1}$ that every Dehn twist along a boundary component of 
$X$ is central in $\Map(X)$. In fact, as long as $X$ has at least  genus $3$, such Dehn twists generate 
the center of $\Map(X)$:

\begin{sat}\label{center}
If $X$ has genus at least $3$ then the group $\BT_{\D X}$ generated by Dehn twists along the boundary components of
 $X$ is the center of $\Map(X)$. Moreover, we have
$$1\to\BT_{\D X}\to\Map(X)\to\Map(X')\to 1$$
where $X'$ is the surface obtained from $X$ by deleting the boundary.
\end{sat}

Notice that if $X$ is a surface of genus $g \in \{1,2\}$, with empty boundary and no marked points,
then the center of $\Map(X)$ is generated by the hyperelliptic involution.

\subsection{Roots}
It is a rather surprising, and annoying, fact that such simple elements in $\Map(X)$ as Dehn twists have non-trivial
 {\em roots} \cite{Margalit-Schleimer}. Recall that a root of $f\in\Map(X)$ is an element $g\in\Map(X)$ for which there 
is $k\in\BZ$ with $f=g^k$. Being forced to live with roots, we state here a few simple but important observations:

\begin{lem}\label{root1}
Suppose that $f\in\BT_\eta$ and $f'\in\BT_{\eta'}$ are generic multitwists along multicurves $\eta,\eta'\subset X$.
 If $f$ and $f'$ have a common root, then $\eta=\eta'$.
\end{lem}

\begin{lem}\label{root2}
Suppose that $f\in\BT_\eta$ is a generic multitwist along a multicurve $\eta$ and $f'\in\Map(X)$ is 
a root of $f$. Then $f'(\eta)=\eta$ and hence $f'\in\CN(\BT_\eta)$.
\end{lem}

If $\eta$ is an essential curve then $\BT_\eta=\langle\delta_\eta\rangle$; hence $\CN(\BT_\eta)=\CZ(\delta_\eta)$ 
is the subgroup of $\Map(X)$ preserving $\eta$. Recall that $\CZ_0(\delta_\eta)$ is the subgroup of $\CZ(\delta_\eta)$ 
which preserves sides of $\eta$.

\begin{lem}\label{root3}
Suppose that $\delta_\eta\in\Map(X)$ is a Dehn twist along an essential curve $\eta$. For $f\in\CZ_0(\delta_\eta)$ the 
following are equivalent:
\begin{itemize}
\item $f$ is a root of a power of $\delta_\eta$, and
\item the image of $f$ in $\Map(X_\eta')$ under the right arrow in \eqref{eq:central2} has finite order.
\end{itemize}
Moreover, $f$ is itself a power of $\delta_\eta$ if and only if the image of $f$ in $\Map(X_\eta')$ is trivial.
\end{lem}

\subsection{Torsion}
In the light of Lemma \ref{root3}, it is clear that the existence of roots is closely related to the presence of 
torsion in mapping class groups. While it is known that every mapping class group always contains a finite index 
torsion-free subgroup, this is not going to be of much use here. The fact that the
mapping class group of a surface with boundary is torsion-free is going to be of more importance.

\begin{sat}\label{thm:boundary-no-torsion}
If $X$ is a surface with nonempty boundary, then $\Map(X)$ is torsion-free. Similarly, if $X$ has marked points, then 
finite subgroups of $\Map(X)$ are cyclic.
\end{sat}

The key to understand torsion in mapping class groups is the resolution by Kerckhoff \cite{Kerckhoff} of the Nielsen 
realization problem: the study of finite subgroups of the mapping class group reduces to the study of groups of automorphism 
of Riemann surfaces. For instance, it follows from the classical Hurewitz theorem that the order of such a group is 
bounded from above solely in terms of the genus of the underlying surface. Below we will need the following bound, 
due to Maclachlan \cite{Maclachlan} and Nakajima \cite{Nakajima},  for the order of finite abelian subgroups of $\Map(X)$.

\begin{sat}\label{Maclachlan}
Suppose that $X$ has genus $g \ge 2$. Then $\Map(X)$ does not contain finite abelian groups with more than $4g+4$ elements.
\end{sat}

We remark that if $g\le 5$ all finite subgroups, abelian or not, of $\Map(X)$ have 
been listed \cite{kuku,kuku5}. In the sequel we will make use of this list in the case that $g=3,4$.

Finally, we  observe that a finite order diffeomorphism which is isotopic to the identity is in fact 
the identity. This implies, for instance, that if $\bar X$ is obtained from $X$ by filling in punctures, and $\tau:X\to X$ 
is a finite order diffeomorphism representing a non-trivial element in $\Map(X)$, then the induced mapping class of
 $\bar X$ is non-trivial as well.

\subsection{Centralizers of finite order elements}
By \eqref{eq:central1} and \eqref{eq:central2}, centralizers of Dehn twists are closely related to other mapping 
class groups. Essentially the same is true for centralizers of other mapping classes. We now discuss the case of 
torsion elements. The following result follows directly from the work of Birman-Hilden \cite{Birman-Hilden}:

\begin{sat}[Birman-Hilden]\label{Birman-Hilden}
Suppose that $[\tau]\in\Map(X)$ is an element of finite order and let $\tau:X\to X$ be a finite order diffeomorphism 
representing $[\tau]$. Consider the orbifold $\CO=X/\langle\tau\rangle$ and let $\CO^*$ be the surface obtained from 
$\CO$ by removing the singular points. Then we have a sequence
$$1\to\langle[\tau]\rangle\to\CZ([\tau])\to\Map^*(\CO^*)$$
where $\Map^*(\CO^*)$ is the group of isotopy classes of all homeomorphisms $\CO^*\to\CO^*$.
\end{sat}

Hidden in Theorem \ref{Birman-Hilden} we have the following useful fact: {\em two finite order diffeomorphisms 
$\tau,\tau':X\to X$ which are isotopic are actually conjugate as diffeomorphisms} (see the remark in \cite[p.10]{Bestvina-Church-Souto}). 
Hence, it follows that the surface $\CO^*$ in Theorem \ref{Birman-Hilden} depends only on the mapping class $[\tau]$. 
Abusing notation, in the sequel we will speak about the fixed-point set of a finite order element in $\Map(X)$.

\section{Homomorphisms induced by embeddings}
\label{sec:embeddings}
In this section we define what is meant by an embedding $\iota:X\to Y$ between surfaces. 
As we will observe, any embedding induces
a homomorphism between the corresponding mapping class groups. We will discuss several
standard examples of such homomorphisms, notably the so-called Birman exact sequences. We will conclude the section with a 
few observations that will be needed later on. Besides the possible differences of terminology, all the facts
that we will state are either well known or simple observations in 2-dimensional topology. A reader
who is reasonably acquainted with \cite{Farb-Margalit,Ivanov} will have no difficulty filling the details.


\subsection{Embeddings}
 Let $X$ and $Y$ be surfaces of finite topological type, and consider 
their cusps to be marked points. Denote by $\vert X\vert $ and $\vert Y\vert$ the compact surfaces obtained 
from $X$ and $Y$, respectively, by forgetting all the marked points, and let
 $P_X\subset\vert X\vert$ and  $P_Y\subset\vert Y\vert$ be the sets of marked points of $X$ and $Y$.

\begin{defi*}
An {\em embedding} $\iota:X\to Y$ is a continuous injective map $\iota:\vert X\vert\to\vert Y\vert$ 
such that $\iota^{-1}(P_Y)\subset P_X$. An embedding is said to be a {\em homeomorphism} if it has an inverse
which is also an embedding.
\end{defi*}

We will say that two embeddings $\iota,\iota':X\to Y$ are {\em equivalent} or {\em isotopic} if there if a continuous injective 
map $f:\vert Y\vert\to\vert Y\vert$ which is isotopic (but not necessarily ambient isotopic) to the identity 
relative to the set $P_Y$ of marked points of $Y$ and which satisfies $f\circ\iota=\iota'$. Abusing terminology, we will often say that equivalent embeddings are the same.

Given an embedding $\iota:X\to Y$ and a homeomorphism $f:X\to X$ which pointwise fixes the boundary and the marked points 
of $X$, we consider the homeomorphism
$$\iota(f):Y\to Y$$
given by $\iota(f)(x)=(\iota\circ f\circ\iota^{-1})(x)$ if $x\in\iota(X)$ and $\iota(f)(x)=x$ otherwise. Clearly, $\iota(f)$
 is a homeomorphism which pointwise fixes the boundary and the marked points of $Y$. In particular $\iota(f)$ 
represents an element $\iota_\#(f)$ in $\Map(Y)$. It is easy to check that we obtain a well-defined group homomorphism 
$$\iota_{\#}:\Map(X)\to\Map(Y)$$
characterized by the following property: for any curve $\gamma\subset X$ we have 
$\iota_\#(\delta_\gamma)=\delta_{\iota(\gamma)}$. Notice that this characterization immediately implies 
that if $\iota$ and $\iota'$ are isotopic, then $\iota_\#=\iota'_{\#}$.

\subsection{Birman exact sequences} As we mentioned above, notable examples of homomorphisms induced by embeddings are 
the so-called Birman exact sequences, which we now describe. 

Let $X$ and $Y$ be surfaces of finite topological type. We will say that $Y$ is obtained from $X$ by 
{\em filling in a puncture} if there is an embedding $\iota:X\to Y$ and a marked point $p\in P_X$, such that
the underlying map $\iota:\vert X\vert\to\vert Y\vert$ is a homeomorphism, and $\iota^{-1}(P_Y)=P_X\setminus\{p\}$.
If $Y$ is obtained from $X$ by filling in a puncture we have the following exact sequence:
\begin{equation}\label{eq:birman1}
\xymatrix{
1\ar[r] & \pi_ 1(\vert Y\vert\setminus P_Y,\iota(p))\ar[r] & \Map(X)\ar[r]^{\iota_\#} & \Map(Y)\ar[r] & 1}
\end{equation}
The left arrow in \eqref{eq:birman1} can be described concretely. For instance, if $\gamma$ is a simple loop in
 $\vert Y\vert$ based $\iota(p)$ and avoiding all other marked points of $Y$, then the image of the element 
$[\gamma]\in\pi_1(\vert Y\vert\setminus P_Y,\iota(p))$ in $\Map(X)$ is the difference of the two Dehn twists 
along the curves forming the boundary of a regular neighborhood of $\iota^{-1}(\gamma)$.

Similarly, we will say that $Y$ is obtained from $X$ by {\em filling in a boundary component} if there is an 
embedding $\iota:X\to Y$, with $\iota^{-1}(P_Y)=P_X$, and such that the complement in $\vert Y\vert$ of the image 
of the underlying map $\vert X\vert\to\vert Y\vert$ is a disk which does not contain any marked point of $Y$. 
If $Y$ is obtained from $X$ by filling in a boundary component then we have the following exact sequence:
\begin{equation}\label{eq:birman2}
1\to \pi_1(T^1(\vert Y\vert\setminus P_Y))\to\Map(X)\to\Map(Y)\to 1
\end{equation}
Here $T^1(\vert Y\vert\setminus P_Y)$ is the unit-tangent bundle of the surface $\vert Y\vert\setminus P_Y$.

We refer to the sequences \eqref{eq:birman1} and \eqref{eq:birman2} as the Birman exact sequences.

\subsection{Other building blocks}
As we will see in the next subsection, any embedding is a composition of four basic building blocks. Filling punctures 
and filling boundary components are two of them; next, we describe 
the other two types.

Continuing with the same notation as above, we will say that $Y$ is obtained from $X$ by {\em deleting a boundary component} if 
there is an embedding $\iota:X\to Y$ with $\iota(P_X)\subset P_Y$ and such that the complement of the image of the 
underlying map $\vert X\vert\to\vert Y\vert$ is disk containing exactly one point in $P_Y$. If $Y$ is obtained from 
$X$ by deleting a boundary component then we have
\begin{equation}\label{eq:birman3}
1\to\BZ\to\Map(X)\to\Map(Y)\to 1
\end{equation}
where, $\BZ$ is the group generated by the Dehn twist along the forgotten boundary component. 

Finally, we will say that $\iota:X\to Y$ is a {\em subsurface embedding} if $\iota(P_X)\subset P_Y$ and if no 
component of the complement of the image of the underlying map $\vert X\vert\to\vert Y\vert$ is a disk containing at most one marked point. 
Notice that if $\iota:X\to Y$ is a subsurface embedding, then the homomorphism 
$$\iota_\#:\Map(X)\to\Map(Y)$$
is injective if and only if $\iota$ is {\em anannular}, i.e. if no component of the complement of the image of 
the underlying map $\vert X\vert\to\vert Y\vert$ is an annulus without marked points; compare with \eqref{eq:central1}
 above.

\subsection{General embeddings}
Clearly, the composition of two embeddings is an embedding. For instance, observe that filling in a boundary component is the same as first forgetting it and then filling in a puncture. The following proposition, whose proof we leave to the reader, asserts that every embedding is isotopic to a suitable composition of the elementary building blocks we have just discussed:

\begin{prop}\label{prop-embedding}
Every embedding $\iota:X\to Y$ is isotopic to a composition of the following three types of embedding: filling punctures, 
 deleting boundary components, and subsurface embeddings. In particular, the homomorphism $\iota_\#:\Map(X)\to\Map(Y)$ 
is injective if and only if $\iota$ is an anannular subsurface embedding.\qed 
\end{prop}

\noindent{\bf Notation.} In order to avoid notation as convoluted as $T^1(\vert Y\vert\setminus P_Y)$, most of the time
we will drop any reference to the underlying surface $\vert Y\vert$ or to the set of marked point $P_Y$; 
notice that this is consistent with taking the freedom to consider punctures as marked  points or as ends. 
For instance, the Birman exact sequences now read 
$$1\to\pi_1(Y)\to\Map(X)\to\Map(Y)\to 1,$$
if $Y$ is obtained from $X$ by filling in a puncture, and 
$$1\to\pi_1(T^1Y)\to\Map(X)\to\Map(Y)\to 1,$$
if it is obtained filling in a boundary component. We hope that this does not cause any confussion.

\subsection{Two observations}
We conclude this section with two observations that will be needed below.

First, suppose that $\iota:X\to Y$ is an embedding and let $\eta\subset X$ be a multicurve. 
The image $\iota(\eta)$ of $\eta$ in $Y$ is an embedded 1-manifold, but it does not need to be a 
multicurve. For instance, some component of $\iota(\eta)$ may not be essential in $Y$; also two 
components of $\iota(\eta)$ may be parallel in $Y$. If this is not the case, that is, if $\iota(\eta)$ is 
a multicurve in $Y$, then it is easy to see that $\iota_\#$ maps the subgroup $\BT_\eta$ of multitwists 
supported on $\eta$ isomorphically onto $\BT_{\iota(\eta)}$. We record this observation in the following lemma:

\begin{lem}\label{BHB}
Let $\iota:X\to Y$ be an embedding and let $\eta\subset X$ be a multicurve such that 
\begin{itemize}
\item every component of $\iota(\eta)$ is essential in $Y$, and
\item no two components of $\iota(\eta)$ are parallel in $Y$,
\end{itemize}
Then $\iota(\eta)$ is also a multicurve in $Y$ and the homomorphism $\iota_\#$ maps $\BT_\eta\subset\Map(X)$ isomorphically to $\BT_{\iota(\eta)}\subset\Map(Y)$. Moreover, the image of a generic multitwist in $\BT_\eta$ 
is generic in $\BT_{\iota(\eta)}$.\qed
\end{lem}

Finally, we recall the well-known fact that the Birman exact sequence \eqref{eq:birman1} does not split:

\begin{lem}\label{no-birman-split}
Suppose that $X$ is a surface of genus $g\ge 3$ with empty boundary and a single puncture. 
If $Y$ is the closed surface obtained from $X$ by filling in the puncture, then the exact sequence \eqref{eq:birman1}
$$1\to\pi_1(Y)\to\Map(X)\to\Map(Y)\to 1$$
does not split.
\end{lem}
\begin{proof}
The mapping class group $\Map(Y)$ of the closed surface $Y$ contains a non-cyclic finite subgroup; namely one 
isomorphic to $\BZ/2\BZ\times\BZ/2\BZ$. By Theorem \ref{thm:boundary-no-torsion} such a subgroup does not exist 
in $\Map(X)$, which proves that there is no splitting of \eqref{eq:birman1}.
\end{proof}

In Lemma \ref{no-birman-split} we proved that in a very particular situation one of the two Birman exact 
sequences does not split. Notice however that it follows from Theorem \ref{dogs-bollocks} that they never do; 
this also follows from the work of Ivanov-McCarthy \cite{Ivanov-McCarthy}.

\section{Triviality theorems}
In this section we remind the reader of two triviality theorems for homomorphisms from mapping class groups 
to abelian groups and permutation groups; these results are widely used throughout this paper. The first of these 
results is a  direct consequence of Powell's theorem \cite{Powell} on the vanishing of the integer homology of the 
mapping class group of surfaces of genus at least $3$:

\begin{sat}[Powell]\label{thm:homology}
If $X$ is a surface of genus $g\ge 3$ and $A$ is an abelian group, then every homomorphism $\Map(X)\to A$ is trivial.
\end{sat}

We refer the reader to Korkmaz \cite{Korkmaz} for a discussion of Powell's theorem and other homological properties of 
mapping  class groups. 

As a first consequence of Theorem \ref{thm:homology} we derive the following useful observation:

\begin{lem}\label{trick}
Let $X,Y$ and $\bar Y$ be surfaces of finite topological type, and let $\iota:Y\to\bar Y$ be an embedding. 
Suppose that $X$ has genus at least $3$ and that $\phi:\Map(X)\to\Map(Y)$ is a homomorphism such that the composition 
$$\bar\phi=\iota_\#\circ\phi:\Map(X)\to\Map(\bar Y)$$ 
is trivial. Then $\phi$ is trivial as well.
\end{lem}
\begin{proof}
By Proposition \ref{prop-embedding} the embedding $\iota:Y\to\bar Y$ is a suitable composition of filling in 
punctures and boundary components, deleting boundary components and subsurface embeddings. In particular, 
we may argue by induction and assume that $\iota$ is of one of these four types. For the sake of concreteness 
suppose $\iota:\bar Y\to Y$ is the embedding associated to filling in a boundary component; the other cases are 
actually a bit easier and are left to the reader. We have the following diagram:
$$\xymatrix{
& & \Map(X)\ar[d]_\phi\ar[dr]^{\bar\phi} & & \\
1\ar[r]& \pi_1(T^1\bar Y)\ar[r] &\Map(Y)\ar[r]^{\iota_\#}& \Map(\bar Y)\ar[r] & 1}$$
The assumption that $\bar\phi$ is trivial amounts to supposing that the image of $\phi$ is contained in 
$\pi_1(T^1\bar Y)$. The homomorphism $\pi_1(T^1\bar Y)\to\pi_1(\bar Y)$ yields another diagram:
$$\xymatrix{
& & \Map(X)\ar[d]_\phi\ar[dr]^{\phi'} & & \\
1\ar[r]& \BZ\ar[r] &\pi_1(T^1\bar Y)\ar[r]& \pi_1(\bar Y)\ar[r] & 1}$$
Since every nontrivial subgroup of the surface group $\pi_1(\bar Y)$ has nontrivial homology, we deduce 
from Theorem \ref{thm:homology}  that $\phi'$ is trivial. Hence, the image of $\phi$ is contained in $\BZ$. 
Applying again Theorem \ref{thm:homology} above we deduce that $\phi$ is trivial, as it was to be shown.
\end{proof}

Before stating another consequence of Theorem \ref{thm:homology} we need a definition:

\begin{defi}
\label{def-irred}
A homomorphism $\phi:\Map(X)\to\Map(Y)$ is said to be {\em irreducible} if its image does not fix any essential curve in $Y$; otherwise we say it is {\em reducible}.
\end{defi}

\begin{bem}
Recall that we consider boundary parallel curves to be essential. In particular, every homomorphism $\Map(X)\to\Map(Y)$ is reducible if $Y$ has non-empty boundary.
\end{bem}

Let $\phi:\Map(X)\to\Map(Y)$ be a reducible homomorphism, where $X$ has genus at least 3, and let
$\eta\subset Y$ be a multicurve which is componentwise invariant under $\phi(\Map(X))$; in other words, 
$\phi(\Map(X))\subset\CZ(\BT_\eta)$. 
Moreover, notice that Theorem \ref{thm:homology} implies that $\phi(\Map(X))\subset\CZ_0(\BT_\eta)$, where
$\CZ_0(\BT_\eta)$ is 
the subgroup of $\CZ(\BT_\eta)$ consisting of those elements that fix the sides of each component of $\eta$.


Now let $Y_\gamma'=Y\setminus\eta$ be the surface obtained by deleting $\eta$ from $Y$. Composing \eqref{eq:central2} 
as often as necessary, we obtain an exact sequence as follows:
\begin{equation}\label{eq:central3}
1\to\BT_\eta\to\CZ_0(\BT_\eta)\to\Map(Y_\eta')\to 1
\end{equation} 
The same argument of the proof of Lemma \ref{trick} shows that $\phi$ is trivial if the composition of $\phi$ 
and the right homomorphism \eqref{eq:central3} is trivial. Hence we have:

\begin{lem}\label{trick2}
Let $X,Y$ be surfaces of finite topological type. Suppose that $X$ has genus at least $3$ and that 
$\phi:\Map(X)\to\Map(Y)$ is a non-trivial reducible homomorphism fixing the multicurve $\eta\subset Y$. 
Then $\phi(\Map(X))\subset\CZ_0(\BT_\eta)$ and the composition of $\phi$ with the homomorphism \eqref{eq:central3} 
is not trivial.\qed
\end{lem}

The second trivially theorem, due to Paris \cite{Paris}, asserts that the mapping class group of a surface of
genus $g\ge 3$ does not have  subgroups of index less than or equal to $4g+4$; equivalently, any homomorphism from the mapping class 
group into a symmetric group $\CS_k$ is trivial if $k\le 4g+4$:

\begin{sat}[Paris]\label{thm:luisito}
If $X$ has genus $g\ge 3$ and $k\le 4g+4$, then there is no nontrivial homomorphism $\Map(X)\to\CS_k$ where the 
latter group is the group of permutations of the set with $k$ elements.
\end{sat}

Before going any further we should mention that in \cite{Paris}, Theorem \ref{thm:luisito} is only stated for closed surfaces. However, the proof works as it is also for surfaces with 
boundary and or punctures. We leave it to the reader to check that this is the case.

As a first consequence of Theorem \ref{thm:homology} and Theorem \ref{thm:luisito} we obtain the following special case of Proposition \ref{blabla}:

\begin{prop}\label{genus2}
If $X$ has genus at least $3$ and $Y$ at most genus $2$, then every homomorphism $\phi:\Map(X)\to\Map(Y)$ is trivial.
\end{prop}
\begin{proof}
Assume for concreteness that $Y$ has genus $2$; the cases of genus $0$ and genus $1$ are in fact easier and are left
to the reader. 

Notice that by Lemma \ref{trick}, we may assume without lossing generality that $Y$ has empty boundary and no marked points.
Recall that $\Map(Y)$ has a central element $\tau$ of order 2, namely the hyperelliptic involution. As we discussed 
above, we identify the finite order mapping class $\tau$ 
with one of its finite order representatives, which we again denote by $\tau$. The surface underlying the orbifold 
$Y/\langle\tau\rangle$ is the $6$-punctured sphere $\BS_{0,6}$. By Theorem \ref{Birman-Hilden} we have the following 
exact sequence:
$$1\to\langle\tau\rangle\to\Map(Y)\to\Map^*(\BS_{0,6})\to 1$$
where $\Map^*(\BS_{0,6})$ is, as always, the group of isotopy classes of all orientation preserving homeomorphisms 
of $\BS_{0,6}$. Therefore, any homomorphism $\phi:\Map(X)\to\Map(Y)$ induces a homomorphism
$$\phi':\Map(X)\to\Map^*(\BS_{0,6})$$
By Paris' theorem, the homomorphism obtained by composing $\phi'$ with the obvious 
homomorphism $\Map^*(\BS_{0,6})\to\CS_6$, the group of permutations of the punctures, is trivial. 
In other words, $\phi'$ takes values in $\Map(\BS_{0,6})$. Since the mapping class group of 
the standard sphere $\BS^2$ is trivial, Lemma \ref{trick} implies that $\phi'$ is trivial. 
Therefore, the image of $\phi$ is contained in the abelian subgroup $\langle\tau\rangle\subset\Map(Y)$.
Finally, Theorem \ref{thm:homology} implies that $\phi$ is trivial, as we had to show.
\end{proof}

\section{Getting rid of the torsion}\label{sec:rid-torsion}
We begin this section reminding the reader of a question posed in the introduction:

\begin{named}{Question \ref{question}}
Suppose that $\phi:\Map(X)\to\Map(Y)$ is a homomorphism between mapping class groups 
of surfaces of genus at least $3$, with the property that the image of every Dehn twist along a non-separating 
curve has finite order. Is the image of $\phi$ finite?
\end{named}

In this section we will give a positive answer to the question above if the genus of $Y$ is exponentially bounded by 
the genus of $X$. Namely:

\begin{prop}\label{no-torsion}
Suppose that $X$ and $Y$ are surfaces of finite topological type with genera $g$ and $g'$ respectively.
 Suppose that $g\ge 4$ and that either $g'< 2^{g-2}-1$ or $g'=3,4$.

Any homomorphism $\phi:\Map(X)\to\Map(Y)$ which maps a Dehn twist along a non-separating curve to a finite 
order element is trivial.
\end{prop}

Under the assumption that $Y$ is not closed, we obtain in fact a complete answer to the question above:

\begin{named}{Theorem \ref{blablabla}}
Suppose that $X$ and $Y$ are surfaces of finite topological type, that $X$ has genus at least $3$, and that $Y$ is not closed. 
Then any homomorphism $\phi:\Map(X)\to\Map(Y)$ which maps a Dehn twist along a non-separating curve to a finite order element is trivial.
\end{named}

Recall that by Theorem \ref{thm:boundary-no-torsion}, the mapping class group of a surface with non-empty boundary is torsion-free.
Hence we deduce from Lemma \ref{dehn-trivial-trivial} that it suffices to consider the case that $\D Y=\emptyset$. From now on, we assume that we are in this situation.

The proofs of Proposition \ref{no-torsion} and Theorem \ref{blablabla} are based on Theorem \ref{thm:homology}, 
the connectivity of the cut system complex, and the following algebraic observation:

\begin{lem}\label{chinese}
For $n\in\BN$, $n\ge 2$, consider $\BZ^n$ endowed with the standard action of the symmetric group $\CS_n$ by 
permutations of the basis elements $e_1, \ldots, e_n$. If $V$ is a finite abelian group equipped with an $\CS_n$-action, 
then for any $\CS_n$-equivariant epimorphism $\phi:\BZ^n\to V$ one of the following two is true:
\begin{enumerate}
\item Either the restriction of $\phi$ to $\BZ^{n-1}\times\{0\}$ is surjective, or
\item $V$ has order at least $2^n$ and cannot be generated by fewer than $n$ elements.
\end{enumerate}
Moreover, if (1) does not hold and $V\neq (\BZ/2\BZ)^n$ then $V$ has at least $2^{n+1}$ elements.
\end{lem}
\begin{proof}
Let $d$ be the order of $\phi(e_1)$ in $V$ and observe that, by $\CS_n$-equivariance, all the elements $\phi(e_i)$ also have 
order $d$. It follows that $(d\BZ)^n\subset\Ker(\phi)$ and hence that $\phi$ descends to an epimorphism
$$\phi':(\BZ/d\BZ)^n\to V$$
Our first goal is to restrict to the case that $d$ is a power of a prime. In order to do this, consider the prime 
decomposition $d=\prod_j p_j^{a_j}$ of $d$, where $p_i\neq p_j$ and $a_i\in\BN$. By the Chinese remainder theorem we have
$$\BZ/d\BZ=\prod_j\left ( \BZ/p_j^{a_j}\BZ \right )$$
Hence, there is a $\CS_n$-equivariant isomorphism
$$(\BZ/d\BZ)^n=\prod_j\left((\BZ/p_j^{a_j}\BZ)^n\right)$$
Consider the projection $\pi_j:\BZ^n\to(\BZ/p_j^{a_j}\BZ)^n$ and observe that if the restriction to 
$\BZ^{n-1}\times\{0\}$ of $\phi'\circ\pi_j$ surjects  onto $\phi'((\BZ/p_j^{a_j}\BZ)^n)$ for all $j$, then 
$\phi(\BZ^{n-1}\times\{0\})=V$. 

Supposing that this were not the case, replace $\phi$ by 
$\phi' \circ \pi_j$ and $V$ by $\phi((\BZ/p_j^{a_j}\BZ)^n)$. In more concrete terms, 
we can assume from now on that $d=p^a$ is a power of a prime. 

At this point we will argue by induction. The key claim is the following surely well-known observation:
\medskip

\noindent{\bf Claim.} {\em  Suppose that $p$ is prime. The only $\CS_n$-invariant subgroups $W$ of $(\BZ/p\BZ)^n$ are the 
following:
\begin{itemize}
\item The trivial subgroup $\{0\}$, 
\item $(\BZ/p\BZ)^n$ itself,
\item $E=\{(a,a,\dots,a)\in(\BZ/p\BZ)^n\vert a=0,\dots,p-1\}$, and
\item $F=\{(a_1,\dots,a_n)\in(\BZ/p\BZ)^n\vert a_1+\dots+a_n=0\}$.
\end{itemize}
}
\begin{proof}[Proof of the claim]
Suppose that $W\subset(\BZ/p\BZ)^n$ is not trivial and take $v=(v_i)\in W$ nontrivial. If $v$ cannot be chosen 
to have distinct entries then $W=E$. So suppose that this is not case and choose $v$ with two distinct entries, 
say $v_1$ and $v_2$. Consider the image $v'$ of $v$ under the transposition $(1,2)$. By the $\CS_n$-invariance of 
$W$ we have $v'\in W$ and hence $v-v'\in W$. By construction, $v-v'$ has all entries but the two first ones equal 
to $0$. Moreover, each of the first two entries is the negative of the other one. Taking a suitable power we find 
that $(1,-1,0,\dots,0)\in W$. By the $\CS_n$-invariance of $W$ we obtain that every element with one $1$, one $-1$ 
and $0$ otherwise belongs to $W$. These elements span $F$. 

We have proved that either $W$ is trivial, or $W=E$ or $F\subset W$. Since the only subgroups containing $F$ are $F$ 
itself and the total space, the claim follows.
\end{proof}

Returning to the proof of Lemma \ref{chinese}, suppose first that $a=1$, i.e. $d=p$ is prime. The kernel of the epimorphism
$$\phi':(\BZ/p\BZ)^n\to V$$
is a $\CS_n$-invariant subspace. Either $\phi'$ is injective, and thus $V$ contains $p^n\ge 2^n$ elements, or its kernel is one of the spaces $E$ or $F$ provided by the claim. Since the union of either one of them with $(\BZ/p\BZ)^{n-1}\times\{0\}$ spans $(\BZ/p\BZ)^n$, 
it follows that the restriction of $\phi$ to $\BZ^{n-1}\times\{0\}$ surjects onto $V$. This concludes the proof if $a=1$.

Suppose that we have proved the result for $a-1$. We can then consider the diagram:
$$\xymatrix{
0\ar[r] & (\BZ/p^{a-1}\BZ)^n\ar[r]\ar[d] &(\BZ/p^a\BZ)^n \ar[r]\ar[d]^{\phi'} & (\BZ/p\BZ)^n\ar[r]\ar[d] & 0 \\
0\ar[r] & \phi'((\BZ/p^{a-1}\BZ)^n)\ar[r] & V \ar[r] & V/\phi'((\BZ/p^{a-1}\BZ)^n)\ar[r] & 0
}$$
Observe that if one of the groups to the left and right of $V$ on the bottom row has at least $2^n$ elements, 
then so does $V$. So, if this is not the case we may assume by induction that the restriction of the left and right
 vertical arrows to $(\BZ/p^{a-1}\BZ)^{n-1}\times\{0\}$ and $(\BZ/p\BZ)^{n-1}\times\{0\}$ are epimorphisms. 
This shows that the restriction of $\phi'$ to $(\BZ/p^a\BZ)^{n-1}\times\{0\}$ is also an epimorphism. It follows that either $V$ has at least $2^n$ elements or the restriction of $\phi$ to $\BZ^{n-1}\times\{0\}$ is surjective, as claimed.

Both the equality case and the claim on the minimal number of elements needed to generate $V$ are left to the reader.
\end{proof}

We are now ready to prove:

\begin{lem}\label{torsion-bound}
Given $n\ge 4$, suppose that $g>0$ is such that $2^{n-2}-1>g$ or $g\in \{3,4\}$. 

If $Y$ is surface of genus $g\ge 3$, $V\subset\Map(Y)$ is a finite abelian group endowed with an 
action of $\CS_n$, and $\phi:\BZ^n\to V$ is a $\CS_n$-equivariant epimomorphism, then the restriction 
of $\phi$ to $\BZ^{n-1}\times\{0\}$ is surjective. 
\end{lem}
\begin{proof}
Suppose, for contradiction, that the restriction of $\phi$ to $\BZ^{n-1}\times\{0\}$ 
is not surjective. Recall that by the resolution of the Nielsen realization problem \cite{Kerckhoff} 
there is a conformal structure on $Y$ such that $V$ can be represented by a group of automorphisms.

Suppose first that $2^{n-2}-1>g$. Since we are assuming that the
restriction of $\phi$ to $\BZ^{n-1}\times\{0\}$ is not surjective, Lemma \ref{chinese} implies that
$V$ has at least $2^n$ elements. Then:
$$2^n=4(2^{n-2}-1)+4>4g+4,$$
which is impossible since Theorem \ref{Maclachlan} asserts that $\Map(Y)$ does  not contain finite abelian 
groups with more than $4g+4$ elements.


Suppose now that $g=4$. If $n \geq 5$ we obtain a contradiction using the same argument as above. 
Thus assume that $n=4$. Since $2^{4+1}=32>20= 4 \cdot 4 + 4$, it follows from the equality statement in 
Lemma \ref{chinese} that $V$ is isomorphic to $(\BZ/2\BZ)^4$. Luckily for us, 
Kuribayashi-Kuribayashi \cite{kuku}  have classified all groups of automorphisms of Riemann surfaces 
of genus $3$ and $4$. From their list, more concretely Proposition 2.2 (c), we obtain that $(\BZ/2\BZ)^4$ 
cannot be realized as a subgroup of the group of automorphisms of a surface of genus $4$, and thus we obtain
the desired contradiction. 

Finally, suppose that $g=3$. As before, this case boils down to ruling out the possibility of having  $(\BZ/2\BZ)^4$
acting by automorphisms on a Riemann surface of genus $3$. This is established in Proposition 1.2 (c) of \cite{kuku}. 
This concludes the case $g=3$ and thus the proof of the lemma.
\end{proof}

\begin{bem}
One could wonder if in Lemma \ref{torsion-bound} the condition $n\ge 4$ is necessary. 
Indeed it is, because the mapping class group of a surface of genus $3$ contains a subgroup 
isomorphic to $(\BZ/2\BZ)^3$, namely the group $H(8,8)$ in the list in \cite{kuku}. 
\end{bem}

After all this immensely boring work, we are finally ready to prove Proposition \ref{no-torsion}.

\begin{proof}[Proof of Proposition \ref{no-torsion}]
Recall that a cut system in $X$ is a maximal multicurve whose complement in $X$ is connected; observe that every cut system 
consists of $g$ curves and that every non-separating curve is contained in some cut system. 

Given a cut system $\eta$ consider the group $\BT_\eta$ generated by the Dehn twists along the components of $\eta$, 
noting that $\BT_\eta\simeq\BZ^g$. Any permutation of the components of $\eta$ can be realized by a homeomorphism of $X$. 
Consider the normalizer $\CN(\BT_\eta)$ and centralizer $\CZ(\BT_\eta)$ of $\BT_\eta$ in $\Map(X)$. As mentioned in Section
\ref{subsec:multitwists}, we have the following exact sequence:
$$1\to\CZ(\BT_\eta)\to\CN(\BT_\eta)\to\CS_g\to 1,$$
where $\CS_g$ denotes the symmetric group of permutations of the components of $\eta$. Observe that the 
action by conjugation of $\CN(\BT_\eta)$ onto $\BT_\eta$ induces an action
$\CS_g=\CN(\BT_\eta)/\CZ(\BT_\eta)\actson\BT_\eta$ which is conjugate to the standard action of $\CS_g\actson\BZ^g$.
Clearly, this action descends to an action $\CS_g\actson\phi(\BT_\eta)$.

Seeking a contradiction, suppose that the image under $\phi$ of a Dehn twist $\delta_\gamma$ along a 
non-separating curve has finite order. Since all the Dehn twists along the components of $\eta$ are conjugate 
to $\delta_\gamma$ we deduce that all their images have finite order; hence $\phi(\BT_\eta)$ is generated by finite order 
elements. On the other hand, $\phi(\BT_\eta)$ is abelian because it is the image of an abelian group. Being
 abelian and generated by finite order elements, $\phi(\BT_\eta)$ is finite. 

It thus follows from Lemma \ref{torsion-bound} that the subgroup of $\BT_\eta$ generated by Dehn twists along $g-1$
 components of $\eta$ surjects under $\phi$ onto $\phi(\BT_\eta)$. This implies that 
$$\phi(\BT_\eta)=\phi(\BT_{\eta'})$$
whenever $\eta$ and $\eta'$ are cut systems which differ by exactly one component. Now, since 
the cut system complex is connected \cite{Hatcher-Thurston},
 we deduce that $\phi(\delta_\alpha)\in\phi(\BT_\eta)$ for every non-separating curve $\alpha$. Since 
 $\Map(X)$ is generated by Dehn twists along non-separating curves,  we deduce that the image of $\Map(X)$ is the 
abelian group $\phi(\BT_\eta)$. 
By Theorem \ref{thm:homology}, any homomorphism $\Map(X)\to\Map(Y)$ with abelian image is trivial, and thus
we obtain the desired contradiction.
\end{proof}

Before moving on we discuss briefly the proof of Theorem \ref{blablabla}. Suppose that $Y$ is not closed. 
Then, every finite subgroup of $\Map(Y)$ is cyclic by Theorem \ref{thm:boundary-no-torsion}. In particular, 
the bound on the number of generators in Lemma \ref{chinese} implies that {\em if $V\subset\Map(Y)$ is a 
finite abelian group endowed with an action of $\CS_n$ and $\phi:\BZ^n\to V$ is a $\CS_n$-equivariant 
epimomorphism then the restriction of $\phi$ to $\BZ^{n-1}\times\{0\}$ is surjective.} Once this has been 
established, Theorem \ref{blablabla} follows with the same proof, word for word, as Proposition \ref{no-torsion}. \qed

\begin{bem}
Let $X$ and $Y$ be surfaces, where $Y$ has a single boundary component and no cusps. Let $G$ be a finite index subgroup
of $\Map(X)$ and let $\phi:G\to\Map(Y)$ be a homomorphism. A simple modification of a construction due to Breuillard-Mangahas 
\cite{Mangahas} yields a closed surface $Y'$ containing $Y$ and a homomorphism
$$\phi':\Map(X)\to\Map(Y')$$
such that for all $g\in G$ we have, up to isotopy, $\phi'(g)(Y)=Y$ and $\phi'(g)\vert_Y=\phi(g)$. 

Suppose now that $G$ could be chosen so that there is an epimorphism $G\to\BZ$. Assume further that $\phi:G\to\Map(Y)$
factors through this epimorphism and that the image of $\phi$ is purely pseudo-Anosov. Then, every element in the image 
of the extension $\phi':\Map(X)\to\Map(Y)$ either has finite order or is a partial pseudo-Anosov. A result of Bridson 
\cite{Bridson}, stated as Theorem \ref{bridson} below, implies that every Dehn twist in $\Map(X)$ is mapped to a finite 
order element in $\Map(Y)$. Hence, the extension homomorphism $\phi'$ produces a negative answer to Question \ref{question}.

We have hence proved that {\em a positive answer to Question \ref{question} implies that 
every finite index subgroup of $\Map(X)$ has finite abelianization.}
\end{bem}

\section{The map $\phi_*$}
In addition to the triviality results given in Theorems \ref{thm:homology} and \ref{thm:luisito}, the third key ingredient in 
the proof of Theorem \ref{dogs-bollocks} is the following result due to Bridson \cite{Bridson}:

\begin{sat}[Bridson]\label{bridson}
Suppose that $X,Y$ are surfaces of finite type and that $X$ has genus at least $3$. Any homomorphism $\phi:
\Map(X)\to\Map(Y)$ maps roots of multitwists to roots of multitwists.
\end{sat}

\begin{proof}[A remark on the proof of Theorem \ref{bridson}]
In \cite{Bridson}, Theorem \ref{bridson} is proved for surfaces without boundary only. However, Bridson's argument 
remains valid if we allow $X$ to have boundary. That the result can also be extended to the case that $Y$ has
non-empty boundary needs a minimal argument, which we now give. 
Denote by $Y'$ the surface obtained from $Y$ by deleting all boundary components and consider the homomorphism
 $\pi:\Map(Y)\to\Map(Y')$ provided by Theorem \ref{center}. By Bridson's theorem, the image under $\pi\circ\phi$ of 
a Dehn twist $\delta_\gamma$ is a root of a multitwist. Since the kernel of $\pi$ is the group of multitwists along 
the boundary of $Y$, it follows that $\phi(\delta_\gamma)$ is also a root of  a multitwist, as claimed.
\end{proof}

A significant part of the sequel is devoted to proving that under suitable 
assumptions the image of a Dehn twist is in fact a Dehn twist. We highlight the apparent difficulties in the 
following example:

\begin{bei}\label{example2}
Suppose that $X$ has a single boundary component and at least two punctures. By \cite{Grossman}, the mapping 
class group $\Map(X)$ is residually finite. Fix a finite group $G$ and an epimorphism $\pi:\Map(X)\to G$. It
 is easy to construct a connected surface $Y$ on which $G$ acts and which contains $\vert G\vert$ disjoint copies 
$X_g$ ($g\in G$) of $X$ with $gX_h=X_{gh}$ for all $g,h\in G$. Given $x\in X$, denote the corresponding element in 
$X_g$ by $x_g$. If $f:X\to X$ is a homeomorphism fixing pointwise the boundary and punctures, we define
$$\hat f:Y\to Y$$ with
$\hat f(x_g)=(f(x))_{\pi([f])g}$ for $x_g\in X_g$ and $\hat f(y)=\pi([f])(y)$ for $y\notin\cup_{g\in G}X_g$; here 
$[f]$ is the element in $\Map(X)$ represented by $f$.

Notice that $\hat f$ does not fix the marked points of $Y$; in order to by-pass this difficulty, consider $\bar Y$ the 
surface obtained from $Y$ by forgetting all marked points, and consider $\hat f$ to be a self-homeomorphism of $\bar Y$. 
It is easy to see that the map $f\mapsto \hat f$ induces a homomorphism
$$\phi:\Map(X)\to\Map(\bar Y)$$
with some curious properties, namely:
\begin{itemize}
\item If $\gamma\subset X$ is a simple closed curve which bounds a disk with at least two punctures then the image 
$\phi(\delta_\gamma)$ of the Dehn twist $\delta_\gamma$ along $\gamma$ has finite order. Moreover, 
$\delta_\gamma\in\Ker(\phi)$ if and only if $\delta_\gamma\in\Ker(\pi)$.
\item If $\gamma\subset X$ is a non-separating simple closed curve then $\phi(\delta_\gamma)$ has 
infinite order. Moreover, $\phi(\delta_\gamma)$ is a multitwist if $\delta_\gamma\in\Ker(\pi)$; otherwise, 
$\phi(\delta_\gamma)$ is a non-trivial root of a multitwist. Observe that in the latter case, 
$\phi(\delta_\gamma)$ induces a non-trivial permutation of the components of the multicurve supporting any 
of its multitwist powers.
\end{itemize}
This concludes the discussion of Example \ref{example2}.
\end{bei}

While a finite order element is by definition a root of a multitwist,  Proposition \ref{no-torsion} ensures that, 
under suitable bounds on the genera of the surfaces involved, any non-trivial homomorphism 
$\Map(X)\to\Map(Y)$ maps Dehn twists to infinite order elements. From now on we assume that we are in the following
 situation:
\begin{quote}
(*) $X$ and $Y$ are orientable surfaces of finite topological type, of genus $g$ and $g'$ respectively,
 and such that one of the following holds:
\begin{itemize}
\item Either $g\ge 4$ and $g'\le g$, or
\item $g\ge 6$ and $g'\le 2g-1$. 
\end{itemize}
\end{quote}

\begin{bem}
It is worth noticing that the reason for the genus bound $g\ge 6$ in Theorem \ref{dogs-bollocks} is that 
$2^{g-2}-1<2g-1$ if $g<6$.
\end{bem}

Assuming (*), it follows from Proposition \ref{no-torsion} that any non-trivial 
homomorphism $\phi:\Map(X)\to\Map(Y)$ maps Dehn twists $\delta_\gamma$ along non-separating curves $\gamma$ 
to infinite order elements in $\Map(Y)$. Furthermore, it follows from Theorem \ref{bridson} that there is $N$ such 
that $\phi(\delta_\gamma^N)$ is a non-trivial multitwist. We denote by $\phi_*(\gamma)$ the multicurve in $Y$ 
supporting $\phi(\delta_\gamma^N)$, which is independent of the choice of $N$ by Lemma \ref{root1}. 
Notice that two multitwists commute if and only if their supports do not intersect; hence, $\phi_*$ preserves the 
property of having zero intersection number. Moreover, the uniqueness of $\phi_*(\gamma)$ implies that for 
any $f\in\Map(X)$ we have $\phi_*(f(\gamma))=\phi(f)(\phi_*(\gamma))$. Summing up we have:

\begin{kor}\label{map-phi}
Suppose that $X$ and $Y$ are as in (*) and let 
$$\phi:\Map(X)\to\Map(Y)$$ 
be a non-trivial homomorphism. For every non-separating curve $\gamma\subset X$, there is a uniquely 
determined multicurve $\phi_*(\gamma)\subset Y$ with the property that $\phi(\delta_\gamma)$ is a root 
of a generic multitwist in $\BT_{\phi_*(\gamma)}$. Moreover the following holds:
\begin{itemize} 
\item $i(\phi_*(\gamma),\phi_*(\gamma'))=0$ for any two disjoint non-separating curves $\gamma$ and $\gamma'$, and
\item $\phi_*(f(\gamma))=\phi(f)(\phi_*(\gamma))$ for all $f\in\Map(X)$. In particular, the multicurve $\phi_*(\gamma)$ 
is invariant under $\phi(\CZ(\delta_\gamma))$. \qed
\end{itemize}

\end{kor}

The remainder of this section is devoted to give a proof of the following result:

\begin{prop}\label{lem:single-component}
Suppose that $X$ and $Y$ are as in (*); further, assume that $Y$ is not closed if it has genus $2g-1$.
Let $\phi:\Map(X)\to\Map(Y)$ be an irreducible homomorphism. 
Then, for every non-separating curve $\gamma\subset X$ the multicurve $\phi_*(\gamma)$ is 
 a non-separating curve.
\end{prop}

Recall that a homomorphism $\phi:\Map(X)\to\Map(Y)$ 
is irreducible if its image does not fix any curve in $Y$, and that if $\phi$ is irreducible
then $\D Y=\emptyset$; see Definition \ref{def-irred} and the remark following.

Before launching the proof of Proposition \ref{lem:single-component} we will establish a few useful facts.

\begin{lem}\label{induction}
Suppose $X$ and $Y$ satisfy (*) and that $\phi:\Map(X)\to\Map(Y)$ is an irreducible homomorphism. Let $\bar Y$ be 
obtained from $Y$ by filling in some, possibly all, punctures of $Y$, and let 
$\bar\phi=\iota_\#\circ\phi:\Map(X)\to\Map(\bar Y)$ be the composition of $\phi$ with the homomorphism $\iota_\#$ induced by 
the embedding $\iota:Y\to\bar Y$. For every non-separating curve $\gamma\subset X$ we have:
\begin{itemize}
\item $\iota(\phi_*(\gamma))$ is a multicurve, and 
\item $\bar\phi_*(\gamma)=\iota(\phi_*(\gamma))$.
\end{itemize}
In particular, $\iota$ yields a bijection between the components of $\phi_*(\gamma)$ and $\bar\phi_*(\gamma)$.
\end{lem}
\begin{proof}
First observe that, arguing by induction, we may assume that $\bar Y$ is obtained from $Y$ by filling in
 a single cusp. We suppose from now on that this is the case and observe that it follows from Lemma \ref{trick} 
that $\bar\phi$ is not trivial. Notice also that since $Y$ and $\bar Y$ have the same genus, $\bar\phi_*(\gamma)$ 
is well-defined by Corollary \ref{map-phi}. 

By definition of $\phi_*$ and $\bar\phi_*$, we can choose $N\in\BN$ such that $\phi(\delta_\gamma^N)$ and 
$\bar\phi(\delta_\gamma^N)$ are generic multitwists in $\BT_{\phi_*(\gamma)}$ and $\BT_{\bar\phi_*(\gamma)}$. 
In particular, it follows from Lemma \ref{BHB} that in order to prove  Lemma \ref{induction} it suffices to show 
that $\iota(\phi_*(\gamma))$  does not contain (1) inessential components, or (2) parallel components. 
\medskip

\noindent{\bf Claim 1.} $\iota(\phi_*(\gamma))$ does not contain inessential components.

\begin{proof}[Proof of Claim 1]
Seeking a contradiction, suppose that a component $\eta$ of $\phi_*(\gamma)$ is inessential 
in $\bar Y$. Since $\bar Y$ is obtained from $Y$ by filling in a single cusp, it follows that $\eta$ bounds a disk 
in $Y$ with exactly two punctures. Observe that this implies that for any element $F\in\Map(Y)$ we have either 
$F(\eta)=\eta$ or $i(F(\eta),\eta)>0$. On the other hand, if $f\in\Map(X)$ is such that $i(f(\gamma),\gamma)=0$ then we have
$$i(\phi(f)(\eta),\eta)\le i(\phi(f)(\phi_*(\gamma)),\phi_*(\gamma))= i(\phi_*(f(\gamma)),\phi_*(\gamma))=0$$
We deduce that $\eta=\phi(f)(\eta)\subset\phi_*(f(\gamma))$ for any such $f$. Since any two non-separating curves 
in $X$ are related by an element of $\Map(X)$ we obtain: 
\begin{quote}
($\star$) {\em If $\gamma'$ is a non-separating curve in $X$ with $i(\gamma,\gamma')=0$ then 
$\eta=\phi(\delta_{\gamma'})(\eta)$ and $\eta\subset\phi_*(\gamma')$.}
\end{quote}
Choose $\gamma'\subset X$ so that $X\setminus(\gamma\cup\gamma')$ is connected. It follows from ($\star$) that 
if  $\gamma''$ is any other non-separating curve which either does not intersect $\gamma$ or $\gamma'$ we have
 $\phi(\delta_{\gamma''})(\eta)=\eta$. Since the mapping class group is generated by such curves, we deduce that every
 element in $\phi(\Map(X))$ fixes $\eta$, contradicting the assumption that $\phi$ is irreducible. 
This concludes the proof of Claim 1.
\end{proof}

We use a similar argument to prove that $\iota(\phi_*(\gamma))$ does not contain parallel components.

\medskip

\noindent{\bf Claim 2.} $\iota(\phi_*(\gamma))$ does not contain parallel components.

\begin{proof}[Proof of Claim 2]
Seeking again a contradiction suppose that there are $\eta\neq\eta'\subset\phi_*(\gamma)$ whose images in $\bar Y$ are 
parallel. Hence, $\eta\cup\eta'$ bounds an annulus which contains a single cusp. As above, it follows that for any 
element $f\in\Map(Y)$ we have either $f(\eta\cup\eta')=\eta\cup\eta'$ or $i(f(\eta),\eta)>0$. By the same argument as
before, we obtain that $\phi(\Map(X))$ preserves $\eta\cup\eta'$. Now, it follows from either Theorem 
\ref{thm:homology} or Theorem \ref{thm:luisito} that $\phi(\Map(X))$ cannot permute $\eta$ and $\eta'$. Hence 
$\phi(\Map(X))$ fixes $\eta$, contradicting the assumption that $\phi$ is irreducible.
\end{proof}

As we mentioned above, Lemma \ref{induction} follows from Claim 1, Claim 2 and Lemma \ref{BHB}.
\end{proof}

Continuing with the preliminary considerations to prove Proposition \ref{lem:single-component}, recall that the 
final claim in Corollary \ref{map-phi} implies that $\phi(\delta_\gamma)$ preserves the multicurve $\phi_*(\gamma)$. 
Our next goal is to show that, as long as $\phi$ is irreducible, the element $\phi(\delta_\gamma)$ preserves  every component of $\phi_*(\gamma)$. 

\begin{lem}\label{prop:curves-invariant}
Suppose that $X$ and $Y$ are as in (*) and let $\phi:\Map(X)\to\Map(Y)$ be an irreducible homomorphism. 
If $\gamma\subset X$ is a non-separating simple closed curve, then $\phi(\CZ_0(\delta_\gamma))$ 
fixes every component of $\phi_*(\gamma)$. Hence, $\phi(\CZ_0(\delta_\gamma))\subset\CZ_0(\BT_{\phi_*(\gamma)})$.
\end{lem}

Recall that $\CZ_0(\delta_\gamma)$ is the subgroup of $\Map(X)$ fixing not only $\gamma$ but also the two sides of 
$\gamma$ and that it has at most index 2 in the centralizer $\CZ(\delta_\gamma)$ of the Dehn twist $\delta_\gamma$.

\begin{proof}
We first prove Lemma \ref{prop:curves-invariant} in the case that $Y$ is closed. As in Section \ref{sec:general}, 
we denote by $X_\gamma$ the surface obtained by deleting the interior of a closed regular neighborhood of 
$\gamma$ from $X$. Recall that by \eqref{eq:central1} there is a surjective homomorphism
$$\Map(X_\gamma)\to \CZ_0(\delta_\gamma)$$
Consider the composition of this homomorphism with $\phi$ and, abusing notation, denote its image by 
$\phi(\Map(X_\gamma))=\phi(\CZ_0(\delta_\gamma))$.

By Corollary \ref{map-phi}, the subgroup $\phi(\Map(X_\gamma))$ of $\Map(Y)$ acts on the set of 
components of $\phi_*(\gamma)$ and hence on $Y\setminus\phi_*(\gamma)$. Since $Y$ is assumed to be 
closed and of at most genus $2g-1$ we deduce that $Y\setminus\phi_*(\gamma)$ has at most $\vert\chi(Y)\vert=2g'-2\le 4g-4$ 
components. Since the surface $X_\gamma$ has genus $g-1\ge 3$, we deduce from Theorem \ref{thm:luisito} that 
$\phi(\Map(X_\gamma))$ fixes each component of $Y\setminus\phi_*(\gamma)$.

Suppose now that $Z$ is a component of $Y\setminus\phi_*(\gamma)$ and let $\eta$ be the set of components of 
$\phi_*(\gamma)$ contained in the closure of $Z$. Noticing that 
$$4-4g\le\chi(Y)\le\chi(Z)\le-\vert\eta\vert+2$$ 
we obtain that $\eta$ consists of at most $4g-2$ components. Since $\phi(\Map(X_\gamma))$ fixes $Z$, 
it acts on the set of components of $\eta$. Again by Theorem \ref{thm:luisito},
it follows that this action is trivial, meaning that every component of $\phi_*(\gamma)$ contained in the closure 
of $Z$ is preserved. Since $Z$ was arbitrary, we deduce that $\phi(\Map(X_\gamma)$ preserves every component of 
$\phi_*(\gamma)$ as claimed. Lemma \ref{trick2} now implies that 
$\phi(\CZ_0(\delta_\gamma))=\phi(\Map(X_\gamma))\subset\CZ_0(\BT_{\phi_*(\gamma)})$. This 
concludes the proof of Lemma \ref{prop:curves-invariant} in the case that $Y$ is closed.
\medskip

We now turn our attention to the general case. Recall that the assumption that $\phi$ is irreducible implies that $\D Y=\emptyset$. 
Let $\bar Y$ be the surface obtained from $Y$ by closing up all the cusps and denote by $\bar\phi:\Map(X)\to\Map(\bar Y)$ 
the composition of $\phi$ with the homomorphism $\iota_\#:\Map(Y)\to\Map(\bar Y)$ induced by the embedding 
$\iota:Y\to\bar Y$. By the above, Lemma \ref{prop:curves-invariant} holds true for $\bar\phi$. On the other hand, 
Lemma \ref{induction} shows that for any $\gamma\subset X$ non-separating there is a bijection between $\phi_*(\gamma)$ 
and $\bar\phi_*(\gamma)$. Thus the claim follows.
\end{proof}

Note that Lemma \ref{prop:curves-invariant} yields the following sufficient 
condition for a homomorphism between mapping class groups to be reducible:

\begin{kor}\label{kor:irreducible}
Suppose that $X$ and $Y$ are as in (*) and let $\phi:\Map(X)\to\Map(Y)$ be a non-trivial homomorphism. 
Let $\gamma$ and $\gamma'$ be distinct, disjoint curves on $X$ such that $X\setminus(\gamma\cup\gamma')$ is connected. 
If the multicurves $\phi_*(\gamma)$ and $\phi_*(\gamma')$ share a component, then the homomorphism 
$\phi:\Map(X)\to\Map(Y)$ is reducible.
\end{kor}
\begin{proof}
Suppose that $\phi$ is irreducible and observe that $\Map(X)$ is 
generated by Dehn twists along curves $\alpha$ which are disjoint from $\gamma$ or $\gamma'$. For any such $\alpha$ we 
have $\delta_\alpha\in\CZ_0(\delta_\gamma)\cup\CZ_0(\delta_{\gamma'})$. In particular, it follows from 
Proposition \ref{prop:curves-invariant} that $\phi(\Map(X))$ fixes every component of 
$\phi_*(\gamma)\cap\phi_*(\gamma')$. The assumption that $\phi$ was irreducible implies that 
$\phi_*(\gamma)\cap\phi_*(\gamma')=\emptyset$.
\end{proof}

We are now ready to prove Proposition \ref{lem:single-component}:

\begin{proof}[Proof of Proposition \ref{lem:single-component}]
Let $\gamma$ be a non-separating curve on $X$. Extend $\gamma$ to a multicurve $\eta\subset X$ with $3g-3$ 
components $\gamma_1,\dots,\gamma_{3g-3}$, 
and such that the surface $X\setminus(\gamma_i\cup\gamma_j)$ is connected for all $i,j$. Since $\delta_{\gamma_i}$ 
and $\delta_{\gamma_j}$ are conjugate in $\Map(X)$ we deduce that $\phi_*(\gamma_i)$ and $\phi_*(\gamma_j)$ 
have the same number $K$ of components for all $i,j$. Since $\phi$ is 
irreducible, Corollary \ref{kor:irreducible} implies that $\phi_*(\gamma_i)$ and $\phi_*(\gamma_j)$ do not share
any components for all 
$i\neq j$. This shows that $\cup_i\phi_*(\gamma_i)$ is the union of $(3g-3)K$ distinct curves. Furthermore, since 
$\delta_{\gamma_i}$ and $\delta_{\gamma_j}$ commute, we deduce that $\cup_i\phi_*(\gamma_i)$ is a 
multicurve in $Y$.

Suppose first that $Y$ has genus $g' \le2g-2$. In the light of Lemma \ref{induction}, 
it suffices to consider the case that $Y$ is closed. Now, the multicurve $\cup_i\phi_*(\gamma_i)$ has at most 
$3g'-3 \leq 3(2g-2) - 3< 6g-6$ components. Hence: $$K<\frac{6g-6}{3g-3}\le2,$$
and thus the multicurve $\phi_*(\gamma)$ consists of $K=1$ components; in other words, it is a curve. It is non-separating 
because otherwise the multicurve $\cup_i\phi_*(\gamma_i)$ would consist of $3g-3$ separating curves, and a closed 
surface of genus $g'\le 2g-2$ contains at most $2g-3$ disjoint separating curves. This concludes the
proof of the proposition in the case that $Y$ has genus at most $2g-2$.

\medskip

Suppose now that $Y$ has genus $g'=2g-1$ and that $Y$ is not closed. Again by Lemma \ref{induction},
we can assume that $Y$ has a single puncture, which we consider as a marked point. 
In this case, the multicurve $\cup_i\phi_*(\gamma_i)$ consists of at most 
$3g'-2=6g-5$ curves. Since we know that $\cup_i\phi_*(\gamma_i)$ is the union of $(3g-3)K$ distinct curves,
we deduce $K\leq 2$. In the case that $\cup_i\phi_*(\gamma_i)$ has fewer than $6g-6$ components,
we proceed as before. Therefore, it remains to rule out the possibility of having exactly $6g-6$ components.  

Suppose, for contradiction, that  $\cup_i\phi_*(\gamma_i)$ has $6g-6$ components. 
Since $Y$ has genus $2g-1$ and exactly one marked point, the complement of $\cup_i\phi_*(\gamma_i)$ in $Y$ is a 
disjoint union of pairs of pants, where exactly one of them, call it $P$, contains the marked point of $Y$. 
Now, the boundary components of $P$ are contained in the image under $\phi_*$ of  curves 
$a_1, a_2, a_3 \in \{\gamma_1,\dots,\gamma_{3g-3}\}$. Assume, for the sake of concreteness, that
$a_i\neq a_j$ whenever $i\neq j$; the remaining case is dealt with using minor modifications of the
argument we give here. 

Suppose first that the multicurve $\alpha=a_1\cup a_2\cup a_3$ does not disconnect $X$ and let $\alpha'\neq\alpha$ be another multicurve with three components satisfying:
\begin{enumerate}
\item $X\setminus\alpha'$ is connected,
\item $i(\alpha,\alpha') = 0$, and
\item $X\setminus(\gamma\cup\gamma')$ is connected for all $\gamma,\gamma'\in\alpha\cup\alpha'$.
\end{enumerate}
Notice that since $X\setminus\alpha$ and $X\setminus\alpha'$ are homeomorphic, there is $f\in\Map(X)$ with 
$f(\alpha)=\alpha'$.
Now, $P'=\phi(f)(P)$ is a pair of pants which contains the marked point of $Y$. Taking into account that 
$\D P\subset\phi_*(\alpha)$ and $\D P'\subset\phi_*(\alpha')$ we deduce from (2) that $i(\D P,\D P')=\emptyset$ 
and hence that $P=P'$. Since $\alpha'\neq\alpha$ we may assume, up to renaming, that $a_1\not\subset\alpha'$. 
Since $\phi(f)(\D P)=\D P'$ and $\D P\cap\phi_*(a_1)\neq\emptyset$, we deduce that is $i$ such that  
$\phi_*(a_i)\cap\phi_*(f(a_1))$ contains a boundary curve of $P$. In the light of (3), it follows from Corollary 
\ref{kor:irreducible} that $\phi$ is reducible; this contradiction shows that $X\setminus\alpha$ cannot be connected.

If $X\setminus\alpha$ is not connected, then it has two components, as $X\setminus(a_1\cup a_2)$ is connected. Suppose first that neither of the two components $Z_1,Z_2$ of $X\setminus\alpha$ is a (possibly punctured) pair of pants and notice that this implies that $Z_1$ and $Z_2$ both have positive genus. Let $P_1\subset Z_1$ be an unpunctured pair of pants with boundary $\D P_1=a_1\cup a_2\cup a_3'$ and let $P_2\subset Z_2$ be second unpunctured a pair of pants with $Z_2\setminus P_2$ connected and with boundary $\D P_2=a_3\cup a_1'\cup a_2'$ where $a_1'$ and $a_2'$ are not boundary parallel in $Z_2$; compare with Figure \ref{fig-que-pesadez}. 
\begin{figure}[tbh] \unitlength=1in
\begin{center} 
\includegraphics[width=1in]{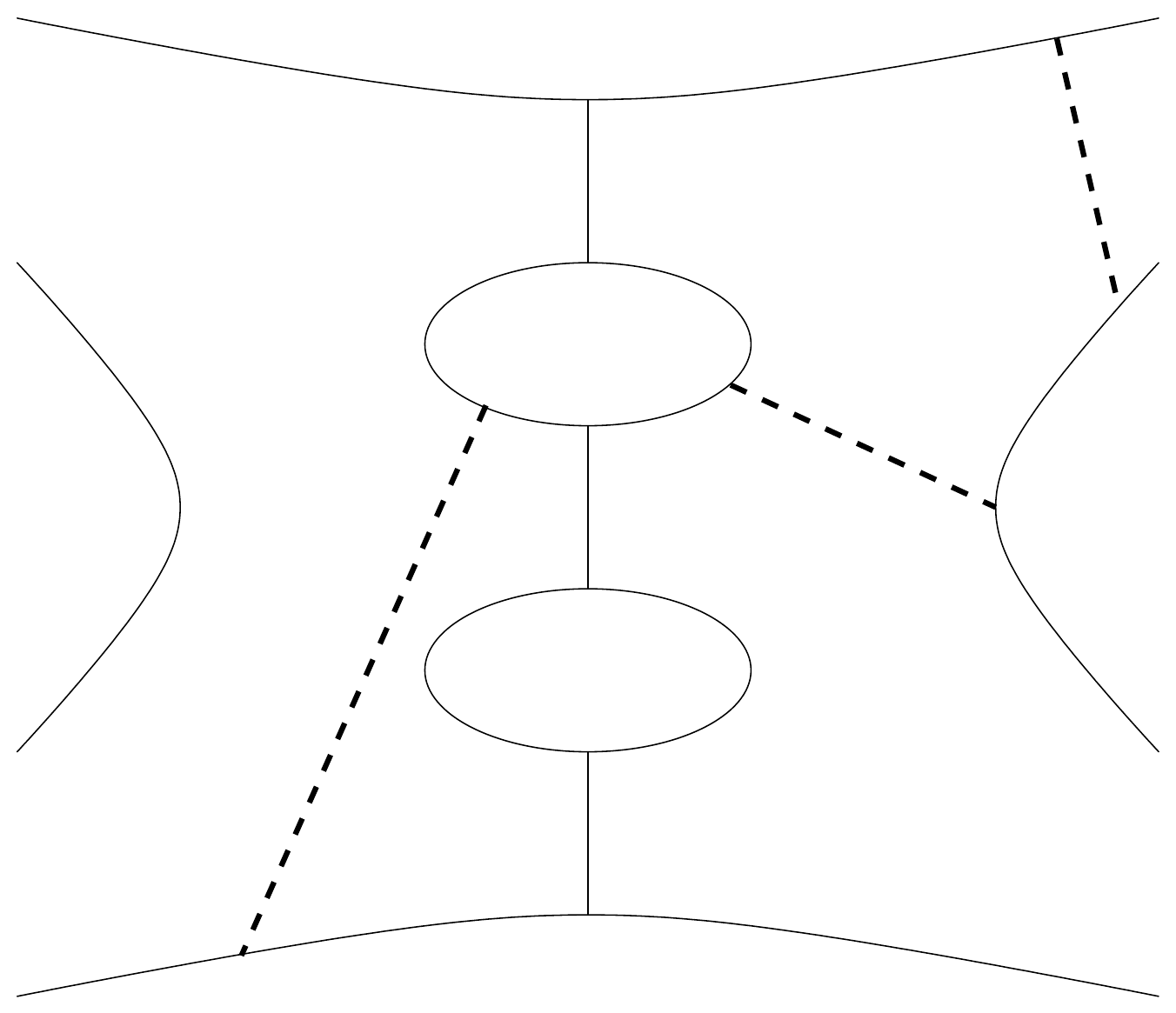}
\end{center}
\caption{}\label{fig-que-pesadez}
\end{figure}
Notice that $Z_1'=(Z_1\cup P_2)\setminus P_1$ is homeomorphic to $Z_1$. Similarly, $Z_2'=(Z_2\cup P_1)\setminus P_2$ is homeomorphic to $Z_2$. Finally notice also that $Z_i'$ contains the same punctures as $Z_i$ for $i=1,2$. It follows from the classification theorem of surfaces that there is $f\in\Map(X)$ with $f(Z_1)=Z_1'$ and $f(Z_2)=Z_2'$. In particular, $f(\alpha)=\alpha'$ where $\alpha'=a_1'\cup a_2'\cup a_3'$. We highlight a few facts:
\begin{enumerate}
\item There is $f\in\Map(X)$ with $f(\alpha)=\alpha'$,
\item $i(\alpha,\alpha') = 0$, and
\item $X\setminus(\gamma\cup\gamma')$ is connected for all $\gamma,\gamma'\in\{a_1,a_2,a_1',a_2'\}$.
\end{enumerate}
As above, we deduce that $\phi(f)(\D P)=\D P'$ and that for all $i=1,2,3$ there is $j$
such that $\phi_*(a_i)\cap\phi_*(f(a_j))$ contains a boundary curve of $P$. In the light of (3), 
it follows again from Corollary \ref{kor:irreducible} that $\phi$ is reducible. We have reduced to 
the case that one of the components of $X\setminus\alpha$, say $Z_1$, is a (possibly punctured) pair of pants.

We now explain how to reduce to the case that $Z_1$ is a pair of pants without punctures. Let 
$a_3'\subset Z_1$ be a curve which, together with $a_3$, bounds an annulus $A\subset Z_1$ such that $Z_1\setminus A$ 
does not contain any marked points. Notice that we may assume without loss of generality that the multicurve 
$\gamma_1\cup\dots\cup\gamma_{3g-3}$ above does not intersect $a_3'$. It follows that 
$i(\phi_*(a_3'),\cup\phi_*(\gamma_i))=0$. Next, observe that a pants decomposition of $Y$ consists of
$3(2g-1)-3+1 = 6g-5$ curves. Since $\phi_*(a_3')$ has two components and $\cup\phi_*(\gamma_i)$
has $6g-6$ components, we deduce that there exists $i$ such that $\phi_*(a_3')$ and $\phi_*(\gamma_i)$ share a
component. If $i \neq 3$, property (3) and Corollary \ref{kor:irreducible} imply that $\phi$ is reducible, since
$a_3'\cup\gamma_i$ does not separate $X$. It thus follows that $\phi_*(a_3')$ and $\phi_*(a_3)$ share a component, and so
$\D P\subset\phi_*(a_1\cup a_2\cup a_3')$.


Summing up, it remains to rule out the possibility that $Z_1$ is a pair of pants without punctures. Choose $\alpha'\subset X$ satisfying:
\begin{enumerate}
\item $\alpha'$ bounds a pair of pants in $X$,
\item $i(\alpha,\alpha') = 0$, and
\item $X\setminus(\gamma\cup\gamma')$ is connected for all $\gamma,\gamma'\in\alpha\cup\alpha'$.
\end{enumerate}
Now there is $f\in\Map(X)$ with $f(\alpha)=\alpha'$ and we can repeat word by word the argument given in the case that $X\setminus\alpha$ was connected. 

After having ruled out all possibilities, we deduce that $\cup_i\phi_*(\gamma_i)$ cannot have $6g-6$ components. This concludes the proof of Proposition \ref{lem:single-component}.
\end{proof}

\section{Proof of Proposition \ref{blabla}}
In this section we show that every homomorphism $\Map(X) \to \Map(Y)$ is trivial if the genus of $X$ is larger
than that of $Y$.  As a consequence we obtain that, under suitable genus bounds, 
the centralizer of the image of a non-trivial homomorphism between mapping class groups is torsion-free.

\begin{named}{Proposition \ref{blabla}}
Suppose that $X$ and $Y$ are orientable surfaces of finite topological type. If the 
genus of $X$ is at least 3 and larger than that of $Y$, then every homomorphism $\phi:\Map(X)\to\Map(Y)$
is trivial.
\end{named}

Recall that Proposition \ref{blabla} is due to Harvey-Korkmaz \cite{Harvey-Korkmaz} in the case that both surfaces $X$ 
and $Y$ are closed. 

\begin{proof}
We will proceed by induction on the genus of $X$. Notice that Proposition \ref{genus2} establishes the base case of the 
induction and observe that by Lemma \ref{trick} we may assume that $Y$ is has empty boundary and no cusps.

Suppose now that $X$ has genus $g\ge 4$ and that we have proved Proposition \ref{blabla} for 
surfaces of genus $g-1$. Our first step is to prove the following: 

\medskip

\noindent {\bf Claim.} {\em Under the hypotheses above, every homomorphism $\Map(X)\to\Map(Y)$ is reducible.}
\begin{proof}[Proof of the claim]
Seeking a contradiction, suppose that there is an irreducible homomorphism $\phi:\Map(X)\to\Map(Y)$, where $Y$ has smaller 
genus than $X$. Let $\gamma\subset X$ be a non-separating curve. Observing that $X$ and $Y$ satisfy (*), we deduce 
that $\phi_*(\gamma)$ is a non-separating curve by Proposition \ref{lem:single-component} and that 
$\phi(\CZ_0(\delta_\gamma))\subset\CZ_0(\delta_{\phi_*(\gamma)})$ by Lemma \ref{prop:curves-invariant}. 
By \eqref{eq:central2}, $\CZ_0(\delta_{\phi_*(\gamma)})$ dominates $\Map(Y_{\phi_*(\gamma)}')$ where 
$Y_{\phi_*(\gamma)}'=Y\setminus\phi_*(\gamma)$. On the other hand, we have by \eqref{eq:central1} that
 $\CZ_0(\delta_\gamma)$ is dominated by the group $\Map(X_\gamma)$ where $X_\gamma$ is obtained from $X$ 
by deleting the interior of a closed regular neighborhood of $\gamma$. 

Since $\phi_*(\gamma)$ is non-separating, the genus of $Y_{\phi_*(\gamma)}'$ and $X_\gamma$ is one less than that of
$Y$ and $X$, respectively. The induction assumption implies that the induced homomorphism
$$\Map(X_\gamma)\to\Map(Y_{\phi_*(\gamma)}')$$
is trivial. The last claim in Lemma \ref{trick2} proves that the homomorphism 
$$\Map(X_\gamma)\to \CZ_0(\delta_{\phi_*(\gamma)})\subset\Map(Y)$$
is also trivial. We have proved that $\CZ_0(\delta_\gamma)\subset\Ker(\phi)$. Since $Z_0(\delta_\gamma)$ 
contains a Dehn twist along a non-separating curve, we deduce that $\phi$ is trivial from Lemma 
\ref{dehn-trivial-trivial}. This contradiction concludes the proof of the claim.
\end{proof}

Continuing with the proof of the induction step in Proposition \ref{blabla}, suppose there exists a non-trivial
homomorphism $\phi:\Map(X)\to\Map(Y)$. By the above claim, $\phi$ is reducible. 
Let $\eta\subset Y$ be a maximal multicurve in $Y$ which is componentwise fixed by $\phi(\Map(X))$, 
and notice that  $\phi(\Map(X))\subset\CZ_0(\BT_\eta)$ by Lemma \ref{trick2}. Consider  $$\phi':\Map(X)\to\Map(Y'_\eta),$$
the composition of $\phi$ with the homomorphism \eqref{eq:central3}. The maximality of the multicurve 
$\eta$ implies that $\phi'$ is irreducible. Since the genus of $Y'_\eta$ is at most equal to that of $Y$,
 we deduce from the claim above that $\phi'$ is trivial. Lemma \ref{trick2} implies hence that $\phi$ is 
trivial as well. This establishes Proposition \ref{blabla}
\end{proof}

As we mentioned before, a consequence of Proposition \ref{blabla} is that, under suitable assumptions, 
the centralizer of the image  of a homomorphism between mapping class groups is torsion-free. Namely, we have:

\begin{lem}\label{trivial-centralizer1}
Let $X$ and $Y$ be surfaces of finite topological type, where $X$ has genus $g \ge 3$ and
$Y$ has genus $g' \le 2g$. Suppose moreover that $Y$ has at least one (resp. three) marked points if $g'=2g-1$ (resp. $g'=2g$). If $\phi:\Map(X) \to \Map(Y)$ is a 
non-trivial homomorphism, then the centralizer of $\phi(\Map(X))$ in $\Map(Y)$ is torsion-free.
\end{lem}

The proof of Lemma \ref{trivial-centralizer1} relies on Proposition \ref{blabla} and the following consequence of 
the Riemann-Hurwitz formula:

\begin{lem}\label{fix-bound}
Let $Y$ be a surface of genus $g'\ge 0$ and let $\tau:Y\to Y$ be a nontrivial diffeomorphism of prime order,  
representing an element in $\Map(Y)$. Then $\tau$ has $F\le 2g'+2$ fixed-points and the underlying 
surface of the orbifold $Y/\langle\tau\rangle$ has genus at most $\bar g=\frac{2g'+2-F}4$.
\end{lem}

\begin{proof}
Consider the orbifold $Y/\langle\tau\rangle$ and let $F$ be the number of its singular points; observe that $F$ is also equal
to the number of fixed points of $\tau$ since $\tau$ has prime order $p$.
Denote by $\vert Y/\langle\tau\rangle\vert$ the underlying surface of the orbifold $Y/\langle\tau\rangle$. 
The Riemann-Hurwitz formula shows that
\begin{equation}\label{RH-formula}
2-2g'=\chi(Y)=p\cdot \chi(\vert Y/\langle\tau\rangle\vert)-(p-1)\cdot F
\end{equation}
After some manipulations, \eqref{RH-formula} shows that 
$$F=\frac {2g'-2+p\cdot (2-2\bar g)}{p-1}$$
where $\bar g$ is the genus of $\vert Y/\langle\tau\rangle\vert$. Clearly, 
the quantity on the right is maximal if $\bar g=0$ and $p=2$. This implies that $F\le 2g'+2$, as claimed.

Rearranging \eqref{RH-formula}, we obtain 
$$\bar g=\frac{2g'+(2-F)(p-1)}{2p}$$
Again this is maximal if $p$ is as small as possible, i.e. $p=2$. Hence $\bar g\le\frac{2g'+2-F}4$.
\end{proof}

We are now ready to prove Lemma \ref{trivial-centralizer1}.

\begin{proof}[Proof of Lemma \ref{trivial-centralizer1}]
First, if $Y$ has non-empty boundary there is nothing to prove, for in this case $\Map(Y)$ is torsion-free. 
Therefore, assume that $\partial Y = \emptyset$. Suppose, for contradiction, that there exists $[\tau]\in\Map(Y)$ non-trivial, of finite 
order, and such that $\phi(\Map(X)) \subset \CZ([\tau])$.

Let $\tau:Y\to Y$ be a finite order diffeomorphism representing $[\tau]$. Passing to a suitable power, 
we may assume that the order of $\tau$ is prime. Consider the orbifold $Y/\langle\tau\rangle$ as a surface with the singular points marked, and recall 
that by Theorem \ref{Birman-Hilden} we have the following exact sequence:
$$\xymatrix{
1 \ar[r] & \langle[\tau]\rangle \ar[r] & \CZ([\tau])\ar[r]^-{\beta} & \Map^*(Y/\langle\tau\rangle)
}$$
On the other hand, we have by definition
$$1\to\Map(Y/\langle\tau\rangle)\to\Map^*(Y/\langle\tau\rangle)\to\CS_F\to 1$$
where $F$ is the number of punctures of $Y/\langle\tau\rangle$. Again, $F$ is equal to the number of fixed 
points of $\tau$ since $\tau$ has prime order. 

 Observe that Lemma \ref{fix-bound} gives that $F\le 2g'+2\le 4g+2$; hence, it follows from Theorem \ref{thm:luisito} 
that the composition of the homomorphism
$$\xymatrix{
\Map(X) \ar[r]^\phi & \CZ([\tau])  \ar[r]^-\beta &  \Map^*(Y/\langle\tau\rangle)\ar[r] &\CS_F
}$$
is trivial; in other words, $(\beta\circ\phi)(\Map(X))\subset\Map(Y/\langle\tau\rangle)$. 

Our assumptions on the genus and the 
marked points of $Y$ imply, by the genus bound in Lemma \ref{fix-bound}, that $Y/\langle\tau\rangle$ 
has genus less than $g$. Hence, the homomorphism $\beta\circ\phi:\Map(X)\to\Map(Y/\langle\tau\rangle)$ 
is trivial by Proposition \ref{blabla}. This implies that the image of $\phi$ is contained in the abelian group 
$\langle[\tau]\rangle$. Theorem \ref{thm:homology} shows hence that $\phi$ is trivial, contradicting our assumption. 
This concludes the proof of Lemma \ref{trivial-centralizer1}
\end{proof}

The following example shows that Lemma \ref{trivial-centralizer1} is no longer true if $Y$ is allowed to
 have genus $2g$ and fewer than 3 punctures. 
 
\begin{bei}\label{me-ha-saltao-los-dientes}
Let $X$ be a surface with no punctures and such that $\D X=\BS^1$. Let $Z$ be a surface of the same genus as $X$, with $\D Z=\emptyset$ but with two punctures. 
Regard $X$ as a subsurface of $Z$ and consider the two-to-one cover $Y\to Z$ corresponding to an arc in $Z\setminus X$ 
joining the two punctures of $Z$. Every homomorphism $X\to X$ fixing poinwise the boundary extends to a homeomorphism of 
$Z$ fixing the punctures and which lifts to a unique homeomorphism $Y\to Y$ which fixes the two components of the preimage 
of $X$ under the covering $Y\to Z$. The image of the induced homomorphism $\Map(X)\to\Map(Y)$ is centralized by the 
involution $\tau$ associated to the two-to-one cover $Y\to Z$. Moreover, if $X$ has genus $g$ then $Y$ has genus $2g$ 
and 2 punctures.
\end{bei}

\section{Proof of Proposition \ref{main}}
We are now ready to prove that under suitable genus bounds, homomorphisms between mapping class 
groups map Dehn twists to Dehn twists. Namely:

\begin{named}{Proposition \ref{main}}
Suppose that $X$ and $Y$ are surfaces of finite topological type, of genus $g\ge 6$ and $g'\le 2g-1$ 
respectively; if $Y$ has genus $2g-1$, suppose also that it is not closed. 
Every nontrivial homomorphism 
$$\phi:\Map(X)\to\Map(Y)$$
maps (right) Dehn twists along non-separating curves to (possibly left) Dehn twists along non-separating curves. 
\end{named}

\begin{bem}
The reader should notice that the proof of Proposition \ref{main} applies, word for word, to homomorphisms between 
mapping class groups of surfaces of the same genus $g \in \{4,5\}$.
\end{bem}

We will first prove Proposition \ref{main} under the assumption that $\phi$ is irreducible and then we will
deduce the general case from here.

\begin{proof}[Proof of Proposition \ref{main} for irreducible $\phi$] 
Suppose that $\phi$ is irreducible and recall that this implies that $\D Y=\emptyset$. Let $\gamma \subset X$
be a non-separating curve. Thus $\phi_*(\gamma)$ is also a non-separating curve, by Proposition 
\ref{lem:single-component}. We first show that $\phi(\delta_\gamma)$ is a power
of $\delta_{\phi_*(\gamma)}$.

Let $X_\gamma$ be the complement in $X$ of the interior of 
a closed regular neighborhood of $\gamma$ and $Y_{\phi_*(\gamma)}'=Y\setminus\phi_*(\gamma)$ the connected
 surface obtained from $Y$ by removing $\phi_*(\gamma)$. We have that:
\begin{itemize}
\item[$(\star)$]$X_\gamma$ and $Y_{\phi_*(\gamma)}'$ have genus $g-1\ge 3$ and $g'-1\le 2g-2$ respectively. Moreover, observe that $Y_{\phi_*(\gamma)}'$ has two more punctures than $Y$; in particular, $Y_\gamma'$ has at least 3 punctures if it has genus $2g-2$.
\end{itemize}
By \eqref{eq:central1} and \eqref{eq:central2} we have epimorphisms
$$\Map(X_\gamma)\to\CZ_0(\phi(\delta_\gamma))\ \hspace{0.2cm} \text{ and } \hspace{0.2cm} 
\ \CZ_0(\delta_{\phi_*(\gamma)})\to\Map(Y_{\phi_*(\gamma)}').$$
In addition, we know that $\phi(\CZ_0(\delta_\gamma))\subset \CZ_0(\delta_{\phi_*(\gamma)})$ by Lemma 
\ref{prop:curves-invariant}. Composing all these homomorphisms we get a homomorphism 
$$\phi':\Map(X_\gamma)\to\Map(Y_{\phi_*(\gamma)}')$$
It follows from Lemma \ref{dehn-trivial-trivial} that the restriction of $\phi$ to $\CZ_0(\delta_\gamma)$ is not trivial because the latter contains a Dehn 
twist along a non-separating curve; Lemma \ref{trick2} implies that $\phi'$ is not trivial either.

Since $\delta_\gamma$ centralizes $\CZ_0(\delta_\gamma)$, it follows that 
$\phi'(\delta_\gamma)\in\Map(Y_{\phi_*(\gamma)}')$ centralizes the image of $\phi'$. 
Now, the definition of $\phi_*(\gamma)$ implies that some power of 
$\phi(\delta_\gamma)$ is a power of the Dehn twist $\delta_{\phi_*(\gamma)}$. Hence, the first claim 
of Lemma \ref{root3} yields that $\phi'(\delta_\gamma)$ has finite order, and thus
$\phi'(\delta_\gamma)\in\Map(Y_{\phi_*(\gamma)}')$ is a finite order element centralizing 
$\phi(\Map(X_\gamma))$. By ($\star$), Lemma \ref{trivial-centralizer1} applies and shows that 
$\phi'(\delta_\gamma)$ is in fact trivial. The final claim of Lemma \ref{root3} now shows that
 $\phi(\delta_{\gamma})$ is a power of $\delta_{\phi_*(\gamma)}$; in other words, there exists 
$N\in\BZ\setminus\{0\}$ such that $\phi(\delta_\gamma)=\delta_{\phi_*(\gamma)}^N$. 

\medskip

It remains to prove  that $N=\pm 1$. Notice that $N$ does not depend on the particular non-separating curve $\gamma$ since 
any two Dehn twists along non-separating curves are conjugate. Consider a collection $\gamma_1, \ldots, \gamma_n$
of non-separating curves on $X$, with $\gamma = \gamma_1$, such that the Dehn twists
$\delta_{\gamma_i}$ generate $\Map(X)$ and $i(\gamma_i, \gamma_j)\le 1$ for all $i,j$ (compare with Figure \ref{fig1}).
Observe that the $N$-th powers of the Dehn twists along the curves 
$\{\phi_*(\gamma_i)\}$ generate $\phi(\Map(X))$. It follows hence from the assumption that $\phi$ is irreducible that the curves
$\{\phi_*(\gamma_i)\}$ fill $Y$ (compare with the proof of Lemma \ref{same-surface-0} below). Thus, since $\phi_*$ preserves disjointness by Corollary \ref{map-phi}, there exists
$\gamma' \in \{\gamma_1, \ldots, \gamma_n\}$ such that $i(\gamma,\gamma')=1$ 
and $i(\phi_*(\gamma),\phi_*(\gamma'))\ge 1$.
Since $i(\gamma,\gamma')=1$, the Dehn twists
$\delta_\gamma$ and $\delta_{\gamma'}$ braid. Thus, the $N$-th powers 
$\delta_{\phi_*(\gamma)}^N=\phi(\delta_\gamma)$ and $\delta_{\phi_*(\gamma')}^N=\phi(\delta_{\gamma'})$ of 
the Dehn twists along $\phi_*(\gamma)$ and $\phi_*(\gamma')$ also braid. Since $i(\phi_*(\gamma),\phi_*(\gamma'))\ge 1$, 
Lemma \ref{inter-1} shows that $i(\phi_*(\gamma),\phi_*(\gamma'))=1$ and $N=\pm 1$, as desired. 
\end{proof}

Before moving on, we remark that in final argument of the proof of the irreducible case of Theorem 
\ref{main} we have proved the first claim of the following lemma:

\begin{lem}\label{intersection-number-0}
Suppose that $X$, $Y$ are as in the statement of 
Proposition \ref{main}, and let and $\phi:\Map(X)\to\Map(Y)$ be an irreducible homomorphism. Then the following holds:
\begin{itemize}
\item $i(\phi_*(\gamma), \phi_*(\gamma')) = 1$ for all curves $\gamma, \gamma'\subset X$ with $i(\gamma, \gamma')=1$.
\item If $a,b,c,d,x,y$ and $z$ is a lantern with the property that no two curves chosen among $a,b,c,d$ and $x$ 
separate $X$, then $\phi_*(a)$, $\phi_*(b)$, $\phi_*(c)$, $\phi_*(d)$, $\phi_*(x)$, $\phi_*(y)$ and $\phi_*(z)$ 
is a lantern in $Y$. 
\end{itemize}
\end{lem}

We prove the second claim. By the irreducible case of Proposition \ref{main} we know that if 
$\gamma$ is any component of the lantern in question, then $\phi_*(\gamma)$ is a single curve and $\phi(\delta_\gamma)=\delta_{\phi_*(\gamma)}$. In particular notice that the Dehn-twists along $\phi_*(a)$, $\phi_*(b)$, $\phi_*(c)$, $\phi_*(d)$, $\phi_*(x)$, $\phi_*(y)$ and $\phi_*(z)$ satisfy the lantern relation. Since $a,b,c,d,x$ are pairwise
disjoint, Corollary \ref{map-phi} yields that the curves $\phi_*(a)$, $\phi_*(b)$, $\phi_*(c)$, $\phi_*(d)$, $\phi_*(x)$ are also
pairwise disjoint. Moreover, the irreducibility of $\phi$, the assumption that that no two curves chosen among $a,b,c,d$ and $x$ 
separate $X$, and Corollary 
\ref{kor:irreducible} imply that the curves $\phi_*(a)$, $\phi_*(b)$, $\phi_*(c)$, $\phi_*(d)$ and $\phi_*(x)$ 
are pairwise distinct. Thus, the claim follows from Proposition \ref{fact-lantern}.\qed
\medskip

We are now ready to treat the reducible case of Proposition \ref{main}.

\begin{proof}[Proof of Proposition \ref{main} for reducible $\phi$] 
Let $\phi:\Map(X)\to\Map(Y)$ be a non-trivial reducible homomorphism, and let $\eta$ be the maximal multicurve 
in $Y$ which is componentwise fixed by $\phi(\Map(X))$. Recall the exact sequence \eqref{eq:central3}:
$$1\to\BT_\eta\to\CZ_0(\BT_\eta)\to\Map(Y_\eta')\to 0$$
Lemma \ref{trick2} shows that $\phi(\Map(X))\subset\CZ_0(\BT_\eta)$ and that the composition 
$$\phi':\Map(X)\to\Map(Y_\eta')$$ 
of $\phi$ and the homomorphism $\CZ_0(\BT_\eta)\to\Map(Y_\eta')$ is not trivial. Observe that $\phi'$ is irreducible 
because $\eta$ was chosen to be maximal.

The surface $Y_\eta'$ may well be disconnected; if this is the case, $\Map(Y_\eta')$ is by definition the direct 
product of the mapping class groups of the connected components of $Y_\eta'$. Noticing that the sum of the genera 
of the components of $Y_\eta'$ is bounded above by the genus of $Y$, it follows from the bound $g'\le 2g-1$ 
and from Proposition \ref{blabla} that $Y_\eta'$ contains at a single component $Y_\eta''$ on which $\phi(\Map(X))$ 
acts nontrivially. Hence, we can apply the irreducible case of Proposition \ref{main} and deduce that 
$\phi':\Map(X)\to\Map(Y_\eta'')$ maps Dehn twists to possibly left Dehn twists. Conjugating $\phi$ by an outer automorphism of $\Map(X)$ we may assume without loss of generality that $\phi'$ maps Dehn twists to Dehn twists.

Suppose now that $a,b,c,d,x,y$ and $z$ form a lantern in $X$ as in Lemma \ref{intersection-number-0}; such a 
lantern exists because $X$ has genus at least $3$. By Lemma \ref{intersection-number-0} we obtain that the images
 of these curves under $\phi_*'$ also form a lantern. In other words, if $S\subset X$ is the four-holed sphere 
with boundary $a\cup b\cup c\cup d$ then there is an embedding $\iota:S\to Y_\eta''\subset Y_\eta'$ such that 
for any $\gamma\in\{a,\dots,z\}$ we have
$$\phi'(\delta_\gamma)=\delta_{\iota(\gamma)}$$
Identifying $Y_\eta''$ with a connected component of $Y_\eta'=Y\setminus \eta$ we obtain an embedding 
$\hat\iota:S\to Y$. We claim that for any $\gamma$ in the lantern $a,b,c,d,x,y,z$ we have 
$\phi(\delta_\gamma)=\delta_{\hat\iota(\gamma)}$. 

A priori we only have that, for any such $\gamma$, both $\phi(\delta_\gamma)$ and $\delta_{\hat\iota(\gamma)}$
 project to the same element $\delta_{\iota(\gamma)}$ under the homomorphism $\CZ_0(\BT_\eta)\to\Map(Y_\eta')$.
 In other words, there is $\tau_\gamma\in\BT_\eta$ with $\phi(\delta_\gamma)=\delta_{\hat\iota(\gamma)}\tau_\gamma$. 
Observe that since any two curves $\gamma,\gamma'$ in the lantern $a,b,c,d,x,y,z$ are non-separating, the 
Dehn twists $\delta_\gamma$ and $\delta_{\gamma'}$ are conjugate in $\Map(X)$. 
Therefore, their images under $\phi$ are also conjugate in $\phi(\Map(X))\subset\CZ_0(\BT_\eta)$. 
Since $\BT_\eta$ is central in $\CZ_0(\BT_\eta)$, it follows that in fact $\tau_\gamma=\tau_{\gamma'}$ 
for any two curves $\gamma$ and $\gamma'$ in the lantern. Denote by $\tau$ the element of $\BT_\eta$ so obtained. 

On the other hand, both $\delta_a,\dots,\delta_z$ and $\delta_{\hat\iota(a)},\dots,\delta_{\hat\iota(z)}$ 
satisfy the lantern relation and, moreover, $\tau$ commutes with everything. Hence
\begin{align*}
1 & = \phi(\delta_a)\phi(\delta_b)\phi(\delta_c)\phi(\delta_d)\phi(\delta_z)^{-1}\phi(\delta_y)^{-1}\phi(\delta_x)^{-1} =\\
& = \delta_{\hat\iota(a)}\tau\delta_{\hat\iota(b)}\tau\delta_{\hat\iota(c)}\tau\delta_{\hat\iota(d)}\tau\tau^{-1}\delta_{\hat\iota(z)}^{-1}\tau^{-1}\delta_{\hat\iota(y)}^{-1}\tau^{-1}\delta_{\hat\iota(x)}^{-1}=\\
& = \delta_{\hat\iota(a)}\delta_{\hat\iota(b)}\delta_{\hat\iota(c)}\delta_{\hat\iota(d)}\delta_{\hat\iota(z)}^{-1}\delta_{\hat\iota(y)}^{-1}\delta_{\hat\iota(x)}^{-1}\tau=\tau
\end{align*}
Hence, we have proved that 
$$\phi(\delta_a)=\delta_{\hat\iota(a)}\tau=\delta_{\hat\iota(a)}$$
In other words, the image under $\phi$ of the Dehn twist along some, and hence every, non-separating curve is a Dehn twist. 
\end{proof}

%

\section{Reducing to the irreducible}

In this section we explain how to reduce the proof of Theorem \ref{dogs-bollocks} to the case
of irreducible homomorphisms between mapping class groups of surfaces without boundary.

\subsection{Weak embeddings}
Observe there are no embeddings $X\to Y$ if $X$
has no boundary but $Y$ does (compare with Corollary \ref{kor-embedding} below). We are going to relax the definition of embedding to allow 
for this possibility. For this purpose, it is convenient to regard $X$ and $Y$ as possibly non-compact surfaces 
without marked points; recall that we declared ourselves to be free to switch between cusps, marked points and ends.

\begin{defi*}
Let $X$ and $Y$ be possibly non-compact surfaces of finite topological type without marked points. 
A {\em weak embedding} $\iota:X\to Y$ is a topological embedding of $X$ into $Y$.
\end{defi*}

Given two surfaces $X$ and $Y$ without marked points there are two, essentially unique, compact surfaces 
$\hat X$ and $\hat Y$ with sets $P_{\hat X}$ and $P_{\hat Y}$ of marked points and with 
$X=\hat X\setminus P_{\hat X}$ and $Y=\hat Y\setminus P_{\hat Y}$. We will say that a weak embedding 
$\iota:X\to Y$ is {\em induced by an embedding} $\hat\iota: (\hat X,P_{\hat X}) \to (\hat Y,P_{\hat Y})$ if 
there is a homeomorphism $f:Y\to Y$ which is isotopic to the identity relative to $P_{\hat Y}$,
and $\hat\iota\vert_{X}=f\circ\iota$.

It is easy to describe which weak embeddings are induced by embeddings: 
A weak embedding $\iota:X\to Y$ is induced by an embedding if and only if the image $\iota(\gamma)$ of every
 curve $\gamma\subset X$ which bounds a disk in $\hat X$ containing at most one marked point bounds a disk 
in $\hat Y$ which again contains at most one marked point. Since $\iota(\gamma)$ bounds a disk without punctures 
if $\gamma$ does, we can reformulate this equivalence in terms of mapping classes:

\begin{lem}\label{weak-to-real}
A weak embedding $\iota:X\to Y$ is induced by an embedding if and only if $\delta_{\iota(\gamma)}$ is trivial 
in $\Map(Y)$ for every, a fortiori non-essential, curve $\gamma\subset X$ which bounds a disk with a puncture.\qed
\end{lem}

Notice that in general a weak embedding $X\to Y$ does not induce a homomorphism $\Map(X)\to\Map(Y)$. On the other hand, the following proposition asserts that if a homomorphism $\Map(X)\to\Map(Y)$ is, as far as it goes, induced by a weak embedding, then it is induced by an actual embedding.

\begin{prop}\label{weak-hom}
Let $X$ and $Y$ be surfaces of finite type and genus at least $3$. Suppose that $\phi:\Map(X)\to\Map(Y)$ is a 
homomorphism such that there is a weak embedding $\iota:X\to Y$ with the property that for every non-separating 
curve $\gamma\subset X$ we have $\phi(\delta_\gamma)=\delta_{\iota(\gamma)}$. Then $\phi$ is induced by an embedding 
$X\to Y$.  
\end{prop}
\begin{proof}
Suppose that $a\subset X$ bounds a disk with one puncture and consider the lantern in $X$ given 
in Figure \ref{fig-lantern-cusp}. We denote the bold-printed curves by $a,b,c,d$ and the dotted lines by $x,y,z$;
observe that $a$ is the  only non-essential curve in the lantern. 

\begin{figure}[tbh] \unitlength=1in
\begin{center} 
\includegraphics[width=3.5in]{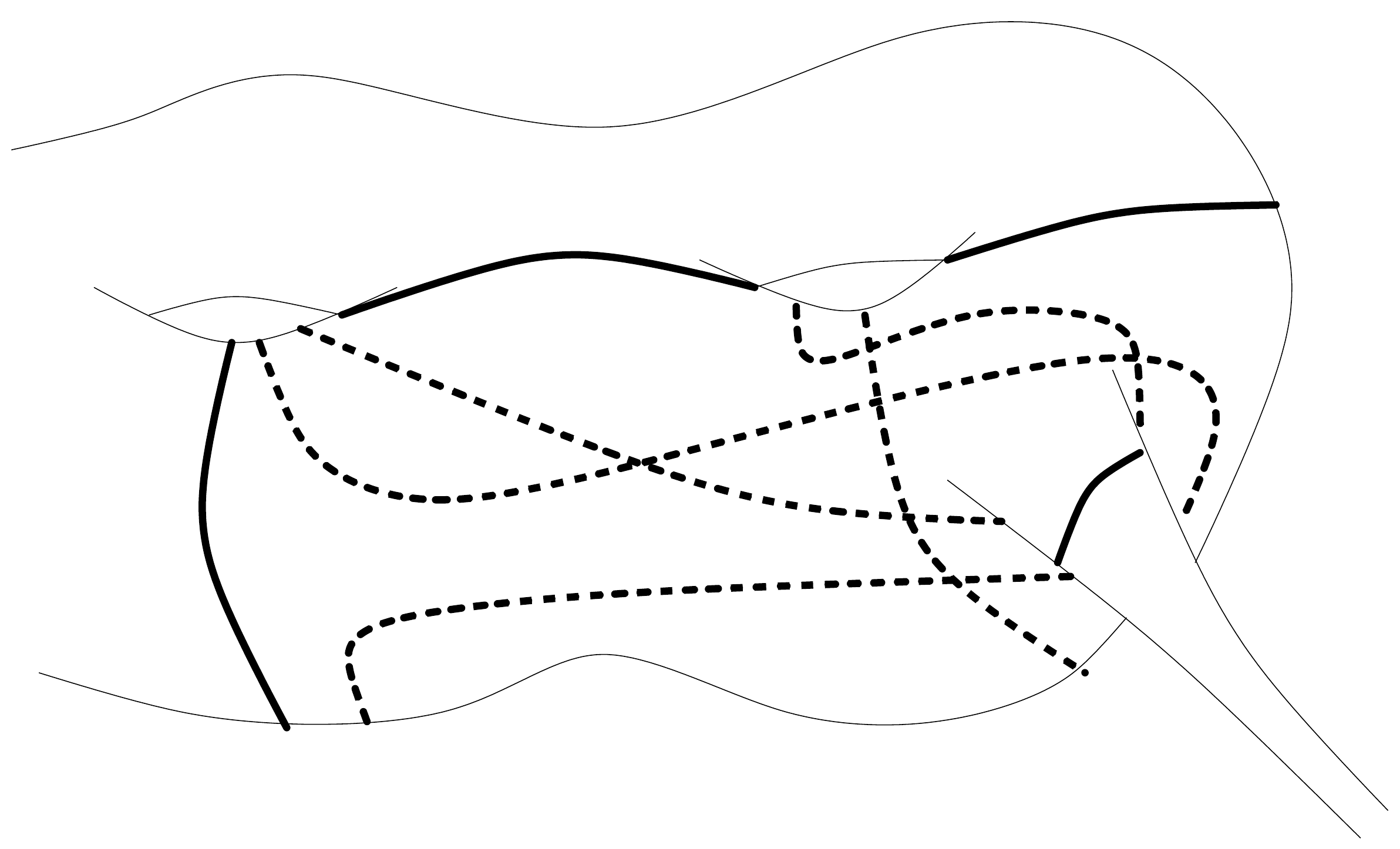}
\end{center}
\caption{A lantern in $X$ where one of the curves is non-essential and all the others are non-separating.}
\label{fig-lantern-cusp}
\end{figure}
 
By the lantern relation and because $a$ is not essential we have
\begin{equation}\label{patata1}
1=\delta_a=\delta_x\delta_y\delta_z\delta_b^{-1}\delta_c^{-1}\delta_d^{-1}
\end{equation}
The images under $\iota$ of the curves $a,b,c,d,x,y,z$ also form a lantern in $Y$ and hence we obtain
\begin{equation}\label{patata2}
\delta_{\iota(a)}=\delta_{\iota(x)}\delta_{\iota(y)}\delta_{\iota(z)}\delta_{\iota(b)}^{-1}
\delta_{\iota(c)}^{-1}\delta_{\iota(d)}^{-1}
\end{equation}
The assumption in the Proposition implies that the image under the homomorphism $\phi$ of the right side of 
\eqref{patata1} is equal to the right side of \eqref{patata2}. This implies that $\delta_{\iota(a)}$ is trivial. 
Lemma \ref{weak-to-real} shows now that the weak embedding $\iota$ is induced by an embedding, which we again 
denote by $\iota$. Let $\iota_\#$ the homomorphism induced by $\iota$. Since,
 by assumption,   $\phi(\delta_\gamma)=\delta_{\iota(\gamma)}=\iota_\#(\delta_\gamma)$ for all non-separating curves $\gamma$,
 and since the Dehn twists along these curves
 generate the mapping class group, we deduce that $\phi=\iota_\#$. In particular, $\phi$ is induced by an embedding, 
as we needed to prove.
\end{proof}

\subsection{Down to the irreducible case}
Armed with Proposition \ref{weak-hom}, we now prove that it suffices to establish Theorem \ref{dogs-bollocks}
for irreducible homomorphisms. Namely, we have:

\begin{lem}\label{tetete1}
Suppose that Theorem \ref{dogs-bollocks} holds for irreducible homomorphisms. Then it also holds for reducible ones.
\end{lem}
\begin{proof}

Let $X$ and $Y$ be surfaces as in the statement of Theorem \ref{dogs-bollocks} and suppose that 
$\phi:\Map(X)\to\Map(Y)$ is a non-trivial reducible homomorphism. Let $\eta$ be a maximal multicurve in $Y$ 
whose every component of $\eta$ is invariant under $\phi(\Map(X))$; by Lemma 
\ref{trick2}, $\phi(\Map(X))\subset\CZ_0(\BT_\eta)$. Consider, as in the proof of Proposition \ref{main}, the composition
$$\phi':\Map(X)\to\Map(Y_\eta')$$
of $\phi$ and the homomorphism in \eqref{eq:central3}. Lemma \ref{trick} shows that $\phi'$ is non-trivial; moreover, it is irreducible 
by the maximality of $\eta$. Now, Proposition \ref{main} implies that for any $\gamma$ non-separating both $\phi(\delta_\gamma)=\delta_{\phi_*(\gamma)}$ and $\phi'(\delta_\gamma)=\delta_{\phi_*'(\gamma)}$ are Dehn twists. As in the proof of the reducible case of Proposition \ref{main} we can consider $Y_\eta'=Y\setminus\eta$ as a subsurface of $Y$. Clearly, $\phi_*(\gamma)=\phi_*'(\gamma)$ after this identification. 

Assume that Theorem \ref{dogs-bollocks} holds for irreducible homomorphisms. Since $\phi'$ is irreducible, we 
obtain an embedding
$$\iota:X\to Y_\eta'$$
inducing $\phi'$. Consider the embedding $\iota:X\to Y_\eta'$ as a weak embedding $\hat\iota:X\to Y$. By the 
above, $\phi(\delta_\gamma) = \delta_{\hat\iota(\gamma)}$, for every $\gamma\subset X$ non-separating. 
Finally, Proposition \ref{weak-hom} implies that $\phi$ is induced by an embedding.
\end{proof}

\subsection{No factors}
Let $\phi:\Map(X)\to\Map(Y)$ be a homomorphism as in the statement of Theorem \ref{dogs-bollocks}. We will say that
$\phi$ {\em factors} 
if there are a surface $\bar X$, an embedding $\bar\iota:X\to\bar X$, and a homomorphism $\bar\phi:\Map(\bar X)\to\Map(Y)$ 
such that the following diagram commutes:
\begin{equation}\label{eq-factor}
\xymatrix{
\Map(X)\ar[d]_{\bar\iota_\#}\ar[rd]^\phi & \\
\Map(\bar X)\ar[r]_{\bar\phi} &\Map(Y)
}
\end{equation}
Since the composition of two embeddings is an embedding, we deduce that $\phi$ is induced by an embedding if $\bar\phi$ is. 
Since a homomorphism $\Map(X) \to \Map(Y)$ may factor only finitely many times, we obtain:  


\begin{lem}\label{tetete2}
If Proposition \ref{main} holds for homomorphisms $\phi:\Map(X)\to\Map(Y)$ which do not factor, then it holds in full 
generality.\qed
\end{lem}

Our next step is to prove that any irreducible homomorphism $\phi:\Map(X)\to\Map(Y)$ factors if $X$ has boundary.
We need to establish the following result first:

\begin{lem}\label{trivial-centralizer}
Suppose that $X$ and $Y$ are as in the statement of Theorem \ref{dogs-bollocks} and let $\phi:\Map(X)\to\Map(Y)$ be an
irreducible homomorphism. Then the centralizer of $\phi(\Map(X))$ in $\Map(Y)$ is trivial.
\end{lem}
\begin{proof}
Suppose, for contradiction, that there is a non-trivial element $f$ in $\CZ(\phi(\Map(X)))$; we will show
that $\phi$ is reducible. Noticing that 
the genus bounds in Theorem \ref{dogs-bollocks} are more generous than those in Lemma \ref{trivial-centralizer1}, 
we deduce from the latter that $f$ has infinite order. 
Let $\gamma \subset X$ be a non-separating curve
and recall that $\phi(\delta_\gamma)$ is a Dehn twist by Proposition \ref{main}.
Since $f$ commutes with 
$\phi(\delta_\gamma)$ it follows that $f$ is not pseudo-Anosov. In particular, $f$ must be reducible; let $\eta$ 
be the canonical reducing multicurve associated to $f$ \cite{BLM}. Since $\phi(\Map(X))$ commutes with $f$ we deduce 
that $\phi(\Map(X))$ preserves $\eta$. We will prove that $\phi(\Map(X))$ fixes some component of $\eta$, obtaining hence a contradiction to the assumption that $\phi$ is irreducible. The arguments 
are very similar to the arguments in the proof of Lemma \ref{induction} and Lemma \ref{prop:curves-invariant}.

First, notice that the same arguments as the ones used to prove Lemma \ref{induction} imply 
that some component of $\eta$ is fixed if some component of $Y\setminus\eta$ is a disk or an annulus. Suppose that this is not the case. Then $Y\setminus\eta$ has at most $2g'-2\le 4g-4$ 
components. Hence Theorem \ref{thm:luisito} implies that $\phi(\Map(X))$ fixes every component of $Y\setminus\eta$. Using again that no component of $Y\setminus\eta$ is a disk or an annulus we deduce that every such component $C$
 has at most $2g'+2\le4g-2$ boundary components. Hence Theorem \ref{thm:luisito} implies that $\phi(\Map(X))$ 
fixes every component of $\D C\subset\eta$. We have proved that some component of $\eta$ is fixed by $\phi(\Map(X))$ 
and hence that $\phi$ is reducible, as desired.
\end{proof}

We can now prove:

\begin{kor}\label{tetete3}
Suppose that $X$ and $Y$ are as in Theorem \ref{dogs-bollocks} and that $\D X\neq\emptyset$. Then every
irreducible homomorphism $\phi:\Map(X)\to\Map(Y)$ factors.
\end{kor}
\begin{proof}
Let $X'=X\setminus\D X$ be the surface obtained from $X$ by deleting the boundary and consider the associated embedding $\iota:X\to X'$. By Theorem \ref{center}, the homomorphism $\iota_\#:\Map(X)\to\Map(X')$ fits in the exact sequence
$$1\to\BT_{\D X}\to\Map(X)\to\Map(X')\to 1$$
where $\BT_{\D X}$ is the center of $\Map(X)$. It follows from Lemma \ref{trivial-centralizer} that if $\phi$ is irreducible, then $\BT_{\D X}\subset\Ker(\phi)$. We have proved that $\phi$ descends to $\phi':\Map(X')\to\Map(Y)$ and hence that $\phi$ factors as we needed to show.
\end{proof}

Combining Lemma \ref{tetete1}, Lemma \ref{tetete2} and Corollary \ref{tetete3} we deduce:

\begin{prop}\label{reduction}
Suppose that Theorem \ref{dogs-bollocks} holds if 
\begin{itemize}
\item $X$ and $Y$ have no boundary, and 
\item $\phi:\Map(X)\to\Map(Y)$ is irreducible and does not factor.
\end{itemize}
Then, Theorem \ref{dogs-bollocks} holds in full generality.\qed
\end{prop}

\section{Proof of Theorem \ref{dogs-bollocks}}
In this section we prove the main result of this paper, whose statement we now recall:

\begin{named}{Theorem \ref{dogs-bollocks}}
Suppose that $X$ and $Y$ are surfaces of finite topological type, of genus $g\ge 6$ and $g'\le 2g-1$ respectively; 
if $Y$ has genus $2g-1$, suppose also that it is not closed. Then every nontrivial homomorphism 
$$\phi:\Map(X)\to\Map(Y)$$
is induced by an embedding $X\to Y$. 
\end{named}

\begin{bem}
As mentioned in the introduction, the same conclusion as in Theorem \ref{dogs-bollocks} applies for homomorphisms $\phi:\Map(X)\to\Map(Y)$ if
 both $X$ and $Y$ have the same genus $g\in\{4,5\}$. This will be shown in the course of the proof.
\end{bem}

By Proposition \ref{reduction} we may assume that $X$ and $Y$ have no boundary, that $\phi$ is irreducible and that it 
does not factor. Moreover, by Proposition \ref{main}, the image of a Dehn twist $\delta_\gamma$ along a non-separating
 curve is either the right or the left Dehn twist along the non-separating curve $\phi_*(\gamma)$. Notice that, up to 
composing $\phi$ with an outer automorphism of $\Map(Y)$ induced by an orientation reversing homeomorphism of $Y$, we
may actually  assume that $\phi(\delta_\gamma)$ is actually a right Dehn twist for some, and hence every, non-separating 
curve 
$\gamma\subset X$. In light of this, from now on we will assume without further notice that we are in the following
situation:

\begin{quote}{\bf Standing assumptions:} $X$ and $Y$ have no boundary; $\phi$ is irreducible and does not factor; $\phi(\delta_\gamma)=\delta_{\phi_*(\gamma)}$ for all $\gamma\subset X$ non-separating.
\end{quote}

Under these assumptions, we now prove $\phi$ is induced by a homeomorphism. We will make extensive
use of the concrete set of generators of $\Map(X)$ given in Figure \ref{fig1}, which we include
here as Figure \ref{fig11} for convenience. The reader should have Figure \ref{fig11} constantly in mind 
through the rest of this section.

\begin{figure}[tbh] \unitlength=1in
\begin{center} 
\includegraphics[width=3.5in]{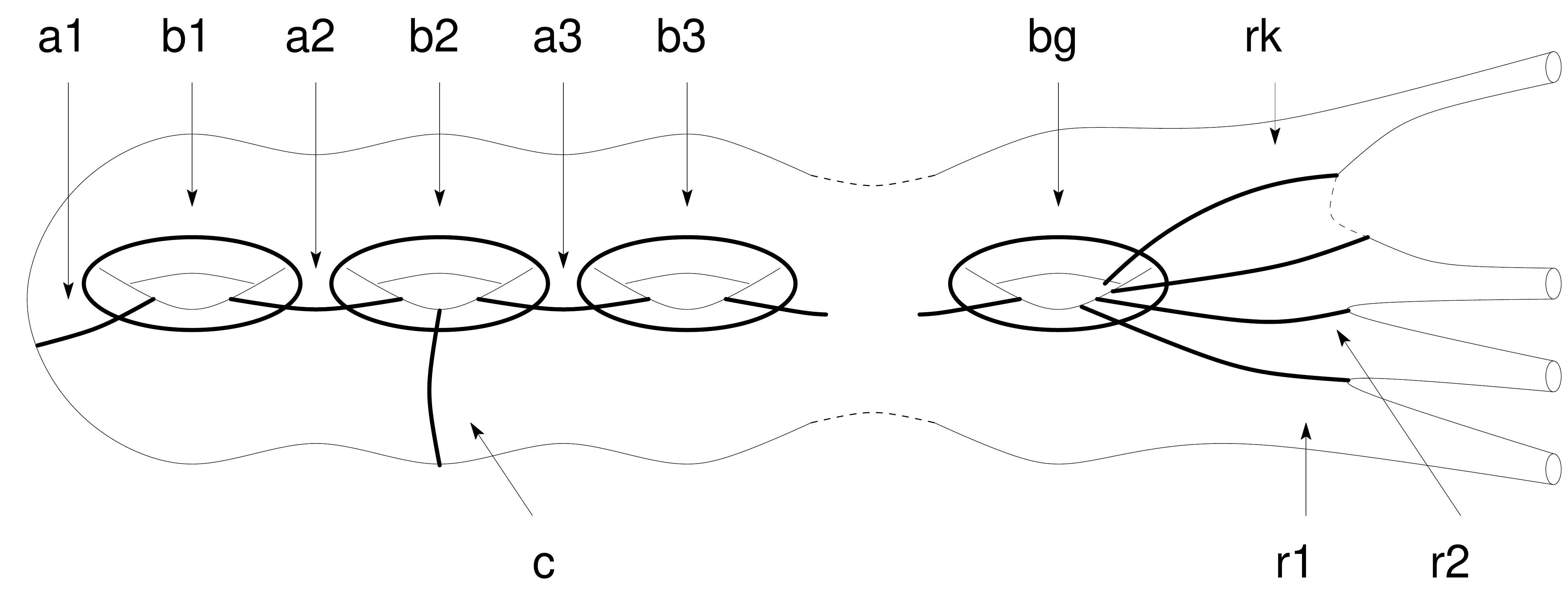}
\end{center}
\caption{Dehn twists along the curves $a_i,b_i,c$ and $r_i$ generate $\Map(X)$.}\label{fig11}
\end{figure}

 Notice that the sequence $a_1,b_1,a_2,b_2,\dots,a_g,b_g$ in Figure \ref{fig11} forms a chain; 
we will refer to it as the {\em $a_ib_i$-chain}. We refer to the multicurve $r_1\cup\dots\cup r_k$ as 
the {\em $r_i$-fan}. The curve $c$ would be worthy of a name such as {\em el pendejo} 
or {\em el pinganillo} but we prefer to call it simply {\em the curve $c$}. 

Observe that all the curves in Figure \ref{fig11} are non-separating. Hence, it follows from Proposition 
\ref{lem:single-component} that $\phi_*(\gamma)$ is a non-separating curve for any curve $\gamma$ in the 
collection $a_i,b_i,r_i,c$. We claim that the images under $\phi_*$ of all these curves in fill $Y$.

\begin{lem}\label{same-surface-0}
The image under $\phi_*$ of the $a_ib_i$-chain, the $r_i$-fan and the $c$-curve fill $Y$.
\end{lem}
\begin{proof}
Since the Dehn twists along the curves $a_i,b_i,r_i,c$ generate $\Map(X)$, any curve in $Y$ which 
is fixed by the images under $\phi$ of all these Dehn twists is fixed by $\phi(\Map(X))$. Noticing 
that the image under $\phi$ of a Dehn twist along any of  $a_i,b_j,r_l,c$ is a Dehn twist 
along the $\phi_*$-image of the curve, the claim follows since $\phi$ is assumed to be irreducible.
\end{proof}

Suppose that $\gamma,\gamma'$ are two distinct elements of the collection $a_i,b_i,r_i,c$. 
We now summarize several of the already established facts about the relative positions of the curves 
$\phi_*(\gamma),\phi_*(\gamma')$:
\begin{enumerate}
\item If $i(\gamma,\gamma')=0$ then $i(\phi_*(\gamma),\phi_*(\gamma'))=0$ by Corollary \ref{map-phi}.
\item If $\gamma, \gamma'$ are distinct and disjoint, and $X\setminus(\gamma\cup\gamma')$ is connected, 
then $\phi_*(\gamma)\neq\phi_*(\gamma')$ by Corollary \ref{kor:irreducible}.
\item If $i(\gamma,\gamma')=1$ then $i(\phi_*(\gamma),\phi_*(\gamma')=1$ by Lemma \ref{intersection-number-0}.
\end{enumerate}
Notice that these properties do not ensure that $\phi_*(r_i)\neq\phi_*(r_j)$ if $i\neq j$. 
We denote by $\CR\subset Y$ the maximal multicurve with the property that each one of its components is homotopic to one of the curves $\phi_*(r_i)$. Notice that $\CR=\emptyset$ if and only if $X$ has at most a puncture and that in any case $\CR$ has at most as many components as curves has the $r_i$-fan. 
The following lemma follows easily from (1), (2) and (3) above:

\begin{lem}\label{mapping-curves}
With the notation of Figure \ref{fig11} the following holds:
\begin{itemize}
\item The image under $\phi_*$ of the $a_ib_i$-chain is a chain of the same length in $Y$. 
\item Every component of the multicurve $\CR$ intersects the $\phi_*$-image of the $a_ib_i$-chain exactly in the component $\phi_*(b_g)$. 
\item The curve $\phi_*(c)$ is disjoint from every curve in $\CR$, intersects $\phi_*(b_2)$ exactly once, and is disjoint
from the images of the other curves in the $a_ib_i$-chain.
\qed
\end{itemize}
\end{lem}

At first glance, Lemma \ref{mapping-curves} yields the desired embedding without any further work, but this
is far from true. A first problem is that $\CR$ may have too few curves. Also, the curves in the $r_i$-fan 
come equipped with a natural ordering and that we do not know yet that $\phi_*$ preserves this ordering. 
Before tackling these problems, we clarify the position of $\phi_*(c)$:

\begin{lem}\label{same-surface-2}
Let $Z$ be a regular neighborhood of the $a_ib_i$-chain, noting that $c\subset Z$. Then there is an orientation 
preserving embedding $F:Z\to Y$ such that $\phi_*(\gamma)=F(\gamma)$ for $\gamma=a_i,b_i,c$ and $i=1,\dots,g$.
\end{lem}

\begin{proof}
The image under $\phi_*$ of the $a_ib_i$-chain is a chain of the same length, by Lemma \ref{mapping-curves}.
 Let $Z'$ be a regular 
neighborhood of the $\phi_*$-image of the $a_ib_i$-chain. Since regular neighborhoods of any two chains of the 
same length are homeomorphic in an orientation preserving manner, there is an orientation preserving embedding
$$F:Z\to Z'$$
with $F(a_i)=\phi_*(a_i)$ and $F(b_i)=\phi_*(b_i)$ for all $i$. It remains to prove that $F$ can be chosen 
so that $F(c)=\phi_*(c)$.

Let $Z_0\subset Z$ be the subsurface of $X$ filled by $a_1,b_1,a_2,b_2$ and observe that, up to isotopy,
 $c\subset Z_0$. The boundary of $Z_0$ is connected and, by the chain relation (see Section \ref{sec:general})
we can write the Dehn twist along $\D Z_0$ as:
$$\delta_{\D Z_0}=(\delta_{a_1}\delta_{b_2}\delta_{a_2}\delta_{b_2})^{10}$$
Hence we have
\begin{align*}
\phi(\delta_{\D Z_0}) 
&=(\phi(\delta_{a_1})\phi(\delta_{b_2})\phi(\delta_{a_2})\phi(\delta_{b_2}))^{10} \\
&=(\delta_{\phi_*(a_1)}\delta_{\phi_*(b_2)}\delta_{\phi_*(a_2)}\delta_{\phi_*(b_2)})^{10}\\
&=(\delta_{F(a_1)}\delta_{F(b_2)}\delta_{F(a_2)}\delta_{F(b_2)})^{10}\\
&=\delta_{F(\D Z_0)}
\end{align*}
where the last equality follows again from the chain relation.

Since $c$ is disjoint from $\D Z_0$ we have that $\delta_c$ and $\delta_{\D Z_0}$ 
commute. Hence $\phi_*(c)$ does not intersect $F(\D Z_0)$. On the other hand, since $\phi_*(c)$ 
intersects $\phi_*(b_2)\subset F(Z_0)$, we deduce from Lemma \ref{mapping-curves} that $\phi_*(c)\subset F(Z_0)$. 

Observe now that $F(Z)\setminus(\cup\phi_*(a_i)\cup\phi_*(b_i))\simeq Z\setminus(\cup a_i\cup b_i)$ is homeomorphic to an
 annulus $A$. It follows from Lemma \ref{mapping-curves} that the intersection of $\phi_*(c)$ with $A$ 
is an embedded arc whose endpoints are in the subsegments of $\D A$ corresponding to $\phi_*(b_2)$. There are 
two choices for such an arc. However, there is an involution $\tau:F(Z)\to F(Z)$ with $\tau(\phi_*(a_i))=\phi_*(a_i)$ and $\tau(\phi_*(b_i))=\phi_*(b_i)$ and which interchanges these two arcs. It follows that, up to possibly replacing $F$ by 
$\tau\circ F$, we have $F(c)=\phi_*(c)$, as we needed to prove.
\end{proof}

At this point we are ready to prove the first cases of Theorem \ref{dogs-bollocks}.

\begin{proof}[Proof of Theorem \ref{dogs-bollocks} ($X$ is closed or has one puncture)]
Let $Z$ be as in Lemma \ref{same-surface-2} and note that, from the assumptions on $X$,
the homomorphism $\sigma_\#:\Map(Z)\to\Map(X)$ induced by the embedding 
$\sigma:Z\to X$ is surjective. Let $\hat\phi:\Map(Z)\to\Map(Y)$ be the composition of $\sigma_\#$ 
and $\phi$. By Lemma \ref{same-surface-2}, $\hat\phi$ is induced by an embedding of $F:Z\to Y$.

Suppose first that $X$ has one puncture. In particular, there is a weak embedding $X\to Z\subset X$ which is 
homotopic to the identity $X\to X$. The composition of this weak embedding $X\to Z$ and of the embedding 
$F:Z\to Y$ is a weak embedding which, in the sense of Proposition \ref{weak-hom}, induces $\phi$. It follows 
from Proposition \ref{weak-hom} that $\phi$ is induced by an embedding, as we needed to prove.

The case that $X$ is closed is slightly more complicated. The assumption that $\phi$ is irreducible and that a 
collection of curves in $F(Z)$ fills $Y$, implies that $Y\setminus F(Z)$ is a disk containing at most one puncture. 
If $Y\setminus F(Z)$ is a disk without punctures, then the we can clearly extend the map $F$ to a continuous 
injective map $X\to Y$. Since any continuous injective map between closed connected surfaces is a homeomorphism, 
it is a fortiori an embedding and we are done in this case.

It remains to rule out the possibility that $X$ is closed and $Y$ has one puncture. Suppose that this is the case 
and let $\bar Y$ be the surface obtained from $Y$ by filling in its unique puncture. We can now apply the above argument
 to the induced homomorphism
$$\bar\phi:\Map(X)\to\Map(\bar Y),$$
to deduce that $\bar\phi$ is induced by an embedding $X \to \bar Y$.
Since any embedding from a closed surface is a homeomorphism we deduce that $\bar\phi$ is an isomorphism. 
Consider $\phi\circ\bar\phi^{-1}:\Map(\bar Y)\to\Map(Y)$ and notice that composing $\phi\circ\bar\phi^{-1}$ 
with the filling-in homomorphism $\Map(Y)\to\Map(\bar Y)$ we obtain the identity. Hence, $\phi\circ\bar\phi^{-1}$ 
is a splitting of the Birman exact sequence
$$1\to\pi_1(\bar Y)\to\Map(Y)\to\Map(\bar Y)\to 1,$$
which does not exist by Lemma \ref{no-birman-split}. It follows that $Y$ cannot have a puncture, as we needed to prove.
\end{proof}

We continue with the preparatory work needed to prove Theorem \ref{dogs-bollocks} in the general case. From now on we assume that $X$ has at least 2 punctures. 

\begin{quote}{\bf Standing assumption:} $X$ has at least 2 punctures.
\end{quote}

Continuing with the proof of Theorem \ref{dogs-bollocks}, we prove next $Y$ has the same genus as $X$.

\begin{lem}\label{same-surface-1}
Both surfaces $X$ and $Y$ have the same genus $g$.
\end{lem}
\begin{proof}
With the same notation as in Lemma \ref{same-surface-2} we need to prove that $S=Y\setminus F(Z)$ is a 
surface of genus $0$. By Lemma \ref{same-surface-0} the arcs $\rho_i=\phi_*(r_i)\cap S$ fill $S$. Denote by 
$\bar S$ the surface obtained by attaching a disk along $F(\D Z)\subset\D S$ and notice that the arcs $\rho_i$ 
can be extended to a collection of disjoint curves $\bar\rho_i$ in $\bar S$. Moreover, every curve in $\bar S$
 either agrees or intersects one of the curves $\bar \rho_i$ more than once, which is impossible if $S$ has genus 
at least $1$; this proves Lemma \ref{same-surface-1}.
\end{proof}

Notice that if $\eta\subset X$ is separating, all we know about 
$\phi(\delta_\eta)$ is that it is a root of a multitwist by Bridson's Theorem \ref{bridson}; 
in particular, $\phi(\delta_\eta)$ may be trivial or have finite order. If this is not the case, we denote 
by $\phi_*(\eta)$ the multicurve supporting any multitwist power of $\phi(\delta_\eta)$. Observe that if $\eta$ 
bounds a disk with punctures then, up to replacing the $a_ib_i$-chain by another such chain, we may assume that 
$i(\eta,a_i)=i(\eta,b_i)=0$ for all $i$. In particular, $\phi_*(\eta)$ does not intersect any of the curves 
$\phi_*(a_i)$ and $\phi_*(b_i)$. It follows that no component of $\phi_*(\eta)$ is non-separating. We record 
our conclusions:

\begin{lem}\label{disk}
Suppose that $\eta\subset X$ bounds a disk with punctures and that $\phi(\delta_\eta)$ has infinite order. 
Then every component of the multicurve $\phi_*(\eta)$ separates $Y$.\qed
\end{lem}

Our next goal is to bound the number of cusps of $Y$:

\begin{lem}\label{same-surface-3}
Every connected component of $Y \setminus \left (\bigcup_i\phi_*(a_i)\cup\CR\right)$
contains at most single puncture. In particular $Y$ has at most as many punctures as $X$.
\end{lem}

Recall that $\CR\subset Y$ is the maximal multicurve with the property that each one of its components is homotopic to one of the curves $\phi_*(r_i)$.

\begin{proof}
Observe that Lemma \ref{same-surface-0} and Lemma \ref{same-surface-2} imply that the union of $\CR$ and the image under $\phi_*$ of the 
$a_ib_i$-chain fill $Y$. In particular, every component of the complement in $Y$ of the union 
of $\CR$ and all the curves $\phi_*(a_i)$ and $\phi_*(b_i)$ contains at most one puncture of $Y$. It follows from Lemma \ref{mapping-curves} that the multicurve $\cup\phi_*(b_i)$ does not separate any of the components of the complement of 
$(\cup\phi_*(a_i))\cup\CR$ in $Y$. We have proved the first claim. 

It follows again from Lemma \ref{mapping-curves} that the multicurve $\cup\phi_*(a_i)\cup\CR$ separates $Y$ into at most $k$ components where $k\ge 2$ is the number of punctures of $X$. Thus,
$Y$ has at most as many punctures as $X$. 
\end{proof}

So far, we do not know much about the relative positions of the curves in $\CR$; this will change once we have established the next three lemmas.
 
\begin{lem}\label{bound-annulus}
Suppose that $a,b\subset X$ are non-separating curves that bound an annulus $A$. Then $\phi_*(a)$ 
and $\phi_*(b)$ bound an annulus $A'$ in $Y$; moreover, if $A$ contains exactly one puncture and 
$\phi_*(a)\neq\phi_*(b)$, then $A'$ also contains exactly one puncture.
\end{lem}
\begin{proof}
Notice that $A$ is disjoint from a chain of length $2g-1$. Since $\phi_*$ maps chains to chains (Lemma \ref{mapping-curves}), since it preserves disjointness (Corollary \ref{map-phi})
and since $Y$ has the same genus as $X$ (Lemma \ref{same-surface-1}), we deduce that $\phi_*(\D A)$ consists of
non-separating curves which are contained in an annulus in $Y$. The first claim follows.

Suppose that $\phi_*(a)\neq\phi_*(b)$; up to 
translating by a mapping class, we may assume that $a=a_1$ and that $b$ is a curve disjoint from
$(\cup a_i) \cup (\cup r_i)$ and with $i(b,b_1)=1$ and $i(b,b_i)=0$ for $i=2,\dots,g$ (compare with the dashed curve 
in Figure \ref{fig111}). 
\begin{figure}[tbh] \unitlength=1in
\begin{center} 
\includegraphics[width=3.5in]{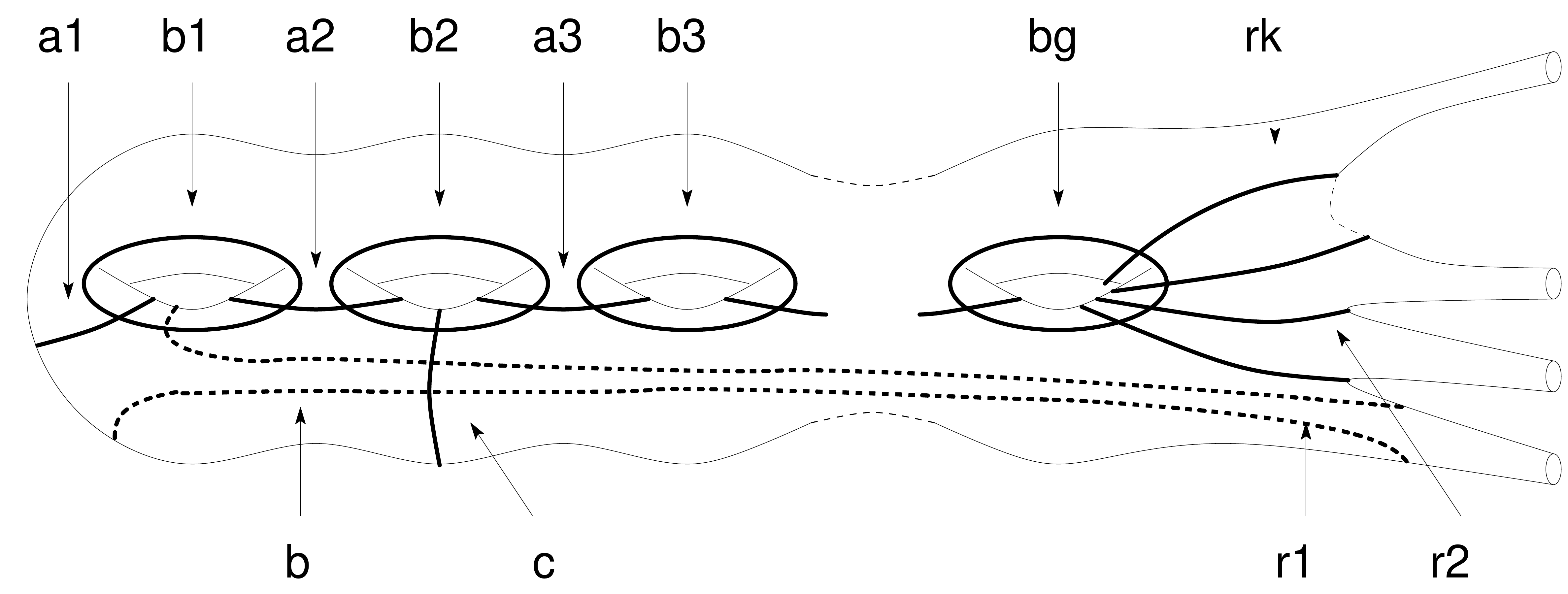}
\end{center}
\caption{}\label{fig111}
\end{figure}
Since $\phi_*$ preserves disjointness and intersection number one (Lemma \ref{intersection-number-0}), it follows 
that the annulus $A'$ bounded by $\phi_*(a)=\phi_*(a_1)$ and $\phi_*(b)$ is contained in one of the two connected components
of $Y \setminus \left (\bigcup_i\phi_*(a_i)\cup \bigcup_i\phi_*(r_i)\right )$ adjacent to $\phi_*(a_1)$. 
By Lemma \ref{same-surface-3}, each one of these components contains at most a puncture, and thus the claim follows.
\end{proof}

\begin{lem}\label{fill-cusp}
Let $\gamma, \gamma' \subset X$ be non-separating curves bounding an annulus with one puncture. 
Then $\phi_*(\gamma)\neq\phi_*(\gamma')$.
\end{lem}

\begin{proof}
We will prove that if $\phi_*(\gamma)=\phi_*(\gamma')$, then $\phi$ factors in the sense of \eqref{eq-factor}; notice
that this contradicts our standing assumption.

Suppose $\phi_*(\gamma)=\phi_*(\gamma')$, noting that Proposition \ref{main} implies that 
$\phi(\delta_\gamma)=\phi(\delta_{\gamma'})$. Let $p$ be the puncture in the annulus bounded by $\gamma$ and $\gamma'$. Consider the $\bar X$ surface 
obtained from $X$ by filling in the puncture $p$ and the Birman exact sequence \eqref{eq:birman1}:
$$1\to\pi_1(\bar X,p)\to\Map(X)\to\Map(\bar X)\to 1$$
associated to the embedding $X\to\bar X$. Let $\alpha\in\pi_1(\bar X,p)$ be the unique essential simply 
loop contained in the annulus bounded by $\gamma\cup\gamma'$. The image of $\alpha$ under the left
arrow of the Birman exact sequence is $\delta_\gamma\delta_{\gamma'}^{-1}$. Hence, $\alpha$ 
belongs to the kernel of $\phi$. Since $\pi_1(\bar X,p)$ has a set of generators consisting of 
translates of $\alpha$ by $\Map(X)$ we deduce that that $\pi_1(\bar X,p)\subset\Ker(\phi)$. This shows 
that $\phi:\Map(X)\to\Map(Y)$ factors through $\Map(\bar X)$ and concludes the proof of Lemma \ref{fill-cusp}.
\end{proof}

\begin{lem}\label{ordering}
Let $a, b \subset X$ be non-separating curves which bound an annulus with exactly two punctures.
Then  $\phi_*(a)$ and $\phi_*(b)$ bound an annulus $A'\subset Y$ with exactly two punctures. Moreover, if $x\subset A$
is any non-separating curve in $X$ separating the two punctures of $A$, then $\phi_*(x)\subset A'$ and separates the two 
punctures of $A'$.
\end{lem}

\begin{proof}
Let $x\subset A$ be a curve as in the statement. 
 Suppose first 
that $\phi_*(a)\neq\phi_*(b)$. Consider the annuli 
$A'$, $A_1'$ and $A_2'$ in $Y$ with boundaries
$$\D A'=\phi_*(a)\cup\phi_*(b),\ \D A_1'=\phi_*(a)\cup\phi_*(x),\ \D A_2'=\phi_*(x)\cup\phi_*(b).$$
By Lemmas \ref{bound-annulus} and \ref{fill-cusp},  the annuli $A_1'$ and $A_2'$ contain exactly
one puncture. Finally, since $\phi_*(x)$ does not intersect $\phi_*(a)\cup\phi_*(b)$, it 
follows that $A'=A_1'\cup A_2'$ and the claim follows.

It remains to rule out the possibility that $\phi_*(a)=\phi_*(b)$. Seeking a contradiction, suppose 
that this is the case. Consider curves $d,c,y,z$ as in Figure \ref{fig1120} 
\begin{figure}[tbh] \unitlength=1in
\begin{center} 
\includegraphics[width=2in]{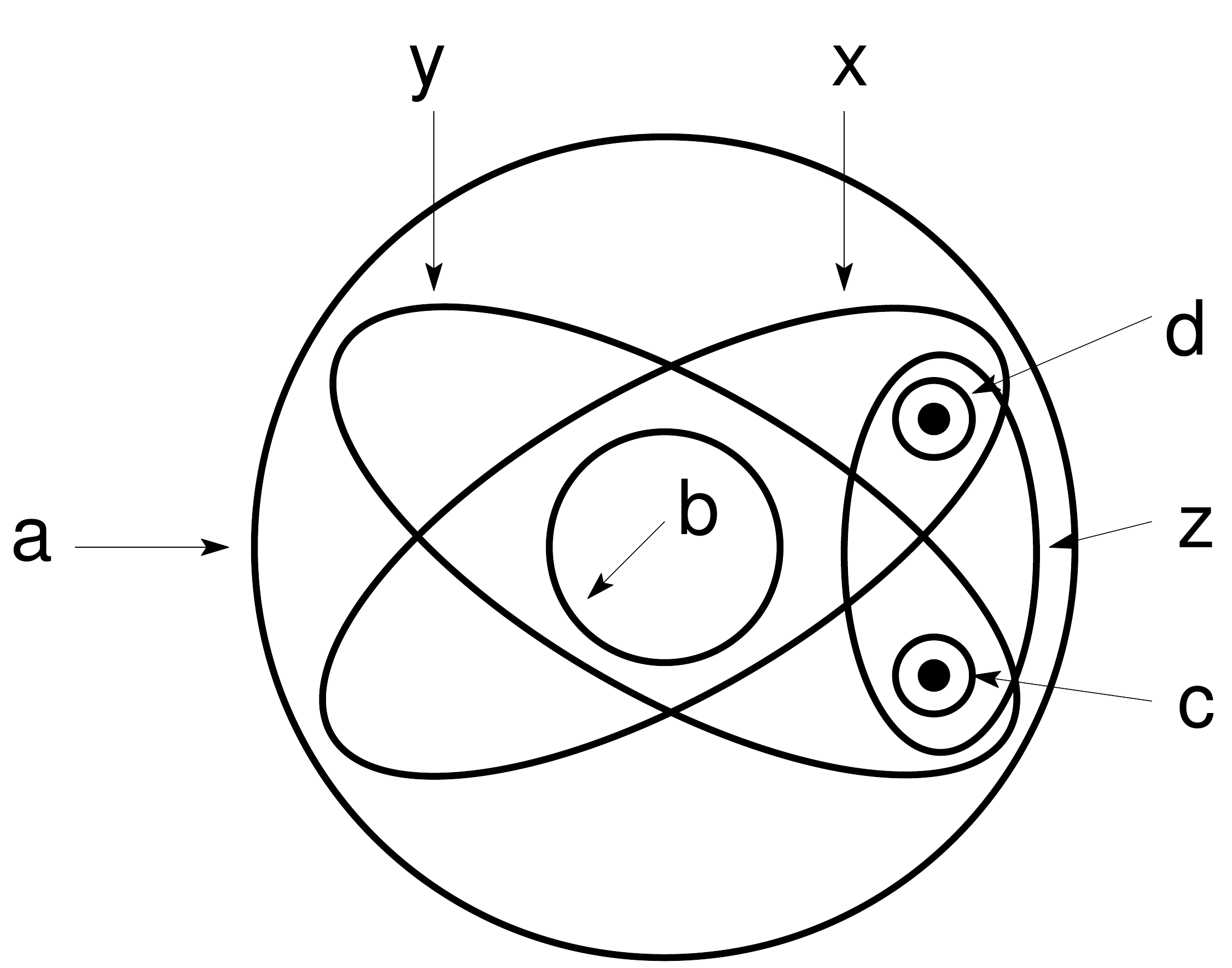}
\end{center}
\caption{The back dots represent cusps}\label{fig1120}
\end{figure}
and notice that $a,b,c,d,x,y,z$ is a lantern, and that $c,d$ are not essential. In particular, the lantern 
relation reduces to $\delta_a \delta_b = \delta_x \delta_y \delta_z$. Applying $\phi$ we obtain
$$\delta_{\phi_*(a)}^2=\delta_{\phi_*(x)}\delta_{\phi_*(y)}\phi(\delta_z)$$
By Bridson's theorem, $\phi(\delta_z)$ is a root of a multitwist. Since $\delta_a$ and $\delta_z$ 
commute, we have that $\delta_{\phi_*(a)}^2\phi(\delta_z)^{-1}$ is also a root of a multitwist. 

By the same arguments as above, there are annuli $A_1',A_2'$ in $Y$, each containing at most 
one puncture, with boundaries
$$\D A_1'=\phi_*(a)\cup\phi_*(x),\ \D A_2'=\phi_*(a)\cup\phi_*(y)$$
Observe that $i(\phi_*(x),\phi_*(y))$ is even. If $i(\phi_*(x),\phi_*(y))>2$, then by 
\cite[Theorem 3.10]{Hamidi-Tehrani} the element $\delta_{\phi_*(x)}\delta_{\phi_*(y)}$ is 
relatively pseudo-Anosov, and hence not a multitwist. If $i(\phi_*(x),\phi_*(y))=2$, then we
 are in the situation of Figure \ref{fig112}, 
\begin{figure}[tbh] \unitlength=1in
\begin{center} 
\includegraphics[width=1in]{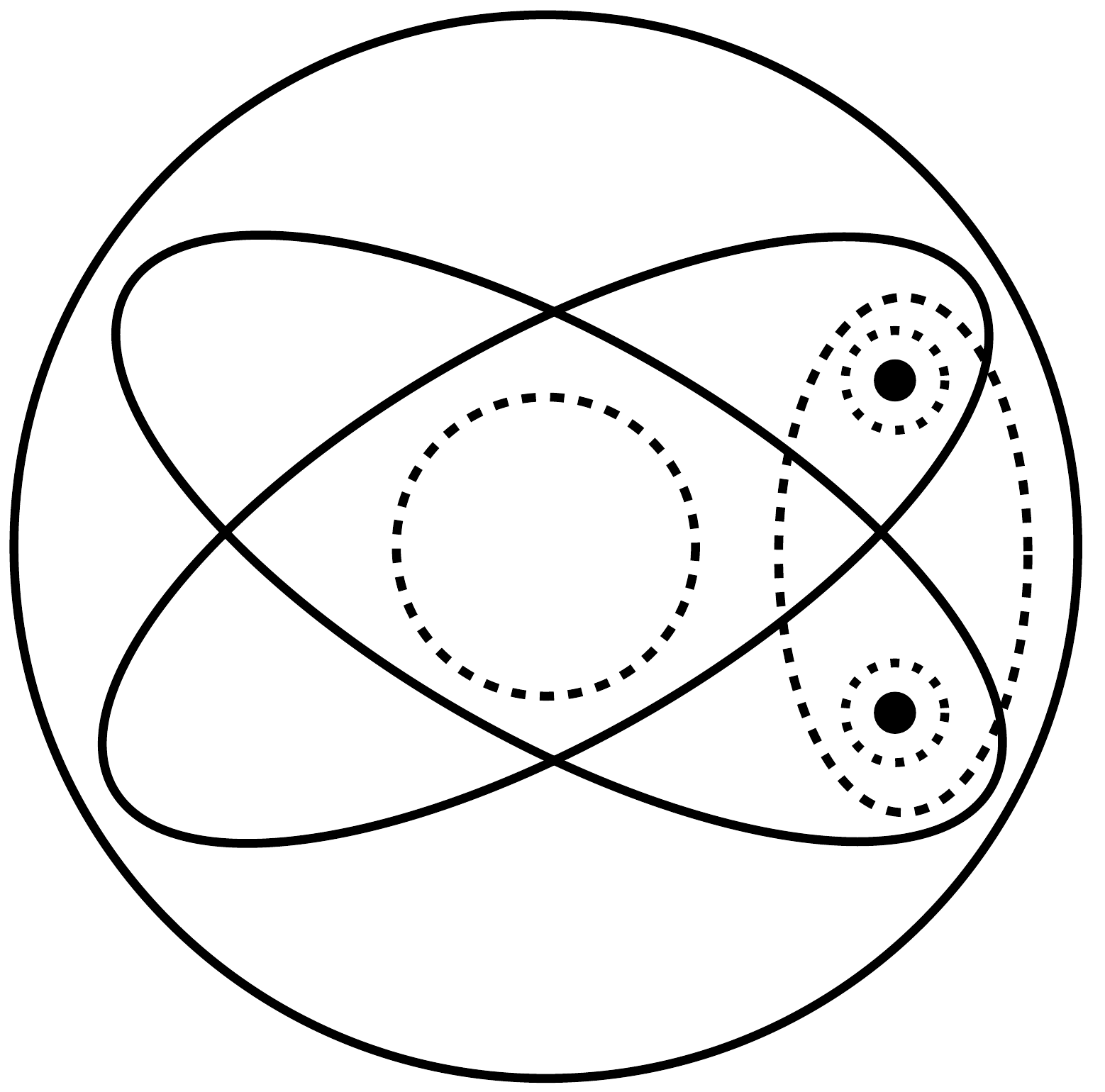}
\end{center}
\caption{The solid lines are $\phi_*(a)$, $\phi_*(x)$ and $\phi_*(y)$ and the black dots are cusps.}\label{fig112}
\end{figure}
meaning that we can extend $\phi_*(a),\phi_*(x),\phi_*(y)$ to a
 lantern $\phi_*(a),\hat b,\hat c,\hat d,\phi_*(x),\phi_*(y),\hat z$ with $\hat c,\hat d$ not-essential and $\hat b$ non-separating. From the lantern relation we obtain:
$$\delta_{\phi_*(y)}^{-1}\delta_{\phi_*(x)}^{-1}\delta_{\phi_*(a)}=\delta_{\hat z}^{-1}\delta_{\hat b}$$
This implies that $\phi(\delta_z)=\delta_{\hat z}^{-1}\delta_{\hat b}\delta_{\phi_*(a)}$ is a multitwist whose support 
contains non-separating components. This contradicts Lemma \ref{disk} and so we deduce that, if $\phi_*(a)=\phi_*(b)$ 
then $\phi_*(x)$ and $\phi_*(y)$ cannot have positive intersection number.

Finally, we treat the case $i(\phi_*(x), \phi_*(y))=0$. Since $x$ and $y$ are disjoint from $a$, it follows that the 
left side of 
$$\delta_{\phi_*(x)}^{-1}\delta_{\phi_*(y)}^{-1}\delta_{\phi_*(a)}^2=\phi(\delta_z)$$
is a multitwist supported on a multicurve contained in $\phi_*(a)\cup\phi_*(x)\cup\phi_*(y)$. Since 
these three curves are non-separating, it follows from Lemma \ref{disk} that $\phi(\delta_z)=\Id$. This shows 
that $\phi_*(a)=\phi_*(x)=\phi_*(y)$. Since $a$ and $x$ bound an annulus which exactly one puncture, 
we get a contradiction to Lemma \ref{fill-cusp}. Thus, we have proved
that $\phi_*(a)\neq\phi_*(b)$; this concludes the proof of Lemma \ref{ordering}
\end{proof}

We are now ready to finish the proof of Theorem \ref{dogs-bollocks}.

\begin{proof}[Proof of Theorem \ref{dogs-bollocks}]
Continuing with the same notation and standing assumptions, we now introduce orderings on 
the $r_i$-fan and the multicurve $\CR\subset Y$. 
In order to do so, observe 
that the union of the multicurve $\cup a_i$ and any of the curves in the $r_i$-fan separates $X$. Similarly,
 notice that by Lemma \ref{mapping-curves} the union of the multicurve $\cup\phi_*(a_i)$ and any of the 
components of $\CR$ is a multicurve consisting of $g+1$ non-separating curves. Since $Y$ has genus $g$, by Lemma 
\ref{same-surface-1}, we deduce that the union of the multicurve $\cup\phi_*(a_i)$ and any of the components 
of $\CR$ separates $Y$. We now define our orderings:
\begin{itemize}
\item Given two curves $r_i,r_j$ in the $r_i$-fan we say that $r_i\le r_j$ if $r_i$ and $c$ are in the same 
connected component of $X\setminus(a_1\cup\dots\cup a_g\cup r_j)$. Notice that the labeling in Figure \ref{fig11} is such that $r_i\le r_j$ for $i\le j$. 
\item Similarly, given two curves $r,r'\in\CR$ we say that $r\le r'$ if $r$ and $\phi_*(c)$ are in the same 
connected component of $X\setminus(\phi_*(a_1)\cup\dots\cup\phi_*(a_g)\cup r')$.
\end{itemize}
The minimal element of the $r_i$-fan, the curve $r_1$ in Figure \ref{fig11}, is called the {\em initial curve}
in the $r_i$-fan; we define the initial curve of the multicurve $\CR$ in an analogous way. 
We claim that its image under $\phi_*$ is the initial curve of $\CR$:
\medskip

\noindent{\bf Claim.} $\phi_*(r_1)$ is the initial curve in $\CR$.

\begin{proof}[Proof of the claim]
Suppose, for contradiction, that $\phi_*(r_1)$ is not the initial curve in $\CR$. 
Consider, besides the curves in Figure \ref{fig11}, a curve $c'$ as in Figure \ref{fig:closest}. In words, 
$c$ and $c'$ bound an annulus with exactly two punctures and
\begin{equation}\label{eq:closest}
i(c',r_i)=0\ \forall i\ge 2\ \ \hbox{and}\ \ i(c',a_i)=0\ \forall i
\end{equation}
Notice that by Lemma \ref{ordering}, $\phi_*(c)$ and $\phi_*(c')$ bound an annulus $A$ which contains exactly two punctures. 

\begin{figure}[tbh] \unitlength=1in
\begin{center} 
\includegraphics[width=3in]{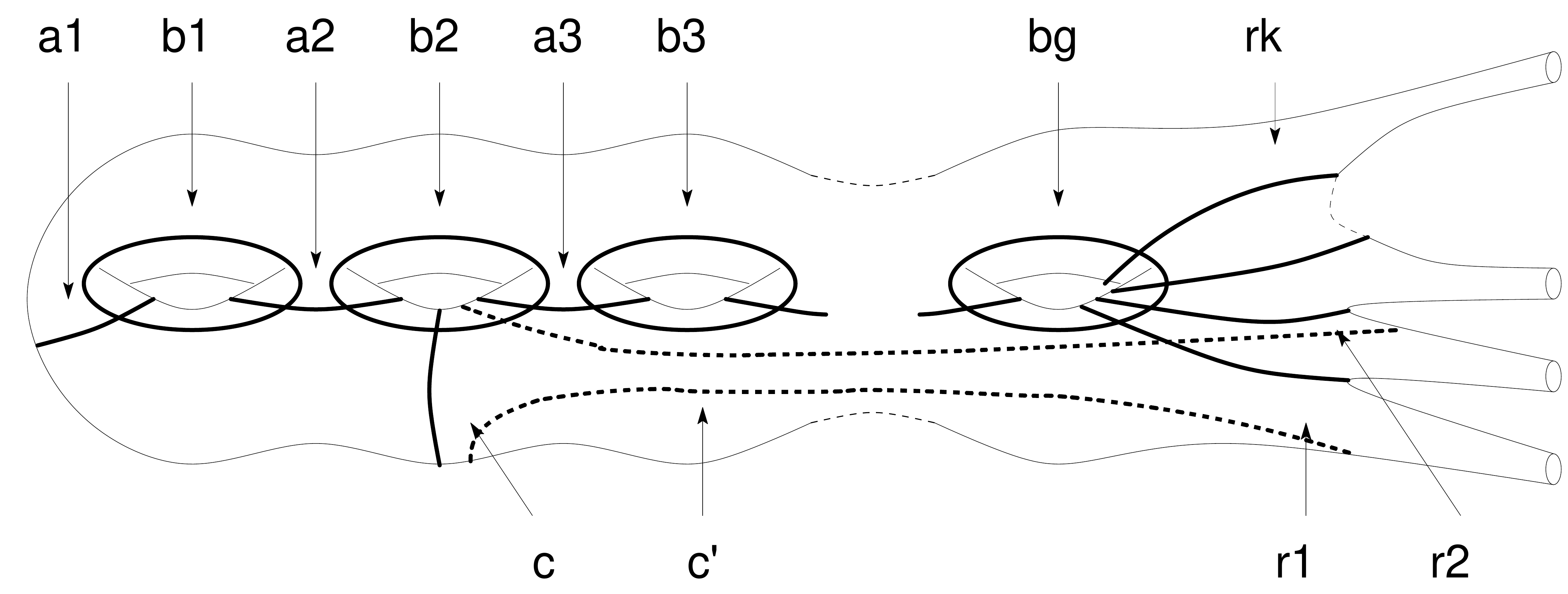}
\end{center}
\caption{The dotted curve $c'$ and $c$ bound an annulus with two punctures.}\label{fig:closest}
\end{figure}

Since $\phi(r_1)$ is not the initial curve, then $i(\phi_*(c'),\cup\phi_*(r_j))= 0$ for all $j$, as $i(c',r_j) = 0$ 
for all $j>1$. Also, by 
disjointness $i(\phi_*(c'),\phi_*(a_i))=0$ for all $i$. Since the boundary $\D A=\phi_*(c)\cup\phi_*(c')$ 
of the annulus $A$ is disjoint of $\cup\phi_*(a_i)\cup\phi_*(r_i)$ it is contained in one of the connected 
components of $X\setminus(\cup\phi_*(a_i)\cup\phi_*(r_i))$. However, each one of these components contains 
at most one puncture, by Lemma \ref{same-surface-3}. This contradicts Lemma \ref{ordering}, and thus we have
established the claim.
\end{proof}

We are now ready to prove that $\phi_*$ induces an order preserving bijection between the $r_i$-fan
and the multicurve $\CR$. Denote the curves in $\CR$ by $r_i'$, labeled in such a way that 
$r_i'\le r_j'$ if $i\le j$. By the previous claim, $\phi_*(r_1)=r_1'$. Next, consider the curve $r_2$, and observe that $r_1$ and $r_2$ bound an annulus with exactly one puncture. 
Hence,  Lemma
 \ref{fill-cusp} yields that $\phi_*(r_1)=r_1'$ and  $\phi_*(r_2)$  also bound an annulus with exactly one 
 puncture. In particular, $\phi_*(r_2)$ cannot be separated from $r_1'$ by any component of $\CR$. 
This proves that $\phi_*(r_2)=r_2'$. We now consider the curve $r_3$. The argument just used for 
$r_2$ implies that either $\phi_*(r_3)=r_3'$ or $\phi_*(r_3)=r_1'$. The latter is impossible, as
the curves $r_1$ and $r_3$ bound an annulus with exactly two punctures and hence so do 
$\phi_*(r_1)=r_1'$ and $\phi_*(r_3)$, by Lemma \ref{ordering}. Thus $\phi_*(r_3)=r_3'$. 
Repeating this argument as often as necessary we obtain that the map $\phi_*$ induces an injective, order 
preserving map from the $r_i$-fan to $\CR$. Since by definition $\CR$ has at most as many components as the 
$r_i$-fan, we have proved that this map is in fact an order preserving bijection.

Let $Z\subset X$ be a regular neighborhood of the $a_ib_i$-chain, and recall $c\subset Z$.
 By Lemma \ref{same-surface-2} there is an orientation preserving embedding $F:Z\to Y$ such that 
$\phi_*(\gamma)=F(\gamma)$ for $\gamma=a_i,b_i,c$ ($i=1,\dots,g$). We choose $Z$ so that it 
intersects every curve in the $r_i$-fan in a segment. Observe that Lemma \ref{mapping-curves} implies 
that $F$ can be isotoped so that 
$$F(Z\cap(\cup r_i))=F(Z)\cap\CR$$ 
The orderings of the $r_i$-fan and of $\CR$ induce orderings of $Z\cap(\cup r_i)$ and $Z\cap\CR$. Since the map 
$\phi_*$ preserves both orderings  we deduce that $F$ preserves the induced orderings of $Z\cap(\cup r_i)$ and $F(Z)\cap\CR$.

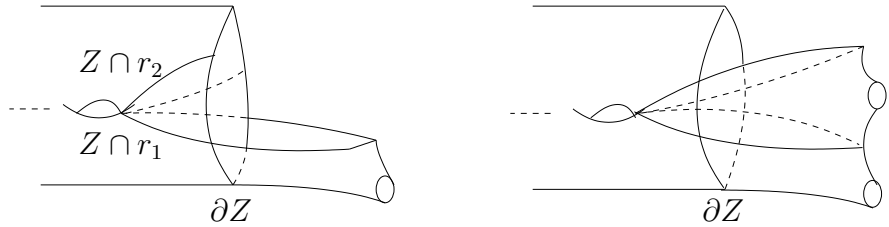
\begin{figure}[tbh] \unitlength=1in
\leavevmode \SetLabels
\L (.23*-.06) $\partial Z $\\
\L (.79*-.06) $\partial Z $\\
\L (.08*.27) $Z \cap r_1$\\
\L (.08*.67) $Z \cap r_2$\\
\endSetLabels
\par
\begin{center} 
\AffixLabels{\includegraphics[width=4.6in]{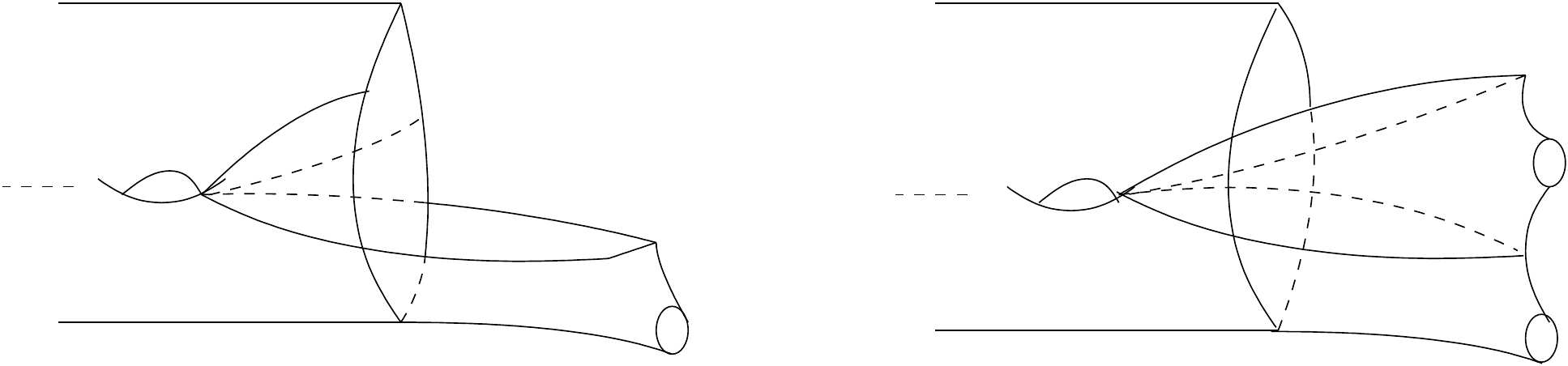}}
\end{center}
\caption{Attaching the first (left) and second (right) annuli along $\partial Z$.}\label{fig:annuli}
\end{figure}

Let $k$ be the number of curves in the $r_i$-fan, and thus in $\CR$. We successively attach $k$ annuli along the boundary $\D Z$ of $Z$, as indicated in Figure \ref{fig:annuli}. In this way we get a surface $Z_1$ naturally homeomorphic to $X$. We perform the analogous operation on $\D F(Z)$, thus obtaining a surface $Z_2$ which is naturally homeomorphic to $Y$. Since the map $\phi_*$ is preserves the orderings of the $r_i$-fan and of $R$, we get that the homeomorphism $F:Z\to F(Z)$ extends to a homeomorphism
$$\bar F:X \to Y$$
such that $$\bar F(\gamma)=\phi_*(\gamma)$$
for every curve  $\gamma$ in the collection $a_i,b_i,c,r_i$. It follows that the homomorphisms $\phi$ and $\bar F_\#$
both map the Dehn twist along $\gamma$ to the Dehn twist along $\phi_*(\gamma)$ and, in particular, to the same 
element in $\Map(Y)$. Since the Dehn twists along the curves $a_i,b_i,c,r_i$ generate $\Map(X)$, we deduce $\phi=F_\#$.
This finishes the proof of Theorem \ref{dogs-bollocks}.
\end{proof}

\section{Some consequences of Theorem \ref{dogs-bollocks}}
\label{sec:applications}
We now present several consequences of Theorem \ref{dogs-bollocks}; each 
of them results from imposing extra conditions on the surfaces involved 
and then reinterpreting the word ``embedding'' in that specific situation. 
Namely, observe that Proposition \ref{prop-embedding} immediately implies
the following:

\begin{kor}\label{kor-embedding}
Suppose that $X$ and $Y$ are surfaces of finite topological type and that $\iota:X\to Y$ is an embedding.
\begin{enumerate}
\item If $X$ is closed, i.e. if $X$ has neither boundary nor marked points, then $\iota$ is a homeomorphism.
\item If $X$ has no boundary, then $\iota$ is obtained by forgetting a (possibly empty) collection of
 punctures of $X$. In particular, $\iota_\#:\Map(X)\to\Map(Y)$ is surjective, and it is injective if and only 
if $\iota$ is a homeomorphism.
\item If $X=Y$ and $X$ has no boundary, then $\iota$ is a homeomorphism.
\item If $X=Y$ and $\iota$ is a subsurface embedding, then $\iota$ is (isotopic to) a homeomorphism. \qed
\end{enumerate}
\end{kor}
\medskip

Combining the first part of Corollary \ref{kor-embedding} and Theorem \ref{dogs-bollocks}, we obtain:

\begin{kor}\label{kor:closed}
Suppose that $X$ and $Y$ are surfaces of finite topological type, of genus $g\ge 6$ and $g'\le 2g-1$ respectively; 
if $Y$ has genus $2g-1$, suppose also that it is not closed. 

If $X$ is closed then every nontrivial homomorphism $\phi:\Map(X)\to\Map(Y)$ is induced by a 
homeomorphism $X\to Y$; in particular $\phi$ is an isomorphism.\qed
\end{kor}

As mentioned in the introduction, if there are no restrictions on the genus of $Y$ then Corollary \ref{kor:closed}
is far from true. Indeed, Theorem 1 of \cite{ALS} shows that for every closed surface $X$ there exist a closed
surface $Y$ and an injective homomorphism $\Map(X) \to \Map(Y)$. 

\medskip

Moving away from the closed case, if $X$ is allowed to have marked points and/or boundary then there are numerous 
non-trivial embeddings of $X$ into other surfaces. 
That said, the second part of Corollary \ref{kor-embedding} tells us that if $X$ has no boundary, then every 
subsurface embedding of $X$ into another surface is necessarily a homeomorphism. Hence we have:

\begin{kor}\label{no-boundary}
Suppose that $X$ and $Y$ are surfaces of finite topological type, of genus $g\ge 6$ and $g'\le 2g-1$ respectively; 
if $Y$ has genus $2g-1$, suppose also that it is not closed. 

If $X$ has empty boundary then any injective homomorphism $\phi:\Map(X)\to\Map(Y)$ is induced by a 
homeomorphism $X\to Y$; in particular $\phi$ is an isomorphism.\qed
\end{kor}

Again, if there are no restrictions on the genus of $Y$ then Corollary \ref{no-boundary} is simply not true; 
see Section 2 of \cite{Ivanov-McCarthy} and Theorem 2 of \cite{ALS}.

\medskip

By Corollary \ref{kor-embedding} (3), if $X$ has no boundary then any embedding $\iota: X \to X$ is a homeomorphism.
Observing that Theorem \ref{dogs-bollocks} applies for homomorphisms between surfaces of the same genus $g\ge 4$, 
(see the remark following the statement of the theorem) we deduce:

\begin{named}{Theorem \ref{no-boundary2}}
Let $X$ be a surface of finite topological type, of genus $g\ge 4$ and with empty boundary. Then any non-trivial 
endomorphism $\phi:\Map(X)\to\Map(X)$ is induced by a homeomorphism $X\to X$; in particular 
$\phi$ is an isomorphism.\qed
\end{named}

Note that Theorem \ref{no-boundary2} fails if $X$ has boundary. 
However, by Corollary \ref{kor-embedding} (4), any subsurface embedding $X\to X$ is isotopic to a homeomorphism. 
Therefore, we recover the following result due to Ivanov-McCarthy \cite{Ivanov-McCarthy} 
(see \cite{Ivanov-2} and \cite{McCarthy} for related earlier results):

\begin{kor}[Ivanov-McCarthy]\label{ivanov-mccarthy}
Let $X$ be a surface of finite topological type, of genus $g\ge 4$. Then any injective homomorphism 
$\phi:\Map(X)\to\Map(X)$ is induced by a homeomorphism $X\to X$; in particular $\phi$ is an isomorphism.\qed
\end{kor}

\section{Proof of Theorem \ref{forgetful}}
\label{sec:forgetful-proof}
Given a Riemann surface $X$ of finite analytic type, endow the associated Teichm\"uller space $\CT(X)$ with the
 standard complex structure. Recall that $\Map(X)$ acts discretely on $\CT(X)$ by biholomorphic automorphisms.
 In particular, we can consider the moduli space
$$\CM(X)=\CT(X)/\Map(X)$$ 
as a complex orbifold; by construction it is a good orbifold, meaning that its universal cover is a manifold.

Suppose now that $Y$ is another Riemann surface of finite analytic type. We will consider maps $f:\CM(X)\to\CM(Y)$ 
in the category of orbifolds. Since $\CM(X)$ and $\CM(Y)$ are both good obifolds we can associate to every such 
map $f$ a homomorphism
$$f_*:\Map(X)\to\Map(Y)$$
and a holomorphic map
$$\tilde f:\CT(X)\to\CT(Y)$$
which is $f_*$-equivariant, that is,
$$\tilde f(\gamma x)=f_*(\gamma)\left(\tilde f(x)\right)\ \ \ \forall \gamma\in\Map(X)\ \hbox{and}\ x\in\CT(X)$$
and such that the following diagram commutes:
$$\xymatrix{
\CT(X)\ar[d]\ar[r]^{\tilde f} & \CT(Y)\ar[d]\\
\CM(X)\ar[r]^f &\CM(Y)
}$$
Here both vertical arrows are the standard projections.

\begin{bem}
In the remainder of this section we will treat $\CM(X)$ and $\CM(Y)$ as if they were manifolds. This is 
justified by two observations. First, every statement we make holds indistinguishable for manifolds and for
 orbifolds. And second, in all our geometric arguments we could pass to a manifold finite cover and work there.
 We hope that this does not cause any confusion.
\end{bem}

Suppose now that $X$ and $Y$ have the same genus and that $Y$ has at most as many marked points as $X$. 
Choosing an identification between the set of marked points of $Y$ and a subset of the set of marked points of 
$X$, we obtain a holomorphic map
$$\CM(X)\to\CM(Y)$$
obtained by forgetting all marked points of $X$ which do not correspond to a marked point of $Y$. Different 
identifications give rise to different maps; we will refer to these maps as {\em forgetful maps}. 

In the rest of the section we will prove Theorem \ref{forgetful}, whose statement we now recall:

\begin{named}{Theorem \ref{forgetful}}
Suppose that $X$ and $Y$ are Riemann surfaces of finite analytic type and assume that $X$ has 
genus $g\ge 6$ and $Y$ genus $g'\le 2g-1$; if $g'=2g-1$ assume further that $Y$ is not closed. 
Then, every non-constant holomorphic map
$$f:\CM(X)\to\CM(Y)$$
is a forgetful map.
\end{named}

In order to prove Theorem \ref{forgetful}, we will first deduce from Theorem \ref{dogs-bollocks} that there is a 
forgetful map $F:\CM(X)\to\CM(Y)$ homotopic to $f$; then we will modify an argument due to Eells-Sampson to conclude that 
that $f=F$. In fact we will prove, without any assumptions on the genus, that any two homotopic holomorphic 
maps between moduli spaces are equal:

\begin{named}{Proposition \ref{Eells-Sampson}}
Let $X$ and $Y$ be Riemann surfaces of finite analytic type and let $f_1,f_2:\CM(X)\to\CM(Y)$ be homotopic 
holomorphic maps. If $f_1$ is not constant, then $f_1=f_2$.
\end{named}

Armed with Proposition \ref{Eells-Sampson}, we now conclude the proof of Theorem \ref{forgetful}.

\begin{proof}[Proof of Theorem \ref{forgetful} from Proposition \ref{Eells-Sampson}]
Suppose that $f:\CM(X)\to\CM(Y)$ is holomorphic and not constant. It follows from the latter assumption and 
from Proposition \ref{Eells-Sampson} that $f$ is not homotopic to a constant map. In particular, the induced homomorphism
$$f_*:\Map(X)\to\Map(Y)$$
is non-trivial. Thus, it follows from Theorem \ref{dogs-bollocks} that $f_*$ is induced by an embedding. 
Now, since $X$ has empty boundary, Corollary \ref{kor-embedding} (2) tells us that every embedding $X \to Y$
is obtained by filling in a collection of punctures of $X$. It follows that there is a forgetful map
$$F:\CM(X)\to\CM(Y)$$
with $F_*=f_*$. We deduce that $F$ and $f$ are homotopic to each other because the universal cover $\CT(Y)$ of $\CM(Y)$ is contractible. Hence, Proposition \ref{Eells-Sampson} shows that $f=F$, as we needed to prove.
\end{proof}

The remainder of this section is devoted to prove Proposition \ref{Eells-Sampson}. Recall at this point that 
$\CT(X)$ admits many important $\Map(X)$-invariant metrics. In particular, we will endow:
\begin{itemize}
 \item $\CT(X)$ with McMullen's K\"ahler hyperbolic metric \cite{McMullen}, and
 \item $\CT(Y)$ with the Weil-Petersson metric.
\end{itemize}

The reason why we do not endow $\CT(X)$ with the Weil-Petersson metric is encapsulated in the following
observation:

\begin{lem}\label{lem1}
Every holomorphic map $f:\CT(X)\to\CT(Y)$ is Lipschitz.
\end{lem}
\begin{proof}
Denote by $\CT(X)_T$ and $\CT(Y)_T$ the Teichm\"uller spaces of $X$ and $Y$, respectively, both equipped with the 
Teichm\"uller metric. Consider $f$ as a composition of maps
$$\xymatrix{
\CT(X)\ar[r]^{\Id} & \CT(X)_T\ar[r]^f & \CT(Y)_T\ar[r]^{\Id} & \CT(Y)
}$$
By \cite[Theorem 1.1]{McMullen} the first arrow is bi-lipschitz. By Royden's theorem \cite{Royden}, the middle map is $1$-Lipschitz. Finally, the last arrow is also Lipschitz because the Teichm\"uller metric dominates the Weil-Petersson metric up to a constant factor \cite[Proposition 2.4]{McMullen}.
\end{proof}

We will also need the following fact from Teichm\"uller theory:

\begin{lem}\label{lem2}
There exists a collection $\{K_i\}_{n\in \BN}$ of subsets of $\CM(X)$ with the following properties:
\begin{itemize}
\item  $\CM(X) = \bigcup_{n \in \BN} K_i$,
\item $K_{n} \subset K_{n+1}$ for all $n$,
\item there is $L>0$ such that $K_{n+1}$ is contained within distance $L$ of $K_n$ for all $n$, and
\item the co-dimension one volume of $\D K_n$ decreases exponentially when $n\to\infty$.
\end{itemize}
\end{lem}

Recall that $\CT(X)$, and hence $\CM(X)$, has been endowed with McMullen's K\"ahler hyperbolic metric.

\begin{proof}
Let $\overline{\CM(X)}$ be the Deligne-Mumford compactification of the moduli space $\CM(X)$; recall that
points $Z \in \overline{\CM(X)}\setminus\CM(X)$ are surfaces with $k$ nodes ($k\ge 1$). 
 Wolpert \cite{Wolpert} proved that every point in $ \overline{\CM(X)}\setminus\CM(X)$ has a small neighborhood 
$\overline{U_Z}$ in $\overline{\CM(X)}$ whose intersection
$$U_Z=\overline{U_Z}\cap\CM(X)$$ 
with $\CM(X)$ is bi-holomorphic to a neighborhood of $(0,\dots,0)$ in
$$((\BD^*)^k\times(\BD)^{\dim_\BC(\CT(X))-k})/G$$
where $G$ is a finite group; here $\BD^*$ and $\BD$ are the punctured and unpunctured open unit disks in $\BC$. 
We may assume, without loss of generality, that
$$U_Z\simeq ((D^*)^k\times(D)^{\dim_\BC(\CT(X))-k})/G$$
where $D\subset\BD$ is the disk of Euclidean radius $\frac 12$ centered at $0$. 

Endow now $\BD^*$ and $\BD$ with the hyperbolic metric. For $n\in \BN \cup \{0\}$ let $D_n\subset D$ be the disk such 
that the hyperbolic distance between $\D D^*$ and $\D D_n^*$ is equal to $n$ in $\BD^*$. We set:
$$U_Z(n)=((D_n^*)^k\times(D_n)^{\dim_\BC(\CT(X))-k})/G$$
Observe that, with respect to the hyperbolic metric, the volume of $\D U_Z(n)$ decreases exponentially with $n$.

Since $\overline{\CM(X)}\setminus\CM(X)$ is compact, we can pick finitely many sets $U_{Z_1},\dots,U_{Z_r}$ 
such that $\CM(X)\setminus\cup_iU_{Z_i}$ is compact. For $n\in\BN$ we set
$$K_n=\CM(X)\setminus\cup_iU_{Z_i}(n)$$
By construction, 
$$\CM(X)=\cup_n K_n\ \ \hbox{and}\ \ K_n\subset K_{n+1}\ \forall n$$
We claim that the sets $K_n$ satisfy the rest of the desired properties.

First, recall that Theorem 1.1 of \cite{McMullen} implies that the Teichm\"uller and K\"ahler hyperbolic metrics on $\CM(X)$
are bi-Lipschitz equivalent to each other.  In particular, it suffices to prove the 
claim if we consider $\CM(X)$ equipped with the Teichm\"uller metric. It is due to Royden \cite{Royden} 
that the Teichm\"uller metric agrees with the Kobayashi metric. It follows that the inclusion
$$U_{Z_i}\subset\CM(X)$$
is 1-Lipschitz when we endow $U_{Z_i}$ with the product of hyperbolic metrics and the $\CM(X)$
 with the Teichm\"uller metric. By the choice of $D_n$, every point in $U_{Z_i}(n+1)$ is within 
distance $\sqrt{\dim_\BC(\CT(X))}$ of $U_{Z_i}(n)$ with respect to the hyperbolic metric. Hence, the same is true with respect 
to the Teichm\"uller metric. Therefore, $K_{n+1}$ is contained within a fixed Teichm\"uller distance of $K_n$.

Noticing that $\D K_n\subset\cup_{i=1,\dots,r}\D U_{Z_i}(n)$, that the volume of $\D U_{Z_i}(n)$ decreases exponentially with respect to the hyperbolic metric, and that $1$-Lipschitz maps contract volume, we deduce that the volume of $\D K_n$ also decreases exponentially with respect to the Teichm\"uller metric. This concludes the proof of Lemma \ref{lem2}.
\end{proof}

We are almost ready to prove Proposition \ref{Eells-Sampson}. We first remind the reader of a few facts and definitions on 
the energy of maps. Suppose that $N$ and $M$ are Riemannian manifolds and that $f:N\to M$ is a smooth map. The {\em energy density} of $f$ at $x\in N$ is defined to be:
$$E_x(f)=\sum_{i=1}^{\dim_\BR N}\Vert df_xv_i\Vert_M^2$$
where $v_1,\dots,v_{\dim_\BR N}$ is an arbitrary orthonormal basis of $T_xM$. The {\em energy} of $f$ is then the integral of the energy density:
$$E(f)=\int_N E_x(f) d\vol_N(x)$$
Here $d\vol_N$ is the Riemannian volume form of $N$.

Suppose now that $N$ and $M$ are K\"ahler and let $\omega_N$ and $\omega_M$ be the respective K\"ahler forms. Recall that 
$$\omega_N^{\dim_\BC N}=\omega_X\wedge\dots\wedge\omega_X$$
is a volume form on $N$; more concretely, it is a constant multiple of the Riemannian volume form $d\vol_X$, where the 
constant depends only on $\dim_\BC N$. 

Pulling back the K\"ahler form $\omega_M$ via $f:N\to M$, we also have the top-dimensional form $(f^*\omega_M)\omega_N^{\dim_\BC N-1}$ on $N$. A local computation due to Eells and Sampson \cite{Eells-Sampson} shows that for all $x\in N$ we have
\begin{equation}\label{eq:ES}
E_x(f)d\vol_N\ge c\cdot(f^*\omega_M)\omega_N^{\dim_\BC N-1}
\end{equation}
where $c>0$ is a constant which again only depends on the dimension $\dim_\BC N$. Moreover, equality holds in \eqref{eq:ES} if and only if $f$ is holomorphic at $x$.

\begin{bem}
We stress that the proof of \eqref{eq:ES} is infinitesimal. In particular, it is indifferent to any global geometric property of the involved manifolds such as completeness.
\end{bem}

We finally have all the ingredients needed to prove Proposition \ref{Eells-Sampson}:

\begin{proof}[Proof of Proposition \ref{Eells-Sampson}]
Suppose that $f_0,f_1:\CM(X)\to\CM(Y)$ are holomorphic maps and recall that we have endowed $\CM(Y)$ 
with the Weil-Petersson metric and $\CM(X)$ with McMullen's K\"ahler hyperbolic metric. We remind the reader that both 
metrics are K\"ahler and have finite volume.

Suppose that $f_0$ and $f_1$ are homotopic and let 
$$\hat F:[0,1]\times\CM(X)\to\CM(Y)$$
be a homotopy (as orbifold maps). The Weil-Petersson metric is negatively curved but not complete. However, it is 
geodesically convex. This allows to consider also the straight homotopy
$$F:[0,1]\times\CM(X)\to\CM(Y),\ \ F(t,x)=f_t(x)$$
determined by the fact that $t\mapsto f_t(x)$ is the geodesic segment joining $f_0(x)$ and $f_1(x)$ in the homotopy 
class of $\hat F([0,1]\times\{x\})$. Clearly, $f_t$ is smooth for all $t$.

Given a vector $V\in T_x\CM(X)$ the vector field $t\mapsto d(f_t)_xV$ is a Jacobi field along $t\mapsto f_t(x)$. Since 
the Weil-Peterson metric is negatively curved, the length of Jacobi fields is a convex function. Now, Lemma
\ref{lem1} gives that $f_0$ and $f_1$ are Lipschitz, and therefore the length of $d(f_t)_xV$ is bounded by the length of $V$ 
times a constant which depends neither on $t$ nor on $x$. It follows that the maps 
$$f_t:\CM(X)\to\CM(Y)$$
are uniformly Lipschitz. In particular, they have finite energy $E(f_t)<\infty$. In fact, the same convexity
 property of Jacobi fields shows that, for any $x$, the function $t\mapsto E_x(f_t)$ is convex. This implies that 
the energy function $t\mapsto E(f_t)$ is also convex. Moreover, it is strictly convex unless both holomorphic maps 
$f_0$ and $f_1$ either agree or are constant. Since the last possibility is ruled out by our assumptions we have:
\medskip

\noindent{\bf Fact.} Either $f_0=f_1$ or the function $t\mapsto E(f_t)$ is strictly convex.\qed
\medskip

Supposing that $f_0\neq f_1$ we may assume that $E(f_0)\ge E(f_1)$. By the fact above, for all $t\in(0,1)$:
we have
\begin{equation}\label{eq:convex}
E(f_t)<E(f_0)
\end{equation}
We are going to derive a contradiction to this assertion. In order to do so, let $K_n\subset\CM(X)$ be one of the 
sets provided by Lemma \ref{lem2} and denote by $\omega_X$ and $\omega_Y$ the K\"ahler forms of $\CM(X)$ and $\CM(Y)$ 
respectively. Since the K\"ahler forms are closed, we deduce from Stokes theorem that
\begin{align*}
0 
 = & \int_{[0,t]\times K_n}d\left((F^*\omega_Y)\omega_X^{\dim_\BC(\CT(X))-1}\right) \\
= & \int_{\D([0,t]\times K_n)}(F^*\omega_Y)\omega_X^{\dim_\BC(\CT(X))-1}\\
 = & \int_{\{t\}\times K_n}(F^*\omega_Y)\omega_X^{\dim_\BC(\CT(X))-1}-
\int_{\{0\}\times K_n}(F^*\omega_Y)\omega_X^{\dim_\BC(\CT(X))-1}\\
& + \int_{[0,t]\times \D K_n}(F^*\omega_Y)\omega_X^{\dim_\BC(\CT(X))-1}\\
 = & \int_{K_n}(f_t^*\omega_Y)\omega_X^{\dim_\BC(\CT(X))-1}-
\int_{K_n}(f_0^*\omega_Y)\omega_X^{\dim_\BC(\CT(X))-1} \\
& + \int_{[0,t]\times \D K_n}(F^*\omega_Y)\omega_X^{\dim_\BC(\CT(X))-1}
\end{align*}
Below we will prove:
\medskip

\noindent{\bf Claim.} $\lim_{n\to\infty}\int_{[0,t]\times \D K_n}(F^*\omega_Y)\omega_X^{\dim_\BC(\CT(X))-1}=0$.
\medskip

Assuming the claim we conclude with the proposition. From the claim and the computation above we obtain:
$$\lim_{n\to\infty}\left(\int_{K_n}(f_t^*\omega_Y)\omega_X^{\dim_\BC(\CT(X))-1}-
\int_{K_n}(f_0^*\omega_Y)\omega_X^{\dim_\BC(\CT(X))-1}\right)=0$$
Taking into account that both maps $f_t$ and $f_0$ are Lipschitz and that $\CM(X)$ has finite volume, we deduce that
$$\int_{\CM(X)}(f_t^*\omega_Y)\omega_X^{\dim_\BC(\CT(X))-1}=
\int_{\CM(X)}(f_0^*\omega_Y)\omega_X^{\dim_\BC(\CT(X))-1}$$
We obtain now from \eqref{eq:ES}
\begin{align*}
E(f_t) &\ge c\int_{\CM(X)}(f_t^*\omega_Y)\omega_X^{\dim_\BC(\CT(X))-1} \\
&=c\int_{\CM(X)}(f_0^*\omega_Y)\omega_X^{\dim_\BC(\CT(X))-1}=E(f_0)
\end{align*}
where the last equality holds because $f_0$ is holomorphic. This contradicts \eqref{eq:convex}.

It remains to prove the claim. 

\begin{proof}[Proof of the claim]
Let $d=\dim_\BR\CT(X)$ and fix $(t,x)\in[0,1]\times\D K_n$. Let $E_1,\dots,E_d$ be an orthonormal basis of 
$T_{(t,x)}([0,1]\times\D K_n)$. We have
\begin{align*}
\big\vert(F^*\omega_Y)&(E_1,E_2)\cdot\omega_X(E_3,E_4)\cdot\ldots\cdot\omega_X(E_{d-1},E_d)\big\vert\\
   &=\big\vert\langle dF_{(t,x)}E_1,i\cdot dF_{(t,x)}E_2\rangle_Y\langle E_3,iE_4\rangle_X\dots\langle E_{d-1},E_d\rangle_X\big\vert\\
   &=\Vert dF_{(t,x)}\Vert^2
\end{align*}
where $\langle\cdot,\cdot\rangle_X$ and $\langle\cdot,\cdot\rangle_Y$ are the Riemannian metrics on $\CM(X)$ and $\CM(Y)$ and where $\Vert dF_{(t,x)}\Vert$ is the operator norm of $dF_{(t,x)}$.

From this computation we deduce that there is a constant $c>0$ depending only on the dimension such that
$$\left\vert\int_{[0,t]\times \D K_n}(F^*\omega_Y)\omega_X^{\dim_\BC(\CT(X))-1}\right\vert\le c
\int_{[0,t]\times \D K_n}\Vert dF_{(t,x)}\Vert^2d\vol_{[0,t]\times \D K_n}(t,x)$$
Recall that $F:[0,1]\times\CM(X)\to\CM(Y)$ is the straight homotopy between the holomorphic (and hence 
Lipschitz) maps $f_0$ and $f_1$. Let $L$ be a Lipschitz constant for these maps and fix, once and for all, a point $x_0\in K_0\subset\CM(X)$. As we mentioned above, the restriction of $F$ to $\{t\}\times\CM(X)$ is $L$-Lipschitz for all $t$. Hence, the only direction that $dF_{(t,x)}$ can really increase is the $\frac\partial{\partial t}$ direction. Since $F$ is the straight homotopy, we have
\begin{align*}
\Vert dF_{(t,x)}\frac\partial{\partial t}\Vert_Y
   &=d_{\CM(Y)}(f_0(x),f_1(x)) \\
   &\le 2Ld_{\CM(X)}(x,x_0)+d_{\CM(Y)}(f_0(x_0),f_1(x_0))
\end{align*}
It follows that there are constants $A,B$ depending only on $L$ and the base point $x_0$ such that for all $(t,x)$ we have
$$\Vert dF_{(t,x)}\Vert\le A\cdot d_{\CM(X)}(x_0,x)^2+B$$
Summing up we have
\begin{multline}\label{final-eq}
\left\vert\int_{[0,t]\times \D K_n}(F^*\omega_Y)\omega_X^{\dim_\BC(\CT(X))-1}\right\vert\le\\
\le\left(A\cdot\max_{x\in\D K_n}d_{\CM(X)}(x,x_0)^2+B\right)\vol(\D K_n)
\end{multline}
By construction, $\max_{x\in\D K_n}d_{\CM(X)}(x,x_0)$ is bounded from above by a linear function of $n$. On the other hand, $\vol(\D K_n)$ decreases exponentially. This implies that the right side of \eqref{final-eq} tends to $0$ with $n\to\infty$. This proves the claim.
\end{proof}

Having proved the claim, we have concluded the proof of Proposition \ref{Eells-Sampson}.
\end{proof}

\bigskip

\noindent Department of Mathematics, National University of Ireland, Galway. \newline \noindent
\texttt{javier.aramayona@nuigalway.ie}

\bigskip

\noindent Department of Mathematics, University of Michigan, Ann Arbor. \newline \noindent
\texttt{jsouto@umich.edu}

\end{document}